%% file: volED.tex
\documentclass{article}

\usepackage{amsmath, amssymb, latexsym, amsthm, xspace, color}
\usepackage[all]{xy}
\usepackage[mathscr]{eucal}
\usepackage[dvips]{graphicx}

\title{\bf Billiards in L-shaped tables with barriers}
\author{Matt Bainbridge}

\numberwithin{equation}{section}

\makeatletter

\renewcommand{\l@section}{\@dottedtocline{0}{1.5em}{2.3em}}
\renewcommand{\l@subsection}{\@dottedtocline{1}{3.8em}{3.2em}}
\renewcommand{\l@subsubsection}{\@dottedtocline{2}{7.0em}{4.1em}}
\makeatother

\newcommand{\A}[1][D]{\mathcal{A}_{#1}}

\newcommand{\barmoduli}[1][2]{{\overline{\mathcal M}}_{#1}}
\newcommand{\barP}[1][D]{\overline{P}_{#1}}
\newcommand{\barW}[1][D]{\overline{W}_{#1}}

\newcommand{\barX}[1][D]{\overline{X}_{#1}}
\newcommand{\bdry}{\partial}
\newcommand{\cD}{\mathcal{D}}
\newcommand{\cE}[1][5]{\mathcal{E}_{#1}}
\newcommand{\comp}{\asymp}
\newcommand{\Cusps}[1][D]{C(\X[#1])}
\newcommand{\deriv}[2]{\frac{d }{d #2}}
\newcommand{\pderiv}[2]{\frac{\partial #1}{\partial #2}}

\newcommand{\domed}{\prec}
\newcommand{\E}[1][D]{E_{#1}}
\newcommand{\fU}{\mathfrak{U}}
\newcommand{\barF}[1][D]{\overline{\mathcal{F}}_{#1}}
\newcommand{\F}[1][D]{\mathcal{F}_{#1}}
\newcommand{\GLtwoR}{{\rm GL}_2\reals}
\newcommand{\GLtwoRplus}{{\rm GL}_2^+\reals}
\renewcommand{\hat}{\widehat}
\renewcommand{\tilde}{\widetilde}

\renewcommand{\Im}{\IM}
\renewcommand{\Re}{\RE}
\newcommand{\isom}{\cong}
\newcommand{\K}[1][D]{K_{#1}}
\newcommand{\kron}[2]{\left(\frac{#1}{#2}\right)}
\newcommand{\moduli}[1][2]{{\mathcal M}_{#1}} 
\newcommand{\omn}[1][g]{\Omega\moduli[#1]({\bf n})}

\newcommand{\ord}[1][D]{\mathcal{O}_{#1}}
\newcommand{\otn}[1][g]{\Omega\teich_{#1}({\bf n})}
\renewcommand{\P}[1][D]{P_{#1}}

\newcommand{\pom}[1][2]{\proj\Omega{\mathcal M}_{#1}}
\newcommand{\potm}[1][2]{\proj\Omega\tildemoduli[#1]}
\newcommand{\pobarm}[1][2]{\proj\Omega\overline{\mathcal M}_{#1}}

\newcommand{\PSL}{{\rm PSL}}
\newcommand{\PSLtwoR}{\PSL_2 \reals}

\newcommand{\sC}{\mathscr{C}}
\newcommand{\siegelmod}{\mathcal{A}_2}
\newcommand{\SL}{{\mathrm{SL}}}
\newcommand{\SLtwoord}[1][D]{\SL(\ord[#1]\oplus\ord[#1]^\vee)}
\newcommand{\SLtwoR}{\SL_2 \reals} 

\newcommand{\SOtwoR}{{\rm SO}_2 \reals} 
\newcommand{\teich}{\mathcal{T}}
\newcommand{\U}[1][D]{\Upsilon_{#1}^3}
\newcommand{\tildemoduli}[1][2]{{\widetilde{\mathcal M}}_{#1}}
\newcommand{\W}[1][D]{W_{#1}}

\newcommand{\X}[1][D]{X_{#1}} 
\newcommand{\Y}[1][D]{Y_{#1}}
\newcommand{\Yprot}[1][D]{\mathcal{Y}_{#1}}

\newcommand{\cx}{{\mathbb C}} 
\newcommand{\half}{{\mathbb H}} 
\newcommand{\integers}{{\mathbb Z}}
\newcommand{\nats}{{\mathbb N}}
\newcommand{\ratls}{{\mathbb Q}} 
\newcommand{\reals}{{\mathbb R}} 
\newcommand{\proj}{{\mathbb P}}
\newcommand{\zed}{\integers}

\newcommand{\Mobius}{M\"obius\xspace}
\newcommand{\Poincare}{Poincar\'e\xspace}

\newcommand{\Teichmuller}{Teichm\"uller\xspace}

\DeclareMathOperator{\Area}{Area}
\DeclareMathOperator{\area}{area}

\DeclareMathOperator{\Collapse}{Collapse}

\DeclareMathOperator{\End}{End}
\DeclareMathOperator{\Hol}{Hol}

\DeclareMathOperator{\IM}{Im}
\DeclareMathOperator{\RE}{Re}
\DeclareMathOperator{\height}{height}
\DeclareMathOperator{\interior}{int}
\DeclareMathOperator{\length}{length}
\DeclareMathOperator{\Link}{Link}

\DeclareMathOperator{\BOX}{Box}
\DeclareMathOperator{\Jac}{Jac}

\DeclareMathOperator{\Plumb}{Plumb}
\DeclareMathOperator{\Res}{Res}
\DeclareMathOperator{\Stab}{Stab}
\DeclareMathOperator{\connsum}{Sum}
\DeclareMathOperator{\zerosplit}{Split}
\DeclareMathOperator{\supp}{support}
\DeclareMathOperator{\Tr}{Tr}
\DeclareMathOperator{\vol}{vol}
\DeclareMathOperator{\Vol}{vol}

\newtheorem{theorem}{Theorem}[section] 
\newtheorem{prop}[theorem]{Proposition} 
\newtheorem{cor}[theorem]{Corollary}
\newtheorem{lemma}[theorem]{Lemma}

\theoremstyle{definition}
\newtheorem*{definition}{Definition}

\theoremstyle{remark}
\newtheorem*{remark}{Remark}

\begin{document}
\bibliographystyle{halpha}

\maketitle

\begin{abstract}
  We compute the volumes of the eigenform loci in the moduli space of genus two Abelian differentials.  From this, we obtain
  asymptotic formulas for counting closed billiards paths in certain L-shaped polygons with barriers.
\end{abstract}

\tableofcontents

\section{Introduction}
\label{sec:introduction}

Let $P=P(a, b, t)$ be the L-shaped polygon with barrier shown in Figure~\ref{fig:Lshaped}.  A \emph{billiards path} on $P$ is a path which
is geodesic on $\interior(P)$ and which bounces off of the sides of $P$ so that the angle of incidence equals the angle of reflection.  We
allow billiards paths to pass through one of the six corners of $P$ but not through the endpoint $p$ of the barrier.  A \emph{saddle
  connection} is a billiards path which joins $p$ to itself.  Each periodic billiards path is contained in a family of parallel periodic
billiards paths called a \emph{cylinder}, and each cylinder is bounded by a union of saddle connections.  A saddle connection is
\emph{proper} if it does not bound a cylinder.  We say that a saddle connection is \emph{multiplicity one} if it passes through a corner and
\emph{multiplicity two} otherwise.

We define the following counting functions:
{\allowdisplaybreaks
  \begin{align*}
    N_{\rm c}(P, L) =& \#\{\text{maximal cylinders of length at most $L$}\}, \\
    N_{\rm s}^1(P, L) =& \#\{\text{proper multiplicity one saddle connections of length at most $L$}\}, \\
    N_{\rm s}^2(P, L) =& \#\{\text{proper multiplicity two saddle connections of length at most $L$}\}. 
  \end{align*}
}
Note that we are counting unoriented saddle connections and cylinders, and we do not count periodic paths which repeat more than once.

In this paper, we will show:

\begin{theorem}
  \label{thm:countingsaddleconnections}
  Consider the polygon $P=P(x+z\sqrt{d}, y+z\sqrt{d}, t)$, with $t>0$, $x, y, z\in\ratls$, $x+y=1$,  $d\in \nats$ nonsquare, and
  \begin{equation}
    \label{eq:65}
    P \neq P\left(\frac{1+\sqrt{5}}{2}, \frac{1+\sqrt{5}}{2}, \frac{5-\sqrt{5}}{10}\right).
  \end{equation}
  We have the following asymptotics:
  \begin{align*}
    N_{\rm c}(P,L)&\sim \frac{15}{2\pi\Area(P)} L^2, \\
    N_{\rm s}^1(P,L)&\sim \frac{27}{16\Area(P)}\pi L^2,\quad\text{and} \\
    N_{\rm s}^2(P,L) &\sim \frac{5}{16\Area(P)}\pi L^2.
  \end{align*}
\end{theorem}

\begin{remark}
  The notation $a(L)\sim b(L)$ means $a(L)/b(L)\to 1$ as $L\to\infty$.
\end{remark}

\begin{figure}[htbp]
  \centering
  \input{Lshaped.pstex_t}
  \caption{$P(a, b, t)$}
  \label{fig:Lshaped}
\end{figure}

Using \cite{veech89}, one can show for the surface $P$ of \eqref{eq:65},
{\allowdisplaybreaks
  \begin{align*}
    N_{\rm c}(P,L)&\sim \frac{75}{16\pi\Area(P)} L^2, \\
    N_{\rm s}^1(P,L)&\sim \frac{1125 +425\sqrt{5}}{128\pi\Area(P)} L^2,\quad\text{and} \\
    N_{\rm s}^2(P,L) &\sim \frac{125-25\sqrt{5}}{32\pi\Area(P)} L^2.
  \end{align*}
}

\paragraph{Spaces of Abelian differentials.}

There is an unfolding construction which associates to a rational angled polygon $P$ a Riemann surface $X$ equipped with an Abelian
differential (holomorphic 1-form) $\omega$.  The \emph{unfolding} $(X, \omega)$ consists of several reflected copies of $P$ glued along their
boundaries.  One step in the proof of Theorem~\ref{thm:countingsaddleconnections} is the computation of the volumes of certain spaces of Abelian
differentials in which the unfoldings of the polygons $P(a, b, t)$ lie.  These volume computations are the focus of this paper.

More precisely, let $\Omega_1\moduli$ be the moduli space of genus two Abelian differentials $(X, \omega)$ such that $\int_X|\omega|^2=1$,
and let $\Omega_1\moduli(1,1)$ be the locus of differentials with two simple zeros.  Given a square-free $D\in\nats$ with $D\equiv 0$ or
$1\pmod 4$, let
\begin{equation*}
  \Omega_1\E\subset\Omega_1\moduli
\end{equation*}
be the locus of eigenforms for real multiplication of discriminant $D$ (we discuss real multiplication in \S\ref{sec:hilb-modul-surf}), and
let $\Omega_1\E(1,1) = \Omega_1\E\cap\Omega_1\moduli(1,1)$.  The locus $\Omega_1\E$ is a circle bundle over a Zariski-open subset of
the Hilbert modular surface of discriminant $D$,
\begin{equation}
  \label{eq:1}
  \X = (\half\times(-\half))/\SL_2\ord,
\end{equation}
where $\ord\subset \ratls(\sqrt{D})$ is the unique real quadratic order of discriminant $D$.  The unfolding of the polygon $P(a, b, t)$ is
genus two.  It has two simple zeros if and only if $t>0$, and it is an eigenform if and only if $a$ and $b$ are as in
Theorem~\ref{thm:countingsaddleconnections}.

There is a natural action of $\SLtwoR$ on $\Omega_1\moduli$ which preserves the spaces $\Omega_1\moduli(1,1)$ and $\Omega_1\E$.  There is
a natural absolutely continuous, finite, ergodic $\SLtwoR$-invariant measure $\mu_D^1$ on $\Omega_1\E$ which is constructed in
\cite{mcmullenabel}.   The main result of this paper is that the volume of $\Omega_1\E$ with respect to this measure is proportional to
the orbifold Euler characteristic of $\X$:
\begin{theorem}
  \label{thm:volED}
  The $\mu_D^1$-volume of $\Omega_1\E$ is
     $4 \pi \chi(\X)$.
\end{theorem}

\begin{remark}
  This theorem was conjectured by Maryam Mirzakhani.
  
  From a formula of Siegel for $\chi(\X)$, we obtain the explicit formula:
  \begin{equation*}
    \Vol\Omega_1\E[f^2 D] = 8\pi f^3 \zeta_{\ratls(\sqrt{D})}(-1) \sum_{r|f}\kron{D}{r}\frac{\mu(r)}{r^2},
  \end{equation*}
  for $D$ a fundamental discriminant.
\end{remark}

\paragraph{Counting functions.}

In order to prove Theorem~\ref{thm:countingsaddleconnections}, we need also to evaluate integrals over $\Omega_1\E(1,1)$ of certain counting
functions.  A Riemann surface with Abelian differential $(X, \omega)$ carries a flat metric $|\omega|$ on the complement of the zeros of
$\omega$.  A \emph{saddle connection} is a geodesic segment joining two zeros of $\omega$ whose interior is disjoint from the zeros.  If $X$
is genus two, it has a hyperelliptic involution $J$.  We say that a saddle connection $I$ joining distinct zeros has multiplicity one if
$J(I)=I$, and otherwise it has multiplicity two.    As for billiards paths, closed geodesics occur in parallel families forming metric
cylinders bounded by saddle connections.  We define for $T\in\Omega_1\moduli(1,1)$:
\begin{align*}
  N_{\rm c}(T, L) =& \#\{\text{maximal cylinders on $T$ of circumference at most $L$}\}, \\
  N_{\rm s}^1(T, L) =& \#\{\text{multiplicity one saddle connections on $T$ of length at most $L$}\\
  &\text{joining distinct zeros of $T$}\}, \quad\text{and}\\
  N_{\rm s}^2(T, L) =& \#\{\text{pairs of multiplicity two saddle connections on $T$ of length at}\\ 
  &\text{most $L$ joining distinct zeros of $T$}\}.
\end{align*}

We will show in \S\ref{sec:siegelveech}:
\begin{theorem}
  \label{thm:svintegral}
  For any $L>0$,
  \begin{align}
    \label{eq:2}
    \int\limits_{\Omega_1\E(1,1)} N_{\rm c}(T, L)\,  d\mu_D^1(T)&= 60 L^2 \chi(\X),\\
    \label{eq:3}
    \int\limits_{\Omega_1\E(1,1)} N_{\rm s}^1(T, L)\, d\mu_D^1(T)&= \frac{27}{2}\pi^2 L^2 \chi(\X),\quad\text{and}\\
    \label{eq:4}
    \int\limits_{\Omega_1\E(1,1)} N_{\rm s}^2(T, L) \,d\mu_D^1(T)&= \frac{5}{2}\pi^2 L^2 \chi(\X).
  \end{align}
\end{theorem}

\paragraph{Uniform distribution of circles.}

Given $T\in\Omega_1\E$, let $m_T$ be the measure on $\Omega_1\E$ obtained by transporting the Haar measure on $\SOtwoR$ to the
orbit $\SOtwoR\cdot T$.  Let
\begin{equation*}
  a_t =
  \begin{pmatrix}
    e^t & 0 \\
    0 & e^{-t}
  \end{pmatrix}.
\end{equation*}
The measure $(a_t)_*m_T$ is the uniform measure on a circle centered at $p$ in the $\SLtwoR$ orbit through $p$.

There is a \emph{\Teichmuller curve} $\Omega_1 D_{10}\subset\Omega_1\E[5](1,1)$ which is the $\SLtwoR$-orbit of the Abelian differential
obtained by identifying opposite sides of a regular decagon.  In \S\ref{sec:equid-large-circl}, we prove
\begin{theorem}
  \label{thm:uniformdistribution}
  For any point $T\in\Omega_1\E(1,1)$ which does not lie on the curve $\Omega_1 D_{10}$, we have
  \begin{equation}
     \label{eq:5}
    \lim_{t\to\infty} (a_t)_* m_T = \frac{\mu_D^1}{\Vol\mu_D^1}.
  \end{equation}
\end{theorem}

The proof of Theorem~\ref{thm:uniformdistribution} is a straightforward adaption of an argument in \cite{emm} together with a partial
classification of ergodic horocycle-invariant measures on $\Omega_1\E$ obtained in \cite{caltawortman}.

It follows from \cite{eskinmasur} that if \eqref{eq:5} holds for $T\in\Omega_1\E(1,1)$, then $$N(T, L)\sim c L^2,$$ where
\begin{equation*}
  c =  \frac{1}{L^2 \Vol(\mu_D^1)} \int\limits_{\Omega_1\E(1,1)} N_{\rm c}(T, L)\,  d\mu_D^1(T),
\end{equation*}
and similarly for the counting functions from saddle connections.

We then obtain from Theorems~\ref{thm:volED} and \ref{thm:svintegral}:

\begin{theorem}
  \label{thm:EDcountings}
  For every $T\in \Omega_1\E(1,1)$ which does not lie on $\Omega_1 D_{10}$, 
  \begin{align*}
    N_{\rm c}(T, L) &\sim \frac{15}{\pi} L^2, \\
    N_{\rm s}^1(T, L) &\sim\frac{27}{8}\pi L^2,\quad\text{and} \\
    N_{\rm s}^2(T, L) &\sim\frac{5}{8}\pi L^2.
  \end{align*}
\end{theorem}

Theorem~\ref{thm:countingsaddleconnections} follows directly from this by applying the unfolding construction.   

\paragraph{Outline of proof of Theorem~\ref{thm:volED}.}

The proof of Theorem~\ref{thm:volED} proceeds in the following steps:

\begin{enumerate}
\item Let $\mu_D$ be the projection of $\mu_D^1$ to $\X$.  It is enough to calculate $\Vol(\mu_D)$.  There is a foliation $\F$ of $\X$ by
  hyperbolic Riemann surfaces covered by $\SLtwoR$-orbits on $\Omega_1\E$.  The foliation $\F$ has a natural transverse measure which
  we discuss in \S\ref{sec:hilb-modul-surf}, and the product of this measure with the leafwise hyperbolic measure is $\mu_D$.
\item There is a compactification $\Y$ of $\X$ which we studied in \cite{bainbridge06}.  The space $\Y$ is a complex orbifold, and the
  boundary $\bdry\X = \Y\setminus\X$ is a union of rational curves.  These curves are parameterized by numerical invariants called
  \emph{prototypes} which we discuss in \S\ref{sec:prototypes}.  There is one rational curve $C_P\subset\bdry\X$ for each prototype $P$.
  We summarize the relevant properties of $\Y$ in \S\ref{sec:compactification-x}.  
\item The foliation $\F$ doesn't extend to a foliation of $\Y$; however, in \S\ref{sec:foliation-f-as}, we show that integration over the
  measured foliation $\F$ defines a closed current on $\Y$.  The form,
  \begin{equation*}
    \tilde{\omega}_1 = \frac{1}{2\pi}\frac{dx_1\wedge dy_1}{y_1^2},
  \end{equation*}
  on $\half\times\half$ descends to a closed 2-form $\omega_1$ on $\X$ which also defines a closed current on $\Y$.  In
  \S\ref{sec:foliation-f-as}, we also show that
  \begin{equation}
    \label{eq:6}
    \Vol(\mu_D) = [\barF]\cdot[\omega_1],
  \end{equation}
  where $[\barF]$ and $[\omega_1]$ are the classes in $H^2(\Y; \reals)$ defined by these currents, and the product is the intersection product
  on cohomology.
\item There are closed leaves $\W$ and $\P$ of $\F$ whose closures are smooth curves in $\Y$.  From \cite{bainbridge06}, we have
  \begin{equation*}
    [\omega_1] = -\frac{1}{3} [\barW] + \frac{3}{5}[\barP] - \frac{4}{15} B_D,
  \end{equation*}
  where $B_D$ is a certain linear combination of fundamental classes of the curves $C_P\subset\bdry\X$.  Thus by \eqref{eq:6}, we need only
  to compute the intersection numbers of these curves with $[\barF]$.
\item In \S\ref{sec:inters-with-clos}, we show that
  \begin{equation}
    \label{eq:7}
    [\barW]\cdot[\barF] = [\barP]\cdot[\barF] = 0.
  \end{equation}
  Heuristically, these intersection numbers are zero because these curves are leaves of $\barF$, and a closed, nonatomic leaf of a measured
  foliation has zero intersection number with that foliation.  Since $\barF$ is singular at the cusps of $\barW$ and $\barP$, this
  intersection number is \emph{a priori} not necessarily zero.  However, leaves close to $\barW$ or $\barP$ tend to diverge from these curves near
  the cusps, so these cusps should not contribute to the intersection number.
  
  Here is the idea of the rigorous proof of \eqref{eq:7}: In \S\ref{sec:canonical-foliations}, we construct for any Riemann surface $X$ a canonical
  measured foliation $\mathcal{C}_X$ of $TX$ by Riemann surfaces and show essentially that the intersection number of the zero section of $TX$ with
  $\mathcal{C}_X$ is zero.  In \S\ref{sec:inters-with-clos}, we show that there is a tubular neighborhood of $\barW$ or $\barP$ which is
  isomorphic to a neighborhood of the zero section of (a twist of) $T\barW$ or $T\barP$.  Furthermore, this isomorphism takes the foliation
  $\F$ to the canonical foliation defined in \S\ref{sec:canonical-foliations}.  From this, we obtain \eqref{eq:7}.
\item In \S\ref{sec:inters-with-bound}, we calculate the intersection numbers $[\barF]\cdot C_P$ and obtain Theorem~\ref{thm:volED}.  
\end{enumerate}  

\paragraph{Notation.}

Here we fix some notation that we will use throughout this paper.  Given a manifold $M$ equipped with a foliation $\mathcal{F}$ by
hyperbolic Riemann surfaces and a $n$-form $\omega$, we will denote by $\|\omega\|_\mathcal{F}$ the function on $M$ defined at $x$ by:
\begin{equation}
  \label{eq:normdef}
  \|\omega_x\|_\mathcal{F} = \sup_{v_1, \dots, v_n \in T_x \mathcal{F}} \frac{|\omega(v_1, \dots, v_n)|}{\|v_1\|\cdots\|v_n\|},
\end{equation}
where the norms of vectors is with respect the hyperbolic metric.  Similarly, if $h$ is a Riemannian metric, then
$\|\omega\|_h$ will denote the analogous norm with the supremum over all vectors tangent $M$.  

We use the notation $x \domed y$ to mean that $|x|<C|y|$ for an arbitrary positive constant $C$.  We use $x\comp y$ to mean that $c
<|x/y|<C$ for arbitrary positive constants $c$ and $C$, and we say that $x$ is \emph{comparable} to $y$.  We write $C(x)$ for an arbitrary
positive constant depending only on the data $x$.

We use the following notation for subsets of $\cx$: $\half$ will denote the upper-half plane, $\Delta_r$ the disk of radius $r$, and
$\Delta^*_r = \Delta_r\setminus\{0\}$ the punctured disk.

$\GLtwoRplus$ denotes the identity component of $\GLtwoR$.  We write $B$ for the subgroup of upper-triangular matrices, and $H$ for the
subgroup of $B$ with ones on the diagonal.  We write:
\begin{equation*}
  a_t =
  \begin{pmatrix}
    e^t & 0 \\
    0 & e^{-t} 
  \end{pmatrix}, \quad
  h_t =
  \begin{pmatrix}
    1 & t \\
    0 & 1
  \end{pmatrix},\quad\text{and}\quad
  r_\theta =
  \begin{pmatrix}
    \cos \theta & -\sin\theta \\
    \sin\theta & \cos\theta
  \end{pmatrix}.
\end{equation*}

\paragraph{Notes and references.}

McMullen classified the orbit closures for the action of $\SLtwoR$ on $\Omega_1\moduli$ in the series of papers,
\cite{mcmullentor,mcmullendecagon,mcmullenspin,mcmulleninfinite,mcmullenabel,mcmullenbild}.  See also \cite{calta} and
\cite{hubertlelievre}.  The orbit closures in this classification are the entire space, the stratum $\Omega_1\moduli(2)$, the eigenform loci
$\Omega_1\E$, the curves $\Omega_1\W$, the decagon curve $\Omega_1 D_{10}$, and infinitely many \Teichmuller curves generated by
square-tiled surfaces.  This paper, together with \cite{bainbridge06} almost finishes computing the volumes of all of these orbit closures
(see also \cite{lelievreroyer}).  It only remains to understand the \Teichmuller curves generated by square-tiled surfaces with two simple
zeros.  

There has been much recent work on computing asymptotics for closed geodesics and configurations of saddle connections on rational billiards
tables and flat surfaces.  It is conjectured for any surface $(X, \omega)$ of any genus,
\begin{equation*}
  N_{\rm c}((X, \omega), L)\sim C L^2
\end{equation*}
for some constant $C$ which depends on $(X, \omega)$.  This constant $C$ is called a \emph{Siegel-Veech constant} Veech \cite{veech89}
showed quadratic asymptotics for any \emph{lattice surface} $(X, \omega)$ (this means that the affine automorphism group of $(X, \omega)$
determines a lattice in $\SLtwoR$).  He also computed the Siegel-Veech constants for an infinite series of examples obtained by gluing
regular $n$-gons.  Eskin and Masur \cite{eskinmasur}, building on work of Veech \cite{veech98}, obtained quadratic asymptotics for almost
every Abelian differential $(X, \omega)$ with a given genus and orders of zeros, and \cite{emz} computed the Siegel-Veech constants.
Eskin, Masur, and Schmoll \cite{ems}  evaluated the Siegel-Veech constants for the polygons $P(a, b, t)$ of Theorem~\ref{thm:countingsaddleconnections}
with $a,b\in\ratls$ and $t\in\reals\setminus\ratls$ (this case can also be handled with the methods of this paper).  See also
\cite{gutkinjudge}, \cite{schmoll05}, \cite{emm}, and \cite{lelievre} for more about Siegel-Veech constants.

I would like to thank Curt McMullen, Maryam Mirzakhani, Alex Eskin, Howard Masur, and Florian Herzig for useful conversations over the
course of this work.  I am also grateful to IHES for their hospitality during the preparation of this paper.

\section{Canonical foliations}
\label{sec:canonical-foliations}

In this section, we discuss a canonical measured foliation $\mathcal{C}_X$ of the tangent bundle $TX$ of any Riemann surface $X$.  In the
case where $X$ has finite volume, we study smooth $2$-forms on a certain extension $L$ of the bundle  $TX$ over the compactification $\overline{X}$.  We
will show that for some compactly supported 2-form $\Psi$ on $L$ which represents the Thom class,
  $\int_{\mathcal{C}_X} \Psi =0$.
These foliated bundles will serve as models for tubular neighborhoods of closed leaves of $\F$ and will be used in the proof
of \eqref{eq:7} in \S\ref{sec:inters-with-clos}.

\paragraph{Action of $\PSLtwoR$ on $\cx$.}

Let $Q$ be the complex plane $\cx$ equipped with the quadratic differential
\begin{equation*}
  q = \frac{1}{2z}dz^2.
\end{equation*}
Define an action of $\PSLtwoR$ on $Q$ by $q$-affine automorphisms as follows.  Let $\SLtwoR$ act on $\cx$ by
\begin{equation*}
  \begin{pmatrix}
    a & b \\ c & d
  \end{pmatrix}
  \cdot (x + i y) = (ax -by) + i(-cx + dy).
\end{equation*}
Given $z\in Q$, define $A\circ z$ so that the following diagram commutes,
\begin{equation*}
  \xymatrix{ \cx \ar[r]^{A\cdot}\ar[d]_f & \cx\ar[d]^f \\
    Q \ar[r]^{A\circ} & Q }
\end{equation*}
where $f(z) = z^2$.

\paragraph{Canonical foliations.}

Let $\PSLtwoR$ act on $\half$ on the left by \Mobius transformations in the usual way:
\begin{equation*}
  \begin{pmatrix}
    a & b\\ c & d
  \end{pmatrix}
  \cdot z = \frac{az+b}{c z+d},
\end{equation*}
and give $\half\times Q$ the product action of $\PSLtwoR$.  Define $\Phi\colon \half\times Q\to T\half$ by
\begin{equation*}
  \Phi(z, (u+iv)^2) = \left(z, (u + zv)^2 \pderiv{}{z}\right).
\end{equation*}
The following lemma is a straightforward calculation:

\begin{lemma}
  \label{lem:actioncommutes}
  $\Phi$ commutes with the $\PSLtwoR$ actions on $\half\times Q$ and $T\half$.
\end{lemma}

The map $\Phi$ sends the horizontal foliation of $\half\times Q$ with the transverse measure induced by the quadratic differential $q$ to a
measured foliation of $T\half$.  Call this measured foliation $\mathcal{C}_\half$.

Let $X = \Gamma\backslash\half$ be a hyperbolic Riemann surface, or more generally a hyperbolic Riemann surface orbifold.  By
Lemma~\ref{lem:actioncommutes}, the foliation $\mathcal{C}_\half$ of $T\half$ descends to a measured foliation of $TX$.  Call this measured
foliation $\mathcal{C}_X$.

\paragraph{Global angular forms.}

Let $X$ be a compact $2$-dimensional real orbifold equipped with a line bundle $\pi\colon L\to X$ with a Hermitian metric.  Given an
open set $U\subset X$ with a trivialization $\pi^{-1}(U)\to U\times \cx$, let $d\theta_U$ be the $1$-form induced by the standard angular
form on $\cx$,
\begin{equation*}
  d\theta = \frac{1}{2\pi}\frac{x\,dy - y\,dx}{x^2 + y^2}.
\end{equation*}
Choose a compactly supported smooth bump function $\rho\colon [0, \infty)\to \reals$ with $\rho(r)\equiv 1$ near $0$.  Using the
Hermitian metric, we regard $\rho(r)$ as a smooth function on $L$.

The following theorem follows from the discussion on pp.70-74 of \cite{botttu}:

\begin{theorem}
  \label{thm:globalangularform}
  There is a smooth $1$-form $\psi$ on the complement of the zero section of $L$ with the following properties:
  \begin{itemize}
  \item For any contractable domain $U\subset X$,  we have on $\pi^{-1}(U)$,
    \begin{equation*}
      \psi = d\theta_U + \pi^*\eta,
    \end{equation*}
    for some smooth $1$-form $\eta$ on $U$.
  \item $d\psi = -\pi^* e,$ with $e$ a smooth $2$-form representing the Euler class of $L$.
  \item $\Psi = d(\rho(r)\psi)$ is a compactly supported smooth $2$-form representing the Thom class of $L$.
  \end{itemize}
\end{theorem}

Such a form $\psi$ is called a \emph{global angular form} for $L$.

\paragraph{Intersection of zero section with $\mathcal{C}_X$.}

Now, consider the following situation.  Let $X=\Gamma\backslash \half$ be a finite volume hyperbolic orbifold and $\overline{X}$ its
compactification.  Let
\begin{equation*}
  L = T\overline{X}\left(-\sum c_i\right),
\end{equation*}
where the $c_i$ are the cusps of $X$.  That is, $L$ is the line bundle whose nonzero holomorphic sections are holomorphic sections of
$T\overline{X}$ which have simple zeros at the cusps $c_i$.  Equip $L$ with a Hermitian metric $h$.  Let $\psi$ be a global angular form on
$L$ and let $\Psi = d(\rho(r)\psi)$.

\begin{theorem}
  \label{thm:thomintegralzero}
  With $X$ and $\Psi$ as above,
  \begin{equation*}
    \int_{\mathcal{C}_X} \Psi = 0.
  \end{equation*}
\end{theorem}

The proof will follow a series of lemmas.

\begin{lemma}
  \label{lem:normrho}
  For any smooth, compactly supported form $\eta$ on $L$,
    $\|\eta\|_{\mathcal{C}_X} \domed 1$.
\end{lemma}

\begin{proof}
  Since $\eta$ is a smooth form on $L$, it suffices to bound $\|\eta\|_{\mathcal{C}_X}$ over a neighborhood of each cusp.  Suppose $c$ is a cusp of
  $X=\Gamma\backslash\half$, and conjugate $\Gamma$ so that $c=i\infty$ and $\Stab_\infty \Gamma$ is generated by $z\mapsto z+2\pi$.

  Identify $T\half$ and $\half\times\cx$  by
  \begin{equation*}
    (z, t)\mapsto \left(z, t\pderiv{}{z}\right).
  \end{equation*}
  Identify a neighborhood of the cusp $c$ in $\overline{X}$ and $\Delta_\epsilon\subset\cx$ with coordinate $w$ by
  $w=e^{iz}$.
  Identify $T\Delta_\epsilon(-0)$ with $\Delta_\epsilon\times\cx$
  by
  \begin{equation*}
    (w, s)\mapsto\left(w, sw\pderiv{}{w}\right).
  \end{equation*}
  The $(z, t)$ and $(w, s)$-coordinates are related by $s=-it$.  We give $T\Delta_\epsilon(-0)$ the Euclidean metric $e$
  inherited from $\cx^2$.

  At $(z, t\pderiv{}{z})\in T\half$, suppose $(u+ zv)^2 = t$ with $u, v\in\reals$, which implies
  \begin{equation*}
    v=\frac{\Im\sqrt{t}}{\Im z}.
  \end{equation*}
  The vector,
  \begin{align}
    \notag
    V &=\Im z\left(\pderiv{}{z} + 2v(u+zv)\pderiv{}{t}\right) \\
    \label{eq:9}
    &= \Im z \pderiv{}{z} + i(|t|-t)\pderiv{}{t},
  \end{align}
  is a unit tangent vector to $\mathcal{C}_X$ at $(z, t\pderiv{}{z})$.  In $(w, s)$-coordinates on $T\Delta_\epsilon(-0)$,
  \begin{equation*}
    V= -i w \log |w| \pderiv{}{w} +(|s|-is)\pderiv{}{s}
  \end{equation*}
  is a unit tangent vector to $\mathcal{C}_X$ at $(w, s)$.  We have $\|V\|_e \domed s$, so for any smooth $p$-form $\eta$ on $L$,
  $\|\eta\|_{\mathcal{C}_X}\domed s^p$ near $c$.  Thus $\|\eta\|_{\mathcal{C}_X}\domed 1$ when $\eta$ is compactly supported.
\end{proof}

\begin{lemma}
  \label{lem:boundpsi}
  $\|\psi\|_{\mathcal{C}_X}$ is bounded on compact subsets of $L$.
\end{lemma}

\begin{proof}
  Let $\tilde{\alpha}$ be the $1$-form on $T\half$ defined by
  \begin{equation*}
    \tilde{\alpha} = \frac{dx}{y} + d\theta.
  \end{equation*}
  The form $\tilde{\alpha}$ is $\PSLtwoR$-invariant, so it descends to a form $\alpha$ on $TX$.  This invariance can be seen from a
  direct computation or more conceptually as follows.  Let $T^1\half$ be the unit tangent bundle, and let $T^0\half$ be the complement in $TX$ of
  the zero section with projection $p\colon T^0\half\to T^1\half$. The connection form $\omega$ (in the sense of \cite{kobayashinomizu}) of
  the Levi-Civita connection is an isometry-invariant form on $T^1\half$, and $\tilde{\alpha} = p^*\omega$, so $\tilde{\alpha}$ is also
  invariant.
  
  In the $(z, t)$ coordinates on $T\half$ from the proof of Lemma~\ref{lem:normrho}, for $r\geq 0$, let $v_r$ be the vector $\pderiv{}{z}$
  at $(i, r)$.  For $\theta\in\reals$, the vector $e^{i\theta}v_r$ is tangent at $(i, r)$ to $\mathcal{C}_\half$, with
  $\tilde{\alpha}(e^{i\theta}v_r)\leq 1$ and equality if $\theta\equiv 0\pmod{2\pi}$.  Furthermore, every unit norm tangent vector to
  $\mathcal{C}_\half$ is $\PSLtwoR$-equivalent to some $e^{i\theta}v_r$.  Thus we have,
  \begin{equation}
    \label{eq:11}
    \|\alpha\|_{\mathcal{C}_X} = 1.
  \end{equation}

  The form $\psi-\alpha$ is smooth on $TX$ by Theorem~\ref{thm:globalangularform}, so 
  $\|\psi-\alpha\|_{\mathcal{C}_X}$ is bounded 
  on compact sets disjoint from the fibers over the cusps, and thus $\|\psi\|_{\mathcal{C}_X}$ is bounded on such sets by \eqref{eq:11}.
  
  We must now bound the norm of $\psi$ around the cusps.  As in the proof of Lemma~\ref{lem:normrho}, suppose the cusp is at $i\infty$ with
  $\Stab_{i\infty}\Gamma$ generated by $z\mapsto z+2\pi$, and identify the bundle $L$ restricted to a neighborhood of the cusp with
  $T\Delta_\epsilon(-0)$.  The form $d\theta$ on $T\half$ is invariant under $z\mapsto z+2\pi$, so it descends to a form $\beta$ on
  $T\Delta_\epsilon(-0)$.  From the coordinates on $T\Delta_\epsilon(-0)$ in the proof of Lemma~\ref{lem:normrho}, and the first part of
  Theorem~\ref{thm:globalangularform}, we see that $\psi-\beta$ is smooth, so has bounded norm on compact sets by Lemma~\ref{lem:normrho}.
  By pairing the tangent vectors $v$ to $\mathcal{C}_\half$ in \eqref{eq:9} with $\beta$, we see that $\|\beta\|_{\mathcal{C}_X}\leq 2$,
  thus the norm of $\psi$ is bounded on compact subsets of $T\Delta_\epsilon(-0)$.
\end{proof}

\begin{lemma}
  \label{lem:boundoverhorocycle}
  For any horocycle $H$ around a cusp $c$ of $X$, the inverse image $\pi^{-1}(H)\subset L$ is transverse to $\mathcal{C}_X$, and
  \begin{equation}
    \label{eq:12}
    \int_{\mathcal{H}}|\rho(r)\psi| \domed \length(H),
  \end{equation}
  where the integral is over the measured foliation $\mathcal{H}$ of $\pi^{-1}(H)$ induced by the intersection with $\mathcal{C}_X$.
\end{lemma}

\begin{proof}
  The transversality statement follows from the fact that every leaf of $\mathcal{C}_\half$ is the graph of a holomorphic section $\half\to
  T\half$.

  From Lemma~\ref{lem:boundpsi}, we have
  \begin{equation*}
    \int_{\mathcal{H}}|\rho(r)\psi| < C(\rho, \psi) \vol_{\pi^{-1}(H)} (\supp\rho(r)),
  \end{equation*}
  where the volume form on $\pi^{-1}(H)$ is obtained by taking the product of the transverse measure to the foliation with the hyperbolic
  arclength measure on the leaves.  We claim that
  \begin{equation}
    \label{eq:13}
    \vol_{\pi^{-1}(H)} \supp\rho(r) < C(\rho) \length(H),
  \end{equation}
  which would imply \eqref{eq:12}.

  Let $g$ be the hyperbolic metric on $TX$.  The hyperbolic norm of a nonzero holomorphic section over a neighborhood of a cusp is
  comparable to the hyperbolic norm of the vector field $z\pderiv{}{z}$ on $\Delta^*$,
  \begin{equation*}
    \left\|z\pderiv{}{z}\right\| = \left(\log\frac{1}{|z|}\right)^{-1},
  \end{equation*}
  which is bounded.  Thus $g/h$ is bounded on $X$, where $h$ is the Hermitian metric on $L$.
  
  For each $x\in X$, the fiber $\pi^{-1}(x)=T_x X$ has the structure of the quadratic differential, $(\cx, dz^2/2z)$, with the unit
  hyperbolic norm vectors the unit circle in $\cx$.  The intersection $\pi^{-1}(x)\cap\supp\rho(r)$ is a disk in of radius
  $R(x)$ in $\cx$ with $R(x) = Kg/h$, where
  \begin{equation*}
    K=\sup\{r : \rho(r)\neq 0\}.
  \end{equation*}
  Since $g/h$ is bounded, so is $R$.  The volume of each disk $\pi^{-1}(x)\cap\supp\rho(r)$ is less than $\pi R/2$, which is then bounded,
  and \eqref{eq:13} follows.
\end{proof}

Let $T^s X$ be the set of vectors of hyperbolic length $s$.

\begin{lemma}
  \label{lem:boundoverTr}
  The intersection of $T^s X$ with $\mathcal{C}_X$ is transverse, and we have
  \begin{equation*}
    \int_{\mathcal{T}(s)} |\rho(r)\psi| \domed s,
  \end{equation*}
  where $\mathcal{T}(s)$ is the induced measured foliation of $T^sX$.
\end{lemma}

\begin{proof}
  At $v_s=(i, s \pderiv{}{z})\in T^s\half$ for $s>0$, the leaf of $\mathcal{C}_\half$ which is the graph of $z\mapsto (z, s\pderiv{}{z})$
  is transverse to $T^s\half$.  Then $T^s\half$ is transverse to $\mathcal{C}_\half$ because $\SLtwoR$ acts transitively on
  $T^s\half$.

  Since $\|\rho(r)\psi\|_{\mathcal{C}_X}$ is bounded by Lemma~\ref{lem:boundpsi}, it is enough to show
  \begin{equation}
    \label{eq:14}
    \Vol T^s X\domed s,
  \end{equation}
  where the volume is with respect to the product of the transverse measure to $\mathcal{T}(s)$ with the hyperbolic length measure on the
  leaves.  Let $h$ be a segment of a leaf of $\mathcal{T}(s)$.  If we identify the leaf $L$ of $\mathcal{C}_X$ containing $h$ with
  $\half$, then $h$ is a segment of a horocycle in $L$.  Let $D_h\subset L$ be the domain bounded by $h$ and the two asymptotic geodesic
  rays perpendicular to $h$ at its endpoints.  A simple calculation shows $\area(D_h) = \length(h)$, so
  \begin{equation*}
    \Vol T^s X = \Vol B_s = \frac{\pi s}{2} \Vol X, 
  \end{equation*}
  where $B_s\subset TX$ is the set of vectors of length at most $s$.  This proves \eqref{eq:14}. 
\end{proof}  

\paragraph{Proof of Theorem~\ref{thm:thomintegralzero}.}

Let $H\subset X$ be a union of embedded horocycles bounding a neighborhood $C$ of the cusps of $X$.

By Stokes' Theorem,
\begin{equation*}
  \int_{\mathcal{C}_X}\Psi = \int\limits_{\pi^{-1}C \cup B_s} \Psi + \int\limits_{\bdry(\pi^{-1}C \cup B_s)} \rho(r) \psi.
\end{equation*}

We have
\begin{align*}
  \bigg|\int\limits_{\pi^{-1}C \cup B_s} \Psi \bigg| &< \|\Psi\|_{\mathcal{C}_X}^\infty \vol((\pi^{-1}C\cap\supp\rho(r)) \cup B_s) \\
  & <C(\rho,\psi)(\vol C + \vol B_s)
\end{align*}
by Lemma~\ref{lem:normrho}.  This can be made arbitrarily small by choosing $C$ and $s$ small.

We also have
\begin{align*}
  \bigg|\int\limits_{\bdry(\pi^{-1}C \cup B_s)} \rho(r) \psi\bigg| &\leq \int_{\pi^{-1}H}|\rho(r)\psi| + \int_{T^s X}|\rho(r)\psi| \\
  \leq C(\rho,\psi)(\length(H) + s),
\end{align*}
by Lemmas~\ref{lem:boundoverhorocycle} and \ref{lem:boundoverTr}.  This can also be made arbitrarily small.  Thus $\int_{\mathcal{C}_X}\Psi
=0$ as claimed.
\qed

\section{Abelian differentials}
\label{sec:abel-diff}

We record in this section some standard background material on Riemann surface equipped with Abelian differentials
and their moduli spaces.   We discuss the flat geometry defined by an Abelian differential and some surgery operations on Abelian
differentials.  We then discuss the action of $\GLtwoRplus$ on the moduli space $\Omega\moduli[g]$ of genus $g$ Abelian
differentials.  

\paragraph{Stable Riemann surfaces.}

A \emph{stable Riemann surface} is a compact, complex, one-dimensional complex analytic space with only nodes as singularities such that
each connected component of the complement of the nodes has negative Euler characteristic.  One can also regard a stable Riemann surface as
a finite union of finite volume Riemann surfaces together with an identification of the cusps into pairs. The \emph{(arithmetic) genus} of a stable
Riemann surface $X$ is the genus of the nonsingular surface obtained by thickening each node to an annulus.

\paragraph{Stable Abelian differentials.}

Let $X$ be a stable Riemann surface and $X_0$ the complement of the nodes.
A \emph{stable Abelian differential} on $X$ is a
holomorphic 1-form $\omega$ on $X_0$ such that:

\begin{itemize}
\item The restriction of $\omega$ to each component of $X_0$ has at worst simple poles at the cusps.
\item At two cusps $p$ and $q$ which have been identified to form a node,
  \begin{equation}
    \label{eq:15}
    \Res_p\omega = -\Res_q\omega.
  \end{equation}
\end{itemize}

It follows from the Riemann-Roch Theorem that the space $\Omega(X)$ of stable Abelian differentials on $X$ is a
$g$-dimensional vector space, where $g$ is the genus.

\paragraph{Translation structure.}

Given a Riemann surface $X$, a nonzero Abelian differential $\omega$ on $X$ determines a metric $|\omega|$ on $X\setminus Z(\omega)$, where
$Z(\omega)$ is the set of zeros of $\omega$.  If $p\in U\subset X$ with $U$ a simply connected domain, we have a chart $\phi_{p, U}\colon
U\to\cx$ defined by
\begin{equation*}
  \phi_{p, U}(z) = \int_p^z\omega
\end{equation*}
satisfying $\phi^*dz = \omega$.  These charts differ by translations, so the metric $|\omega|$ has trivial holonomy.  In this metric, a zero
of $\omega$ of order $p$ is a cone point with cone angle $2\pi(p+1)$.

A \emph{direction} on an Abelian differential $(X, \omega)$ is a parallel line field on $X\setminus Z(\omega)$.  An Abelian differential
$(X, \omega)$ has a canonical \emph{horizontal direction}, the kernel of $\Im \omega$.

A surface with a flat cone metric with trivial holonomy and a choice of horizontal direction is often called a \emph{translation surface}.
We have seen that an Abelian differential $(X, \omega)$ naturally has the structure of a translation surface.

\paragraph{Projective Abelian differentials.}

A \emph{projective Abelian differential} on a Riemann surface $X$ is an element of the projectivization $\proj\Omega(X)$.  We 
will denote the projective class of a form $\omega$ by $[\omega]$.  One should think of a projective Abelian differential as a translation
surface without a choice of horizontal direction or scale for the metric.

\paragraph{Homological directions.}

A direction $v$ on $(X, [\omega])$ is a \emph{homological direction} if when we choose a representative $\omega$ of $[\omega]$ so that $v$
is horizontal, there is a homology class $\gamma\in H_1(X;\zed)$ so that $\omega(\gamma)\in\reals\setminus\{0\}$.

\paragraph{Directed Abelian differentials.}

A \emph{directed Abelian differential} is a triple $(X, [\omega], v)$, where $(X, [\omega])$ is a
projective Abelian differential, and $v$ is a direction on $(X, [\omega])$.  Two such objects $(X_i, [\omega_i], v_i)$ are equivalent if
there is an isomorphism $f\colon X_1\to X_2$ such that $f^*[\omega_2]=[\omega_1]$ and $f^* v_2=v_1$.  We can regard an Abelian differential
$(X, \omega)$ as directed by taking the horizontal direction.

\paragraph{Cylinders.}

Every closed geodesic on an Abelian differential $(X, \omega)$ is contained in a family of parallel closed geodesics isometric to a flat
metric cylinder $S^1\times(a, b)$.  The maximal such cylinder is bounded by a union of saddle connections.

A (directed) Abelian differential $(X, \omega)$ is \emph{periodic} if every leaf of the horizontal foliation which does not meet a zero of
$\omega$ is closed.  In this case, every leaf which meets a zero is a saddle connection.  The union of all such leaves is called the
\emph{spine} of $(X, \omega)$.  The complement of the spine is a union of maximal cylinders.  We call $(X, \omega)$ a $n$-cylinder surface
if the complement of the spine has $n$ cylinders.

If $(X, \omega)$ is a stable Abelian differential, and the residue of $\omega$ at a node $n$ is real and nonzero, then $(X, \omega)$ has two
horizontal, half-infinite cylinders facing the node $n$.  
When we speak of an $n$-cylinder surface, we count both of these half-infinite cylinders as one cylinder.

\paragraph{Connected sum.}

Consider two Abelian differentials $(X_i, \omega_i)$.  Let $I\subset\cx$ be a geodesic segment with two embeddings $\iota_i\colon I\to
X_i\setminus Z(\omega_i)$ which are local translations.  The \emph{connected sum} of $(X_1, \omega_1)$ and $(X_2, \omega_2)$ along $I$ is
the Abelian differential obtained by cutting each $(X_i, \omega_i)$ along $\iota_i$ and regluing the two slits by local translations to
obtain a connected Abelian differential.

Now suppose $(X, \omega)$ is a genus two stable Abelian differential with one separating node.  The differential $(X, \omega)$ is the union
of two genus one differentials $(E_1, \omega_1)$ and $(E_2, \omega_2)$ joined at a single point.  Let $I\subset\cx$ be a segment which
embeds in each $(E_i, \omega_i)$ by local translations. We write
\begin{equation*}
  \connsum((X, \omega), I),
\end{equation*}
for the connected sum of the $(E_i, \omega_i)$ along $I$.  Note that the choice of embedding $\iota_i\colon I\to E_i$ is irrelevant because
the automorphism group of $(E_i, \omega_i)$ is transitive.  This connected sum is a genus two Abelian differential with two simple zeros.

\paragraph{Splitting a double zero.}

Consider an Abelian differential $(X, \omega)$ with a double zero $p$.  There is a surgery operation which we call \emph{splitting}
which replaces $(X, \omega)$ with a new differential where $p$ becomes two simple zeros.  Consider a geodesic segment $I=\overline{0w}\subset\cx$.  An
\emph{embedded ``X''} is a collection of four embeddings $\iota_1, \ldots, \iota_4$ of $I$ on $(X, \omega)$ with the following properties:
\begin{itemize}
\item Each $\iota_i$ is a local translations.
\item Each $\iota_i$ sends $0$ and no other point of $I$ to $p$.
\item The segments $\iota_i(I)$ are disjoint except at $p$.
\item The segments $I_i=\iota_i(I)$ meet at $p$ with angles
  \begin{equation*}
    \angle I_1I_2 = \pi, \quad \angle I_2I_3=2\pi, \quad\text{and}\quad \angle I_3I_4=\pi.
  \end{equation*}
\end{itemize}
Given $E$ an embedded ``X'', we define the splitting
\begin{equation*}
  \zerosplit((X, \omega), E)
\end{equation*}
to be the form obtained by cutting along $E$ and regluing as indicated in Figure~\ref{fig:splitting}.
\begin{figure}[htb]
  \centering
  \includegraphics{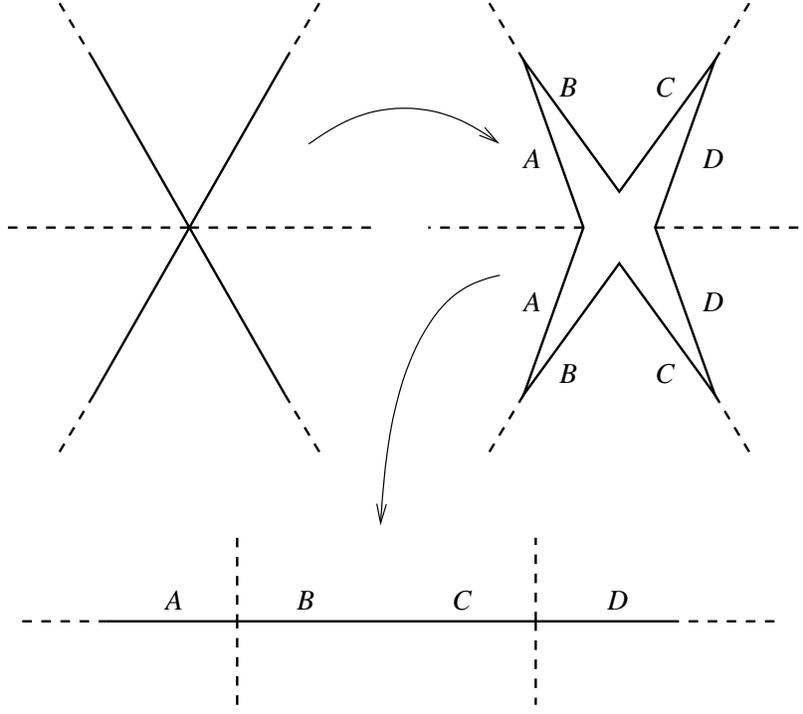}
  \caption{Splitting a double zero}
  \label{fig:splitting}
\end{figure}

\paragraph{Collapsing a saddle connection.}

The operations of forming a connected sum and splitting a double zero have inverses which we call collapsing a saddle connection.  Let $(X,
\omega)$ be a genus two differential with two simple zeros and with $J\colon X\to X$ its hyperelliptic involution.  Given a saddle
connection $I$ joining distinct zeros such that $J(I)\neq I$, the union $J(I)\cup I$ is homologous to zero and so separates $X$ into two
genus one components.  Let
\begin{equation*}
  \Collapse((X, \omega), I)
\end{equation*}
be the stable differential with one separating node obtained by cutting $(X, \omega)$ along $I$, regluing the boundary components to obtain
two genus one forms, and then identifying these two genus one forms at a single point.

Now suppose $I$ is a saddle connection of length $\ell$ joining zeros $p_1$ and $p_2$ such that $J(I)=I$.  We say that $I$ is
\emph{unobstructed} if there are embedded geodesic segments $I_1, I_2\subset X$ of length $\ell/2$ with the following properties:
\begin{itemize}
\item $I_j$ has one endpoint on $p_j$ and is otherwise disjoint from the zeros and $I$
\item $I_j$ and $I$ meet at an angle of $2\pi$ at $p_j$.
\item $I_1$ and $I_2$ are disjoint.
\end{itemize}
If $I$ is unobstructed, then we can cut along $I_1\cup I_2\cup I$ and then reglue following Figure~\ref{fig:splitting} in reverse to obtain
a new Abelian differential with two simple zeros.  We also denote this differential by $\Collapse((X, \omega), I)$.

Note that a saddle connection is always unobstructed if it is the unique shortest saddle connection up to the action of $J$, so any such
saddle connection can be collapsed.

\paragraph{Plumbing a cylinder.}

Let $(X, \omega)$ be a stable Abelian differential with a node $n$ where $\omega$ has a simple pole.  There are two half-infinite cylinders
$C_i$ facing $n$; let $\gamma_i$ be a closed geodesic in $C_i$.  The condition \eqref{eq:15} ensures that the $\gamma_i$ have the same
holonomy with respect to the translation structure.  Thus we can cut $(X, \omega)$ along the $\gamma_i$ and reglue to obtain a new
differential $(X', \omega')$ where the node $n$ has been replaced with a finite cylinder.  We call this operation \emph{plumbing a
  cylinder}.  This operation depends on two real parameters: the height of the resulting cylinder and a twist parameter which determines the
gluing of the $\gamma_i$.

\paragraph{\Teichmuller space.}

Let $\teich_g$ be the \Teichmuller space of closed genus $g$ Riemann surfaces, and let $\Omega\teich_g$ be the trivial complex vector
bundle whose fiber over a Riemann surface $X$ is the space of nonzero Abelian differentials on $X$.

There is a natural stratification of $\Omega\teich_g$.  Given a partition ${\bf n} = (n_1, \ldots, n_r)$ of $2g-2$, let 
\begin{equation*}
  \otn[g]\subset\Omega\teich_g
\end{equation*}
be the locus of forms having zeros of orders given by the $n_i$.  By \cite{veech90}, the strata $\otn[g]$ are complex submanifolds of
$\Omega\teich_g$.

\paragraph{Period coordinates.}

There are simple holomorphic coordinates on the strata $\otn$ defined in terms of periods.  A form $(X, \omega)\in\otn$ defines a cohomology
class in $H^1(X, Z(\omega); \cx)$.  Over a contractable neighborhood $U\subset\otn$ of $(X, \omega)$, the bundle of surfaces with marked
points whose fiber over $(Y, \eta)$ is the surface $Y$ marked by the points of $Z(\eta)$ marked is topologically trivial.  Thus we can
canonically identify the groups $H^1(Y, Z(\eta); \cx)$ as $(Y, \eta)$ varies over $U$ (this is called the Gauss-Manin connection).  This
defines a map
\begin{equation*}
  \phi\colon U\to H^1(X, Z(\omega); \cx) \isom \cx^n.
\end{equation*}

\begin{theorem}[{\cite{veech90}}]
  \label{thm:periodcharts}
  The maps $\phi_U$ are biholomorphic coordinate charts.
\end{theorem}

\paragraph{Moduli space.}

Let $\moduli[g] = \teich_g/\Gamma_g$ be the moduli space of genus $g$ Riemann surfaces, where $\Gamma_g$ is the mapping class group.  The
moduli space of genus $g$ Abelian differentials is the quotient
$$\Omega\moduli[g] = \Omega\teich_g/\Gamma_g.$$
It is a holomorphic rank $g$ orbifold vector bundle over $\moduli[g]$ and is stratified by the suborbifolds,
\begin{equation*}
  \omn = \otn/\Gamma_g.
\end{equation*}

Let $\Omega_1\moduli[g]$ be the locus of forms $(X, \omega)$ of unit norm $\|\omega\|=1$, where
\begin{equation*}
  \|\omega\| = \int_X |\omega|^2.
\end{equation*}
Similarly, if $\Omega X$ denotes any space of Abelian differentials, $\Omega_1 X$ will denote the unit area differentials.

\paragraph{Deligne-Mumford compactification.}

Let $\barmoduli[g]$ denote the moduli space of genus $g$ stable Riemann surfaces, the Deligne-Mumford compactification of $\moduli[g]$.  The
bundle $\Omega\moduli[g]$ over $\moduli[g]$ extends to a bundle $\Omega\barmoduli[g]$ over $\barmoduli[g]$, the moduli space of stable
Abelian differentials.

\paragraph{Action of $\GLtwoRplus$.}

There is a natural action of $\GLtwoRplus$, on $\Omega\teich_g$.  Let $(X, \omega)$ be an Abelian differential with an atlas of charts
$\{\phi_\alpha\colon U\to \cx\}$ covering $X\setminus Z$ such that $\phi_\alpha^* dz = \omega$.  Given $A\in \GLtwoRplus$, define a new form
$A\cdot(X, \omega)$ by pulling back the conformal structure on $\cx$ and form $dz$ via the new atlas $\{A\circ \phi_\alpha\}$, where $A$ acts on
$\cx\isom\reals^2$ in the usual way.

This defines a free action of $\GLtwoRplus$ on $\Omega\teich_g$.  This action commutes with the action of $\Gamma_g$ and preserves the
stratification, so we obtain an action of $\GLtwoRplus$ on $\omn$ and its strata.

The action of $\SLtwoR$ on these spaces preserve the loci of unit area forms.

The $\GLtwoRplus$ action naturally extends to an action on $\Omega\barmoduli[g]$ which is no longer free.

\paragraph{Projectivization.}

It will be convenient to work with the projectivized moduli spaces $\pom[g]$ of projective Abelian differentials.  The foliation of
$\Omega_1\moduli[g]$ by $\SLtwoR$-orbits descends to a foliation $\mathcal{F}$ of $\pom$ by hyperbolic Riemann surfaces.

\section{Real multiplication}
\label{sec:hilb-modul-surf}

In this section, we discuss background material on real multiplication, eigenforms for real multiplication, and Hilbert modular surfaces.
We also introduce the eigenform loci in $\Omega\moduli$ and the \Teichmuller curves that they contain.

\paragraph{Orders.}

A \emph{real quadratic discriminant} is a positive integer $D$ with $D\equiv 0$ or $1 \pmod 4$.  In this paper, we will also require our quadratic
discriminants to by nonsquare.  A quadratic discriminant $D$ is \emph{fundamental} if $D$ is not of the form $f^2 E$ with $E$ a quadratic
discriminant and $f>1$ an integer.

Given a quadratic discriminant $D$, let $\ord$ be the ring,
\begin{equation*}
  \ord = \zed[T]/(T^2 + bT + c),
\end{equation*}
where $b, c\in\zed$ with $b^2 - 4c=D$.  The ring $\ord$ depends only on $D$ and is the unique \emph{quadratic order} of discriminant $D$.

Let $\K\isom \ratls(\sqrt{D})$ be the quotient field of $\ord$.  If
$D$ is fundamental, then $\ord$ is the ring of integers in $\K$.

There is an inclusion $\ord\to \ord[E]$ if and only if $D=f^2E$ for some integer $f$ and quadratic discriminant $E$.

We fix for the rest of this paper two embeddings $\iota_i\colon \K\to\reals$ with $i=1, 2$.  We will sometimes implicitly assume that
$\K\subset\reals$ via the first embedding $\iota_1$.  We will also use the notation $\lambda^{(i)} = \iota_i(\lambda)$ for $\lambda\in\K$.

\paragraph{Real multiplication.}

Consider a polarized Abelian surface $A= \cx^2/\Lambda$.  
\emph{Real multiplication} by $\ord$ on $A$ is a embedding of rings $\rho\colon\ord\to\End A$ with the following properties:
\begin{itemize}
\item $\rho$ is self-adjoint with respect to the polarization (a symplectic form on $\cx^2$).
\item $\rho$ does not extend to some $\ord[E]\supset \ord$.
\end{itemize}

\paragraph{Eigenforms.}

Given a form $(X, \omega)\in\Omega\moduli$, we say that $(X, \omega)$ is an \emph{eigenform for real multiplication} by $\ord$ if $\Jac(X)$
admits real multiplication $$\rho\colon\ord\to\End\Jac(X)$$ with $\omega$ as an eigenform.  More precisely, the self-adjointness of $\rho$
ensures an eigenspace decomposition of $\Omega(X)$,
\begin{equation}
  \label{eq:16}
  \Omega(X) = \Omega\Jac(X) =  \Omega^1(X) \oplus \Omega^2(X),
\end{equation}
and $\omega$ is an eigenform if it lies in one of the $\Omega^i(X)$.

Given a choice of real multiplication $\rho$ on $\Jac(X)$, the eigenspaces  $\Omega^i(X)$ correspond naturally to embeddings of $\ord$.  We
say that a form $\omega_i$ is an $i$-eigenform if
\begin{equation*}
  \omega(\rho(\gamma)\cdot\lambda) = \lambda^{(i)} \omega(\lambda)
\end{equation*}
for every $\gamma\in H_1(X; \zed)$ and $\lambda\in\ord$.  We let $\Omega^i(X)$ be the space of $i$-eigenforms.

\paragraph{Eigenform locus.}

Given a real quadratic discriminant $D$, let $$\Omega\E\subset\Omega\moduli$$ be the locus of eigenforms for real multiplication by $\ord$.  Let
$$\Omega\W = \Omega\E\cap \Omega\moduli(2),$$ the locus of eigenforms for real multiplication with a double zero.  McMullen proved in
\cite{mcmullenbild}: 

\begin{theorem}
  \label{thm:invariantorbifolds}
  The spaces $\Omega\E$ and $\Omega\W$ are closed, $\GLtwoRplus$-invariant suborbifolds of $\Omega\moduli$.
\end{theorem}

Let $\E$ and $\W\subset \pom$ be the respective projectivizations.  By \cite{mcmullenbild}, $\W$ is a (possibly disconnected) curve which is
a union of closed leaves of the foliation $\mathcal{F}$ of $\pom$.  Such closed leaves in $\proj\Omega\moduli[g]$ are called
\emph{\Teichmuller curves} because they project to curves in $\moduli[g]$ which are isometrically embedded with respect to the
\Teichmuller metric on $\moduli[g]$.

Let $\Omega \E(1,1) = \Omega\E\setminus\Omega\W$, the locus of eigenforms with two simple zeros, and let $\E(1,1)$ be its projectivization.

\paragraph{Hilbert modular surfaces.}

Let $\SLtwoord\subset\SL_2\K$ be the subgroup preserving the lattice $\ord\oplus\ord^\vee$, where $\ord^\vee$ is the \emph{inverse
  different},
\begin{equation*}
  \ord^\vee = \frac{1}{\sqrt{D}}\ord.
\end{equation*}
Concretely,
\begin{equation*}
  \SLtwoord = \left\{
    \begin{pmatrix}
      a & b \\
      c & d
    \end{pmatrix}
    \in \SL_2\K : a, d\in \ord, b\in(\ord^\vee)^{-1}, c\in\ord^\vee\right\}.
\end{equation*}

$\SLtwoord$ is the group of symplectic $\ord$-linear automorphisms of $\ord\oplus\ord^\vee$, where $\ord\oplus\ord^\vee$ is equipped with the
symplectic form,
\begin{equation*}
  \langle u, v\rangle = \Tr^{\K}_\ratls(u\wedge v)
\end{equation*}
with
\begin{equation*}
  (u_1, u_2)\wedge(v_1, v_2) = u_1 v_2 - u_2 v_1.
\end{equation*}

The \emph{Hilbert modular surface} $\X$ is the quotient,
\begin{equation*}
  \X = \half\times\half/\SLtwoord,
\end{equation*}
where $\SL_2\K$ acts on the $i$th factor of $\half\times\half$ by \Mobius transformations using the $i$th embedding, $\SL_2\K\to\SLtwoR$.
This definition of $\X$ is equivalent to the one in (\ref{eq:1}).

\begin{theorem}
  \label{thm:x-moduli-space}
  $\X$ is the moduli space for principally polarized Abelian surfaces with a choice of real multiplication by $\ord$.
\end{theorem}

\begin{proof}[Sketch of proof]
  Let $\widetilde{X}_D$ be the space of triples $(A, \rho, \phi)$, where $A$ is a principally polarized Abelian surface, $\rho\colon
  \ord\to\End A$ is real multiplication of $\ord$ on $A$, and $\phi\colon\ord\oplus\ord^\vee\to H_1(A; \zed)$ is a symplectic isomorphism of
  $\ord$-modules.

  Let $\{\alpha, \beta\}$ be a basis of $\ord\oplus\ord^\vee$ over $\ord$ with $\langle\alpha, \beta\rangle = 1$.  Define a map $\Phi\colon
  \widetilde{X}_D\to\half\times\half$, which can be shown to be an isomorphism, by
  \begin{equation*}
    \Phi(A, \rho, \phi) = \left(\frac{\omega_1(\beta)}{\omega_1(\alpha)}, \frac{\omega_2(\beta)}{\omega_2(\alpha)}\right),
  \end{equation*}
  where $\omega_i \in\Omega A$ is a nonzero $i$-eigenform.
  
  Forgetting the marking $\phi$, the map $\Phi$ descends to an isomorphism between the space of pairs $(A, \rho)$ and $\X$.
\end{proof}

See \cite{bl} or \cite{mcmullenhilbert} for a detailed proof of this theorem.

\paragraph{Partial compactification.}

Define the partial Deligne-Mumford compactification $\tildemoduli$ of $\moduli$ to be the open subvariety of $\barmoduli$ obtained by
adjoining to $\moduli$ those points of $\barmoduli$ representing two elliptic curves joined at a node.

Let $\siegelmod$ be the Siegel modular variety parameterizing principally polarized Abelian surfaces.  There is a natural morphism
$\Jac\colon\tildemoduli\to\siegelmod$ associating a surface $X$ to its Jacobian $\Jac(X)$, where the Jacobian of two elliptic curves joined
at a node is defined to be the product of those elliptic curves.  It is a well-known fact that $\Jac$ is an isomorphism.  We sketched a
proof of this in \cite[Proposition~5.4]{bainbridge06}.

\paragraph{Embedding of $\X$.}

There are natural embeddings $j_i\colon \X\to\potm$ defined by
\begin{equation*}
  j_i(A, \rho) = (X, [\omega_i]),
\end{equation*}
where $X$ is the unique Riemann surface with $\Jac(X)\isom A$, and $\omega_i$ is an $i$-eigenform for $\rho$.  In this paper, we will regard
$\X$ to be embedded in $\potm$ by $j_1$.  So one can think of a point in $\X$ as representing a either an Abelian variety with a choice of
real multiplication or a projective Abelian differential $(X, [\omega])$ which is an eigenform for real multiplication by $\ord$.  

For any eigenform $(X, [\omega])\in \X$, there are two choices of real multiplication $\rho\colon\ord\to\Jac(X)$ realizing $\omega$ as an
eigenform.  These choices are related by the Galois involution of $\ord$.  We always choose $\rho$ so that $\omega$ is a $1$-eigenform

We identify $\E$ with $j_1^{-1}(\proj\Omega\moduli)\subset\X$.  The bundle $\Omega\E$ extends to a line bundle $\Omega\X$ over $\X$ whose
fiber over $(X, [\omega])$ is $\Omega^1(X)$, the forms in the projective class $[\omega]$.  One can think of $\Omega\X$ as the bundle of
choices of scale for the metric on $(X, [\omega])$.

\paragraph{Product locus.}

Define the \emph{product locus} $\P\subset\X$ to be the locus of points in $\X$ which represent two elliptic curves joined to a node.
Alternatively, in terms of Abelian surfaces, $\P$ is the locus of Abelian surfaces in $\X$ which are polarized products of elliptic curves.

In the universal cover $\half\times\half$ of $\X$, the curve $\P$ is a countable union of graphs of M\"obius transformations.

As a subset of $\X$, we have $\E = \X\setminus\P$.  Thus it is a Zariski-open subset of $\X$.  As $\W\subset\E$, the curves $\W$ and
$\P$ are disjoint.

For more information about the curve $\P$, see \cite{mcmullenhilbert} (where it is called $\X(1)$) and \cite{bainbridge06}.

\paragraph{Involution of $\X$.}

The involution $\tilde{\tau}(z_1, z_2) = (z_2, z_1)$ of $\half\times\half$ descends to an involution $\tau$ of $\X$.  On the level of
Abelian varieties with real multiplication, $\tau$ sends the pair $(A, \rho)$ to $(A, \rho')$, where $\rho'$ is the Galois conjugate real
multiplication.

On the level of projective Abelian differentials, we have $\tau(X, [\omega]) = (X, [\omega'])$, where $\omega'$ is an eigenform in the
eigenspace which does not contain $\omega$.  Since $\tau$ does not change the underlying stable Riemann surface, we obtain

\begin{prop}
  \label{prop:taufixesP}
  The involution $\tau$ satisfies $\tau(\P) = \P$.
\end{prop}

\paragraph{The $\SLtwoR$-orbit foliation.}

We have the foliation $\mathcal{F}$ of $\potm$ whose leaves are the images of $\GLtwoRplus$ orbits in $\Omega\tildemoduli$.  By
Theorem~\ref{thm:invariantorbifolds}, $\X$ is saturated with respect to this foliation, meaning that a leaf which intersects $\X$ is
contained in $\X$.  Thus there is an induced foliation $\F$ of $\X$ by Riemann surfaces.

The foliation $\F$ was defined in \cite{mcmullenhilbert} where it was shown to be transversely quasiconformal and given a transverse
invariant measure, which we will define below.

\paragraph{Kernel foliation.}

The vertical foliation of $\half\times\half$ is preserved by the action of $\SLtwoord$ and so descends to a foliation of
$\X$.  We will call this foliation $\A$.

Let $\Omega_1\A$ be the foliation of $\Omega_1\X$ obtained by pulling back $\A$ by the natural projection.  The foliations $\Omega_1\A$ and
$\A$ are characterized by the following property, proved in \cite[Proposition~2.10]{bainbridge06} and \cite{mcmullenhilbert}.

\begin{prop}
  \label{prop:periodsconstant}
  The absolute periods of forms in $\Omega_1\X$ are constant along the leaves of $\Omega_1\A$.  The absolute periods of $1$-eigenforms
  are constant along the leaves of $\A$ up to the action of $\cx^*$.
\end{prop}

The foliation $\A$ is transverse to $\F$ (see \cite[Theorem~8.1]{mcmullenhilbert}).

There is a foliation of $\Omega_1\moduli[g]$ along which absolute periods of forms are constant, called the kernel foliation.
Proposition~\ref{prop:periodsconstant} implies that $\Omega_1\A$ is the foliation of $\Omega_1\X$ induced by the kernel foliation.

\paragraph{Transverse measure.}

We now define a canonical quadratic differential along the leaves of $\Omega_1\A$ and use it to define a
transverse invariant measure to $\F$ following \cite{mcmullenhilbert}.  

Let $f\colon \Omega_1\E(1,1)\to\cx$ be the multivalued holomorphic function defined by
\begin{equation*}
  f(X, \omega) = \int_p^q \omega,
\end{equation*}
where $p$ and $q$ are the zeros of $\omega$.  The function $f$ is multivalued because there is a choice of ordering of the zeros and a
choice of path between them.  Define
\begin{equation*}
  q = (\partial f)^2,
\end{equation*}
where the derivative is along leaves of $\Omega_1\A$.  Since the absolute periods are constant along $\Omega_1\A$, the form $q$ is well
defined.  The form $q$ is a meromorphic quadratic differential on $\Omega_1\X$ along the leaves of $\Omega_1\A$ which is zero along
$\Omega_1\W$, has simple poles along $\Omega_1\P$, and is elsewhere nonzero and finite (see \cite{mcmullenhilbert} or \cite[\S
10]{bainbridge06} for proofs).

The measure $|q|$ on leaves of $\Omega_1\A$ is invariant under the action of $\SOtwoR$, so it descends to a leafwise measure $|q|$ on $\A$.

\begin{theorem}[\cite{mcmullenhilbert}]
  \label{thm:transversemeasure}
  The leafwise measure $|q|$ defines a transverse, holonomy invariant measure to $\F$.
\end{theorem}

Let $\mu_D$ be the measure on $\X$ (supported on $\E(1,1)$) defined by taking the product of the transverse measure $|q|$ with the leafwise
hyperbolic metric on $\F$.

We say that a measure $\mu$ on $\pom[g]$ is invariant if it is the push-forward of a $\SLtwoR$-invariant measure on
$\Omega_1\moduli[g]$, or equivalently if it is the product of a transverse invariant measure with the leafwise hyperbolic metric.  In this
sense, $\mu_D$ is invariant.  McMullen proved in \cite{mcmullenabel}:

\begin{theorem}
  \label{thm:finiteinvariant}
  The measure $\mu_D$ is finite and ergodic.  Moreover, it is the only such invariant measure supported on $\E(1,1)$.
\end{theorem}

\begin{remark}
  The invariant measure in \cite{mcmullenabel} was actually defined in a slightly different way, but it is not hard to see directly that the two are
  the same up to a constant multiple independent of $D$.
\end{remark}

We also give $\Omega_1\E(1,1)$ the unique $\SLtwoR$-invariant measure $\mu_D^1$ such that $\pi_*\mu_D^1 = \mu_D$.  This is the product of
$\mu_D$ with the uniform measure of unit mass on the circle fibers of $\pi$.

\paragraph{Euler characteristic of $\X$.}

Siegel calculated the orbifold Euler characteristic $\chi(\X)$ in terms of values of the Dedekind zeta function $\zeta_{\K}$ of $\K$.  He
showed in \cite{siegel36} (see also \cite[Theorem~2.12]{bainbridge06}):

\begin{theorem}
  \label{thm:chiX}
  If $D$ is a fundamental discriminant, and $f\in\nats$, then
  \begin{equation}
    \label{eq:17}
    \chi(\X[f^2D]) = 2 f^3 \zeta_{K_D}(-1) \sum_{r|f}\kron{D}{r}\frac{\mu(r)}{r^2}.
  \end{equation}
\end{theorem}

\begin{remark}
  Here, $\kron{D}{r}$ is the Kronecker symbol, defined in \cite{miyake}, and $\mu$ is the M\"obius function.
\end{remark}

For a fundamental discriminant $D$, Cohen defined in \cite{cohen75}:
\begin{equation}
  \label{eq:18}
  H(2, f^2 D) = -12 \zeta_{K_D}(-1) \sum_{r|f}\mu(r)\kron{D}{r}r \sigma_3\left(\frac{f}{r}\right),
\end{equation}
where
\begin{equation*}
  \sigma_m(n) = \sum_{d|n}d^m.
\end{equation*}
From M\"obius inversion, and \eqref{eq:17}, we obtain:
\begin{equation*}
  \sum_{r|f} \chi(\X[r^2 D]) = -\frac{1}{6} H(2, f^2D).
\end{equation*}
Cohen proved in \cite{cohen75} (see also \cite{siegel69}):
\begin{theorem}
  \label{thm:h2dsum}
  For any real quadratic discriminant $D$,
  \begin{equation*}
     H(2, D) = -\frac{1}{5}\sum_{e\equiv D\,(2)} \sigma_1\left(\frac{D-e^2}{4}\right).
  \end{equation*}
\end{theorem}

\paragraph{Cusps of $\X$.}

Given $\sigma\in \proj^1(\K)$ and $r>0$, define
\begin{equation*}
  N_r(\sigma) = A^{-1}\{x_j +i y_j\in\half\times\half : y_1 y_2 >r\}
\end{equation*}
for some $A\in\SL_2\K$ such that $A\sigma = \infty$.  The set $N_r(\sigma)$ is independent of the choice of $A$.
Let $\Gamma_\sigma$ be the stabilizer of $\sigma$ in $\SLtwoord$, and define
\begin{equation*}
  \mathfrak{N}_r(\sigma) =  N_r(\sigma)/\Gamma_\sigma.
\end{equation*}

We call a $\SLtwoord$-orbit in $\K$ a \emph{cusp} of $\X$. Let $\Cusps\subset \proj^1(\X)$ be a set of representatives for the set of cusps of $\X$.

\begin{theorem}
  \label{thm:coreandcusps}
  If $r$ is sufficiently large, then for each $\sigma\in \proj^1(\K)$, the natural embedding $\mathfrak{N}_r(\sigma)\to \X$, is injective.
  Furthermore, $\X$ is the disjoint union,
  \begin{equation*}
    \X = K\cup \bigcup_{\sigma\in\Cusps} \mathfrak{N}_r(\sigma),
  \end{equation*}
  where $K$ is a compact submanifold with boundary which is a deformation retract of $\X$.
\end{theorem}

\begin{proof}
  See the discussion on pp.~7-11 of \cite{vandergeer88}.
\end{proof}

$\X$ can be compactified by adding one point for each cusp of $\X$.  The sets $\mathfrak{N}_r(\sigma)$ form a neighborhood basis of the cusp
$\sigma$.

Given a cusp $\sigma$ of $\X$, define the covering,
\begin{equation*}
  \X^\sigma = \half\times\half / \Gamma_\sigma.
\end{equation*}

\paragraph{Full $\ord$-modules.}

Given an Abelian surface with real multiplication $(A, \rho)\in \X$, we say that an $\ord$-submodule $M\subset H_1(A; \zed)$ is \emph{full}
if for any $x\in M$ and $n\in\zed$, we have  $n x \in M$ if and only if $x\in M$.  We say that the rank of $M$ is the dimension of
$M\otimes\ratls$ as a $\K$-vector space.  Let $\mathcal{X}_D$ be the disjoint union,
\begin{equation*}
  \mathcal{X}_D= \bigcup_{c\in\Cusps} \X^c.
\end{equation*}

\begin{prop}
  \label{prop:fullrankone}
  $\mathcal{X}_D$ is the moduli space of all triples $(A, \rho, M)$ with $(A, \rho)\in \X$  and $M\subset H_1(A; \zed)$ a full, rank one
  $\ord$-submodule.
\end{prop}

\begin{proof}
  Let $[x:y]\in\proj^1(\K)$.  By the proof of Theorem~\ref{thm:x-moduli-space}, $\X^{[x:y]}$ is the moduli space of triples $(A, \rho,
  \phi)$, where $A$ is a principally polarized Abelian surface, $\rho$ is real multiplication by $\ord$, and $\phi$ is a symplectic isomorphism
  $\phi\colon\ord\oplus\ord^\vee\to H_1(A;\zed)$ defined up to the action of $\Gamma_{[x:y]}$ on $\ord\oplus\ord^\vee$.
 
  Given such a triple $(A, \rho, \phi)\in\X^{[x:y]}$, define 
  $M\subset H_1(A; \zed)$ as follows.   Define
  \begin{equation*}
    M_{[x:y]} = \K\cdot(x, y) \cap \ord\oplus\ord^\vee.
  \end{equation*}
  Since $M_{[x:y]}$ is preserved by the action of $\Gamma_{[x:y]}$, we can define  $M = \phi(M_{[x:y]})$.

  Conversely, given a triple $(A, \rho, M)$, choose a symplectic $\ord$-isomorphism $\psi\colon \ord\oplus\ord^\vee\to H_1(A; \zed)$.  Then $\psi^{-1}(M) =
  M_{[x:y]}$ for some $[x:y]\in\proj^1(\K)$.  Composing $\psi$ with an element of $\SLtwoord$, we can assume $[x:y]\in C(\X)$.  The marking
  $\psi$ is then unique up to composition with elements of $\Gamma_{[x:y]}$.  Then $(A, \rho, \psi)$ yields a well-defined point in $\X^{[x:y]}$.
\end{proof}

\paragraph{Homological directions.}

$\mathcal{X}_D$ also parameterizes eigenforms for real multiplication equipped with a homological direction:

\begin{prop}
  \label{prop:homologicaldirection}
  $\mathcal{X}_D$ is the moduli space of directed Abelian differentials $(X, [\omega], v)$ with $(X, [\omega])\in\X$ an eigenform for real
  multiplication and $v$ a homological direction.
\end{prop}

\begin{proof}
  Let $(X, [\omega])\in\X$.  Using Proposition~\ref{prop:fullrankone}, we need only to show that there is a natural correspondence between
  full, rank one $\ord$-modules $M\subset H_1(X; \zed)$ and homological directions on $(X, [\omega])$.

  Given such an $M$, choose a representative $\omega$ of $[\omega]$ so that $\omega(\gamma)\in \reals$ for some $\gamma\in M$.  Since $M$ is
  rank one, for any $\gamma_1\in M$, we have $\gamma_1 = \lambda\cdot\gamma$ for some $\lambda\in\K$, so $\omega(\gamma_1) =
  \lambda^{(1)}\omega(\gamma)$.  Thus $\omega(\gamma_1)\in \reals$ for any $\gamma_1\in M$, and so this normalization of $[\omega]$ is
  canonically determined by $M$.  We then set $v$ to be the horizontal direction of $\omega$.

  Conversely, if $v$ is a homological direction on $(X, [\omega])$, choose a representative $\omega$ of $[\omega]$ so that $v$ is horizontal, and let
  \begin{equation*}
    M= \{\gamma\in H_1(X; \zed): \omega(\gamma)\in\reals\}.
  \end{equation*}
  $M$ is clearly a full $\ord$-submodule, and $M\neq H_1(X; \zed)$ because the periods can not all be real by the Riemann relations.  Thus
  $M$ is rank one.
\end{proof}

\paragraph{Detecting eigenforms.}

We conclude this section by showing that eigenforms for real multiplication in genus two can be detected in terms of their periods.

\begin{prop}
  \label{prop:realmultcriterion}
  Let $A$ be a principally polarized Abelian surface with a holomorphic 1-form  $\omega\in\Omega(A)$.  Suppose there is a monomorphism
  $\rho\colon\ord\to \End H_1(A; \zed)$ with the following properties:
  \begin{itemize}
  \item $\rho$ is self-adjoint with respect to the polarization.
  \item $\rho$ doesn't extend to a larger order $\ord[E]\supset\ord$.
  \item $\omega(\rho(\lambda)\cdot x) = \lambda^{(i)} \omega(x)$ for every $x\in H_1(A; \zed)$ and $\lambda\in \ord$.
  \end{itemize}
  Then $\rho$ defines real multiplication by $\ord$ on $A$ with $\omega$ an $i$-eigenform.
\end{prop}

\begin{proof}
  An endomorphism $\rho(\lambda)$ determines a real-linear endomorphism $T$ of $H_1(A; \reals)\isom\Omega(A)^*$.  The dual endomorphism
  $T^*$ of $\Omega(A)$ is self-adjoint, preserves the complex line spanned by $\omega$, and is complex linear on this line.  Thus $T^*$ is complex linear by a linear algebra
  argument (see \cite[Lemma~7.4]{mcmullenbild}), and so $T$ is complex linear as well.  Therefore $T$ defines real multiplication on $\ord$,
  with $\omega$ an eigenform.
\end{proof}

\section{Prototypes}
\label{sec:prototypes}

We now introduce a combinatorial object called a \emph{prototype}, which is closely related to the combinatorics of various subsets of
Hilbert modular surfaces and their compactifications.  Our prototypes are nearly the same as the splitting prototypes introduced in
\cite{mcmullenspin}.  What we call a prototype here was called a $\Y$-prototype in \cite{bainbridge06}.

\begin{definition}
  A \emph{prototype} of discriminant $D$ is a quadruple, $(a, b, c, \bar{q})$ with $a, b,  
  c\in\integers$ and $\bar{q}\in\integers/\gcd(a, b, c)$ which satisfies the following five properties:
  \begin{equation*}
    b^2-4 a c=D,\quad a>0, \quad c < 0, \quad \gcd(a, b, c, \bar{q})=1, \quad\text{and}\quad a+b+c < 0.
  \end{equation*}
\end{definition}

We let $\Yprot$ denote the set of prototypes of discriminant $D$.

We associate to each prototype $P=(a, b, c, \bar{q})\in\Yprot$ the unique algebraic number
$\lambda(P)\in \K$ such that $a \lambda(P)^2 + b\lambda(P) + c=0$ and $\lambda(P)>0$.
This makes sense because the two roots of $ax^2+bx+c=0$ have opposite signs.  It is easy to check that
the last condition $a + b + c <0$ is equivalent to $\lambda(P) > 1$.

\paragraph{Notation.}

Unless said otherwise, a prototype $P$ will always be given by the letters $P=(a,b , c, \bar{q})$, and the letters $a$, $b$, $c$, and
$\bar{q}$ will always belong to some implied prototype.  The letter $\lambda$ will usually denote the number $\lambda(P)$ for an implied
prototype $P$.  We define
\begin{equation*}
  (a', b', c', \bar{q}'):= \frac{(a, b, c, \bar{q})}{\gcd(a, b, c)},
\end{equation*}
with $\bar{q}'\in \ratls/\zed$.

\paragraph{Operations on prototypes.}

Given a prototype $P$, define the \emph{next prototype} $P^+$ by
\begin{align*}
  P^+&=
  \begin{cases}
    (a, 2a+b, a+b+c, \bar{q}), &\text{if } 4a + 2b + c < 0;\\
    (-a-b-c, -2a-b, -a, \bar{q}), &\text{if } 4a + 2b + c > 0.
  \end{cases}\\
  \intertext{Given a prototype $P$, define the \emph{previous prototype} $P^-$ by}
  P^-&=
  \begin{cases}
    (a, -2a+b, a-b+c, \bar{q}), &\text{if } a-b+c < 0;\\
    (-c, -b+2c, -a+b-c, \bar{q}), &\text{if } a-b+c > 0.
  \end{cases}
\end{align*}
Its easy to check that $P^+$ and $P^-$ are actually prototypes of the same discriminant and that
$$(P^+)^-=(P^-)^+=P.$$

Define an involution $t$ on the set of prototypes of discriminant $D$ by
$$
t(a, b, c, \bar{q})=
\begin{cases}
  (a, -b, c, \bar{q}), & \text{if } a - b + c <  0;\\
  (-c, b, -a, \bar{q}), & \text{if } a - b + c > 0.
\end{cases}
$$

\section{Three-cylinder surfaces}
\label{sec:three-cylind-surf}

We define
\begin{equation*}
  \Upsilon_D\subset\mathcal{X}_D
\end{equation*}
to be the set of directed surfaces $(X, [\omega], v)$ for which the direction $v$ is periodic, and we let
$\Upsilon_D^i\subset \mathcal{P}$ be the locus where the associated cylinder decomposition has $i$ cylinders.  The goal of this section is to
classify the connected components of $\U$ and define explicit coordinates on each connected component.

\paragraph{Parameterizing three-cylinder surfaces.}

Given two complex numbers $z_i$, let $P(z_1, z_2)$ be the parallelogram containing as sides the two segments $[0, z_i]$.

Given positive real numbers $x_1$, $x_2$, and $x_3$ such that
\begin{equation*}
  x_2 = x_1+x_3,
\end{equation*}
and complex numbers $y_1$, $y_2$, and $y_3\in \half$, define a surface,
\begin{equation*}
  S(x_1, x_2, x_3, y_1, y_2, y_3),
\end{equation*}
to be the surface $(X, \omega)$ obtained by gluing the three parallelograms $P_i = P(x_i, y_i)$ as in Figure~\ref{fig:Zshaped}.  This
surface only depends on $y_i$ mod ${x_i\zed}$.  We equip $(X, \omega)$ with
the homology classes $\alpha_i\in H_1(X; \zed)$ and $\gamma_i \in H_1(X, Z(\omega); \zed)$ indicated in Figure~\ref{fig:Zshaped}.
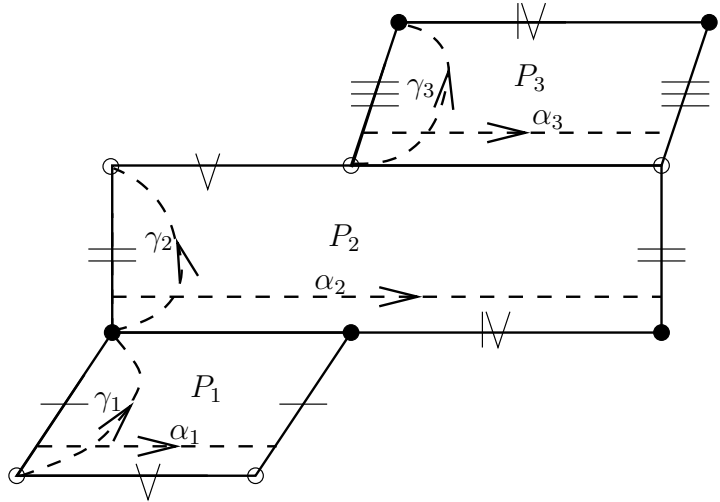
\begin{figure}[htbp]
  \centering
  \input{Zshaped.pstex_t}
  \caption{Three-cylinder surface}
  \label{fig:Zshaped}
\end{figure}

Define homology classes,
\begin{equation*}
  \beta_1 = \gamma_1 - \gamma_3, \quad\text{and}\quad  \beta_2 = \gamma_2 + \gamma_3.
\end{equation*}
Then
\begin{equation}
  \label{eq:19}
  (\alpha_1, \alpha_2, \beta_1, \beta_2)
\end{equation}
forms a symplectic basis of $H_1(X; \zed)$.

\paragraph{Canonical representation.}

We now define a canonical representation of every surface in $\U$ as some $S({\bf x}, {\bf y})$.

\begin{prop}
  \label{prop:canonicalrepresentation}
  Each directed surface $(X, [\omega], v)\in \U$ is of the form,
  \begin{equation*}
    (X, [\omega], v) = S({\bf x}, {\bf y}) = S(1, \mu, \mu-1, y_1, y_2, y_3),
  \end{equation*}
  for some $\mu\in\K$ with $\mu >1$ and $N(\mu)<0$.  This representation is unique taking $y_i$ mod $x_i \zed$.
\end{prop}

\begin{proof}
  Since the horizontal direction is periodic, we can put the surface in the form, $(X, [\omega], v) = S({\bf x}, {\bf y}),$ for
  some ${\bf x}$ and ${\bf y}$.  The module $M = \langle \alpha_1, \alpha_2\rangle\subset H_1(X;\zed)$ is an $\ord$-submodule because it is
  the set of all homology classes with real periods.  Thus we have
  $x_1/x_j\in\K$ for any $i, j$ since $\omega$ is an eigenform.

  We claim that $N^{\K}_\ratls(x_3/x_1)<0$.  To see this, let $N= H_1(X; \zed)/M$, another $\ord$-module.  The intersection pairing on homology gives a
  perfect pairing of $\ord$-modules,
  \begin{equation*}
    M\times N\to \zed.
  \end{equation*}
  The bases
  \begin{equation*}
    (u_1, u_2) = (\alpha_1, \alpha_3) \quad\text{and}\quad(v_1, v_2)=(\beta_1+\beta_2, \beta_2)
  \end{equation*}
  are dual bases of $M$ and $N$ respectively.  They satisfy $\omega(u_i)>0$ and $\Im \omega(v_i)>0$.  The claim then follows directly from
  Theorem~3.5 of \cite{bainbridge06}.

  We now have either
  \begin{equation}
    \label{eq:20}
    N^{\K}_\ratls(x_2/x_1)<0 \quad\text{or}\quad N^{\K}_\ratls(x_3/x_2)<0,
  \end{equation}
  but not both.  If the second holds, then swap $x_1$ and $x_3$ as well as $y_1$ and
  $y_3$, which does not change the surface.  We can then assume $N(x_2/x_1)<0$.  Finally divide $({\bf x}, {\bf y})$ by $x_1$ to put the
  surface in the required form.

  This representation is unique because if there were two such representations, then both alternatives in \eqref{eq:20} would hold, a
  contradiction. 
\end{proof}

We will always assume that any point in $\U$ is represented by the surface $(X, \omega)$ which is the canonical representative given in
Proposition~\ref{prop:canonicalrepresentation}, and $X$ will be equipped with the homology classes $\alpha_i$,
$\beta_i$, and $\gamma_i$ defined above.  The class $\gamma_i$ is really only defined up to adding a multiple of $\alpha_i$.  We just
choose any $\gamma_i$ so that $\alpha_i\cdot\gamma_j = \delta_{ij}$.  

\paragraph{Prototypes.}

We now assign to every $(X, \omega)\in \U$ a prototype $P(X, \omega)$.

Let $\mu=\omega(\alpha_2)\in\K$, and define $\phi_\mu(x) = ax^2 + bx +c$ to be the unique multiple of the minimal
polynomial of $\mu$ such that $b^2-4ac=D$ and $a>0$. Let $T$ be the matrix of the action of $a\mu$ on $H_1(X; \ratls)$ in the
symplectic basis \eqref{eq:19}.  Since $T$ is self-adjoint with respect to the intersection pairing, it is of the form
\begin{equation}
  \label{eq:21}
  T =
  \begin{pmatrix}
    0 & -c & \phantom{-}0 & \phantom{-}q \\
    a & -b & -q & \phantom{-}0 \\
    0 & \phantom{-}0 & \phantom{-}0 & \phantom{-}a \\
    0 & \phantom{-}0 & -c & -b
  \end{pmatrix}.
\end{equation}
We define $P(X, \omega) = (a, b, c, \overline{q})$.

\begin{prop}
  \label{prop:Pisaprototype}
  $P(X, \omega)$ is a well-defined prototype.
\end{prop}

\begin{proof}
  We have $c<0$ and $a+b+c<0$ because $\mu>1$ and $N^{\K}_\ratls(\mu)<0$ by
  Proposition~\ref{prop:canonicalrepresentation}.  Since $\ord\isom\zed[a\mu]$, we must have $\gcd(a, b, c, \bar{q})=1$, or else the
  action of $\ord$ would extend to an order $\ord[E]\supset\ord$.
  
  To see that $P(X, \omega)$ is well-defined, we must check that it is independent of the choice of the classes $\gamma_i$.  This is
  straightforward; see \cite[Theorem~7.21]{bainbridge06} for the proof.
\end{proof}

\paragraph{Coordinates on $\U$.}

Let $U_P\subset \U$ be the set of surfaces with prototype $P$. 
We now give coordinates for $\U$ by parameterizing each $U_P$.

\begin{lemma}
  \label{lem:eigenformequation}
  The surface $(X, \omega)=S(1, \mu, \mu-1, y_1, y_2, y_3)$ represents an eigenform in $U_P$ if and only if
  \begin{gather}
    \notag
    \mu=\lambda(P),\quad\text{and}\\
    \label{eq:22}
    a' y_1 + \frac{c'}{\mu}y_2 + \frac{a'+b'+c'}{\mu-1}y_3 \equiv -q' \pmod \zed,
  \end{gather}
  where $P=(a, b, c, \overline{q})$.
\end{lemma}

\begin{proof}
  First, suppose the equations \eqref{eq:22} hold.  We can then choose $\gamma_i$ such that $\alpha_i\cdot\gamma_j = \delta_{ij}$ and the following
  equations hold:
  \begin{align}
    \label{eq:73}
    a\mu\omega(\alpha_2) &= -c \omega(\alpha_1) - b\omega(\alpha_2), \\
    \notag
    a\mu\omega(\beta_1) &= -q \omega(\alpha_2) - c \omega(\beta_2),\quad\text{and} \\
    \notag
    a\mu\omega(\beta_2) &= q \omega(\alpha_1) + a \omega(\beta_1) - b \omega(\beta_2).
  \end{align}
  We then define real multiplication of $\ord$ on $\Jac(X)$ by
  \begin{equation}
    \label{eq:74}a\mu\cdot x = Tx,
  \end{equation}
  where $T$ is the matrix of \eqref{eq:21}.  This defines
  an action of $\ord$ using $\ord\isom\zed[a\mu]$,  By the above equations \eqref{eq:73}, $\omega(\lambda\cdot x) = \lambda^{(1)}\omega(x)$ for all $x\in
  H_1(X; \zed)$.  So by Proposition~\ref{prop:realmultcriterion}, this exhibits $(X, \omega)$ as an
  eigenform with prototype $P$.
  
  Conversely, suppose $(X, \omega)$ is an eigenform with prototype $P$.  The real multiplication is defined by \eqref{eq:74} with $T$ as in
  \eqref{eq:21}.  Since $(X, \omega)$ is an eigenform, the equations \eqref{eq:73} hold mod $\zed$, and \eqref{eq:22} follows.
\end{proof}

Define
\begin{multline}
  \label{eq:23}
  \fU_P = \biggl\{(y_1, y_2, y_3)\in \half/\zed\times\half/\lambda\zed\times\half/(\lambda-1)\zed : \\
  a' y_1 + \frac{c'}{\lambda}y_2 + \frac{a'+b'+c'}{\lambda-1}y_3 \equiv -q' \pmod \zed\biggr\},
\end{multline}
and define $\phi_P\colon \fU_P\to U_P$ by
\begin{equation*}
  \phi_P(y_1, y_2, y_3) = S(1, \lambda, \lambda-1, y_1, y_2, y_3).
\end{equation*}

\begin{theorem}
  \label{thm:threecylcoordinates}
  The map $\phi_P\colon\fU_P\to U_P$ is a biholomorphic isomorphism.
\end{theorem}

\begin{proof}
  By Lemma~\ref{lem:eigenformequation}, $\phi_P$ has image in $U_P$ and is surjective.  By Proposition~\ref{prop:canonicalrepresentation},
  $\phi_P$ is injective.  Finally, $\phi_P$ is locally biholomorphic by Theorem~\ref{thm:periodcharts}.
\end{proof}

\begin{cor}
  \label{cor:componentsofU}
  The connected components of $\U$ correspond bijectively to prototypes of discriminant $D$.
\end{cor}

\begin{proof}
  By Theorem~\ref{thm:threecylcoordinates}, $U_P$ is connected, so $P\mapsto U_P$ defines the required bijection.
\end{proof}

\paragraph{Covering of $U_P$.}

We will have use for the following covering of $U_P$.  Define
\begin{equation*}
  \hat{U}_P = \half / \frac{a'}{\gcd(a', c')}\lambda\zed \times \half/\frac{a'}{\gcd(a', a'+b'+c')}(\lambda-1)\zed, 
\end{equation*}
which is an $a'/\gcd(a', c')\gcd(a', a'+b'+c')$-fold covering of $U_P$ via
\begin{equation}
  \label{eq:24}
  g(y_2, y_3) = \phi_P\left(-\frac{c'}{a'\lambda}y_2 - \frac{a'+b'+c'}{a'(\lambda-1)}y_3 - \frac{q'}{a'}, y_2, y_3\right).
\end{equation}

\section{Compactification of $\X$}
\label{sec:compactification-x}

Let $\Y$ be the normalization as an algebraic variety of $\barX$, the closure of $\X$ in $\pobarm$.
We showed in \cite{bainbridge06}:

\begin{theorem}
  \label{thm:YDorbifold}
  $\Y$ is a compact, complex projective orbifold.
\end{theorem}

In this section, we summarize more of
the properties of the compactification $\Y$ which we proved in \cite{bainbridge06}.  In particular, we discuss the combinatorics of curves
in $\bdry \X$, and we give explicit local coordinates around points in $\bdry\X$.

\paragraph{Curves in $\pobarm$.}

We start by discussing some curves in $\pobarm$ which will be covered by curves in $\bdry\X$.

Given $\lambda>1$ define $\sC_\lambda\subset\pobarm$ to be the closure of the locus of stable Abelian differentials $(X, [\omega])$ with two
separating nodes such that the ratio of the residues of $\omega$ at the nodes is $\pm \lambda^{\pm 1}$.

Let $c_\lambda\in\pobarm$ be
the point representing a stable Abelian differential with three nonseparating nodes having residues $1$, $\lambda$, and $\lambda-1$.

Let $p_\lambda\in \pobarm$ be the differential $(X, \omega)$ obtained by joining infinite cylinders of circumference $1$ and $\lambda$ at
one point to form a node.  The form $(X, \omega)$ has one separating node where $\omega$ has residue zero, and two nonseparating nodes where
$\omega$ has residues $1$ and $\lambda$.

Let $w_\lambda\in\pobarm$ be the unique form $(X, \omega)$ having a double zero and two nonseparating nodes having residues $1$ and
$\lambda$.  The form $(X, \omega)$ can be formed as follows.  Start with an infinite cylinder $C_1$ of circumference $\lambda$.  Cut $C_1$
along a segment of length $1$ and identify opposite ends of the segment to form a ``figure eight'' as in Figure~\ref{fig:wlambda}.  Then glue two
half-infinite cylinders of circumference $1$ to the boundary of the figure eight as shown in Figure~\ref{fig:wlambda}.

\begin{figure}[htbp]
  \centering
  \includegraphics{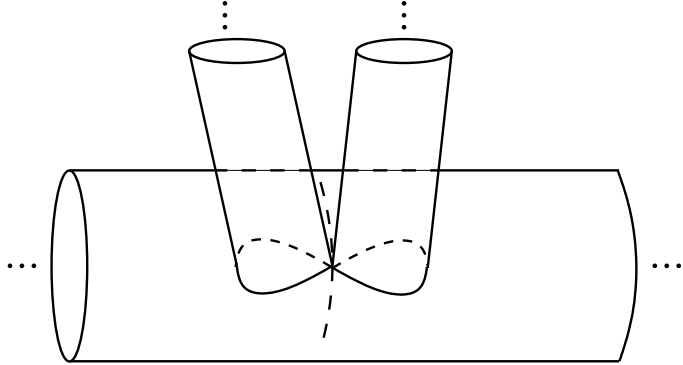}
  \caption{$w_\lambda$}
  \label{fig:wlambda}
\end{figure}

From Proposition~6.9 of \cite{bainbridge06}, we have

\begin{prop}
  \label{prop:sCrationalcurve}
  The curve $\sC_\lambda$ is a rational curve containing the points $c_\lambda$, $c_{\lambda+1}$, $w_\lambda$, and $p_\lambda$.  Every other
  point of $\sC_\lambda$ represents a stable differential with two nonseparating nodes and two simple zeros.
\end{prop}

We normalize every differential $(X, \omega)$ in $\sC_\lambda$ so that $\omega$ has residues $1$ and $\lambda$ at two of the nodes.

\begin{prop}
  \label{prop:horizontalperiodic}
  With this normalization, the horizontal foliation of every $(X, \omega)\in \sC_\lambda$ is periodic.
\end{prop}

\begin{proof}
  Each of the four half-infinite cylinders of $(X, \omega)$ is bounded by a union of saddle connections.  None of these saddle connections bounds two
  cylinders, or else $(X, \omega)$ would be a one-point connected sum of two infinite cylinders and we would be done.  Since $\omega$ has
  two zeros, counting multiplicity, this means that every separatrix of $(X, \omega)$ starts and ends at a zero.  But this implies $(X,
  \omega)$ is periodic.
\end{proof}

We now define a meromorphic quadratic differential $q_\lambda$ on $\sC_\lambda$ as we did on the leaves of $\A$ in
\S\ref{sec:hilb-modul-surf}.  Let $\sC_\lambda(1,1) = \sC_\lambda\setminus\{c_\lambda, c_{\lambda+1}, p_\lambda, w_\lambda\}$.  Define a
multivalued holomorphic function $f_\lambda$ on $\sC_\lambda(1,1)$  
by
\begin{equation*}
  f_\lambda(X, \omega) = \int_p^q\omega,
\end{equation*}
where $p$ and $q$ are the zeros of $\omega$, and define $q_\lambda = (\partial f_\lambda)^2$.

\begin{prop}
  \label{prop:qlambda}
  The differential $q_\lambda$ has double poles at $c_\lambda$ and $c_{\lambda+1}$, a simple pole at $p_\lambda$, a simple zero at
  $w_\lambda$, and is holomorphic and nonzero on $\sC_\lambda(1,1)$.  The horizontal foliation of $q_\lambda$ is  periodic, and
  the spine of $q_\lambda$ coincides with the locus of two-cylinder surfaces.  
\end{prop}

\begin{proof}
  Any surface $(X, \omega)$ sufficiently close to $c_\lambda$ or $c_{\lambda+1}$ is a three cylinder surface with one finite area cylinder
  $C$.  Define a loop on $\sC_\lambda$ by cutting $(X, \omega)$ along a closed geodesic of $C$ and regluing after a rotation.  This is a
  loop of length $\lambda$ or $\lambda+1$ contained in the horizontal foliation of $q_\lambda$.  Thus $(\sC_\lambda, q_\lambda)$ has
  half-infinite cylinders around $c_\lambda$ and  $c_{\lambda+1}$, so $q_\lambda$ has double poles there.

  The other zeros and poles are located just as for the quadratic differentials we defined along $T\A$ (see \S10 of \cite{bainbridge06}).

  Points along the separatrices emanating from $p_\lambda$ and $w_\lambda$ are obtained by performing a connected sum on $p_\lambda$ or
  splitting the double zero of $w_\lambda$ along a horizontal segment, which does not create a new cylinder.  Thus the separatrices are
  contained in the two-cylinder locus, an embedded graph on $\sC_\lambda$.  Thus the separatrices are saddle connections, and the horizontal
  foliation of $q_\lambda$ is periodic.
  
  By period coordinates, there is a tubular neighborhood $U$ of the spine such that $\bdry U$ consists of three cylinder surfaces.  We then have $a_t\cdot\bdry
  U\to \{c_\lambda\cup c_{\lambda+1}\}$ as $t\to\infty$, so $\cup a_t\cdot U$ covers $\sC_\lambda\setminus\{c_\lambda\cup c_{\lambda+1}\}$.
  It follows that the complement of the spine consists of three-cylinder surfaces as claimed. 
\end{proof}

The surface $(\sC_\lambda, q_\lambda)$ with its horizontal foliation is shown in Figure~\ref{fig:clambda} with the spine indicated by a
solid line.  The singular points of the foliation with three or one prong are $w_\lambda$ and $p_\lambda$ respectively.

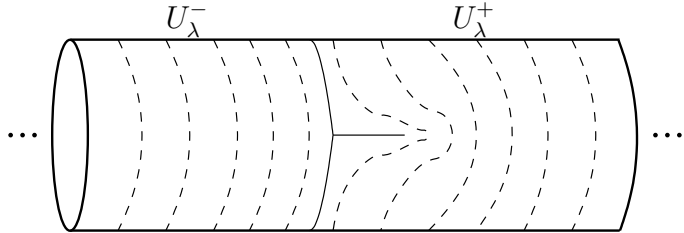
\begin{figure}[htbp]
  \centering
  \input{clambda.pstex_t}
  \caption{$\sC_\lambda$ with the quadratic differential $q_\lambda$}
  \label{fig:clambda}
\end{figure} 

The locus of three-cylinder surfaces in $\sC_\lambda$ consists of two connected components.  We refer to the component containing
$c_{\lambda+1}$ as $U_\lambda^-$ and we call the other component $U_\lambda^+$.

\paragraph{Curves in $\bdry\X$.}

In \cite{bainbridge06}, we defined for every prototype $P=(a, b, c, \overline{q})$ a curve $C_P\subset \bdry\X$ and a point $c_P\in\bdry
\X$.  The following properties of $\Y$ are proved in \S8 of that paper:

\begin{theorem}
  \label{thm:propertiesofYD}
  $\Y$ has the following properties:
  \begin{itemize}
  \item $\bdry\X = \bigcup_P C_P$, where the union is over all prototypes of discriminant $D$.
  \item $C_P$ is a rational curve which intersects $C_{P^+}$ at $c_P$ and intersects $C_{P^-}$ at $c_{P^-}$.  There are no other
    intersections between the curves $C_P$.
  \item The restriction of the natural morphism $\pi\colon\Y\to\pobarm$ realizes $C_P$ as a $\gcd(a',c')$-fold branched cover of
    $\sC_{\lambda(P)}$ with $\pi(c_P) = c_{\lambda(P)}$ and $\pi(c_{P^-}) = c_{\lambda(P)+1}$.  The restriction $\pi|_{C_P}$ is ramified to
    order $\gcd(a',c')$ at these two points and is elsewhere unramified.
  \item The points $c_P$ are cyclic quotient singularities of $\Y$ of order
    \begin{equation}
      \label{eq:25}
      m_P =\frac{a'}{\gcd(a', c')\gcd(a', b'+c')},
    \end{equation}
    and are the only singular points in $\bdry\X$.
  \item The curves $\barP$ and $\barW$ intersect $C_P$ transversely in $\gcd(a',c')$ points each and are disjoint from the points $c_{Q}$.    
  \item The involution $\tau$ of $\X$ extends to an involution $\tau$ of $\Y$ which sends $C_P$ to $C_{t(P)}$.
  \end{itemize}
\end{theorem}

It follows from this theorem that $\bdry\X$  consists of finitely many chains of rational curves $C_P$.  The connected components of $\bdry\X$ are naturally in
bijection with the cusps of $\X$.

We equip $C_P$ with a quadratic differential $q_P$ defined by,
\begin{equation*}
  q_P = (\pi|_{C_P})^*q_{\lambda(P)}.
\end{equation*}
The locus of three-cylinder surfaces in $C_P$ is the complement of the spine of $q_P$ by Proposition~\ref{prop:qlambda}.  We define
$U_P^-\subset C_P$ to be the component of the three-cylinder locus containing $c_P$, and we define $U_{P^-}^+\subset C_P$ to be the
component containing $c_{P^-}$.  Let $S_P\subset C_P$ be the spine of $q_P$, the two-cylinder locus.

Define
\begin{equation*}
  \mathcal{U}_P = U_P \cup U_P^+\cup U_P^- \cup c_P \subset \Y^\sigma,
\end{equation*}
where $\sigma$ is the cusp of $\X$ corresponding to $P$.  The set  $\mathcal{U}_P$ is a neighborhood of $c_P$.

\paragraph{Local coverings.}

Given a cusp $\sigma$ of $\X$, we defined the covering $\X^\sigma\to\X$ which is injective on a neighborhood $\mathfrak{N}_r$ of the
cusp of $\X^\sigma$.  Let $\overline{\mathfrak{N}}_r$ be the closure of $\mathfrak{N}_r$ in $\Y$, and define $\Y^\sigma$ to be the complex
orbifold obtained by gluing $\overline{\mathfrak{N}}_r$ to $\X^\sigma$.  Then $\Y^\sigma$ is a complex orbifold consisting of $\X^\sigma$
together with a single chain of rational curves, and the natural map $\Y^\sigma\to\Y$ is locally biholomorphic.

Each surface $(X, [\omega])\in \Y^\sigma$ has a canonical horizontal direction, and we can normalize $\omega$ (up to real multiple) so that
the horizontal direction of $\omega$ coincides with this canonical direction.

We then obtain an action of the upper-triangular subgroup $B\subset\SLtwoR$ on $\Y^\sigma$  coming from the action of $\GLtwoRplus$ on
$\Omega\barmoduli$.  This action fixes the points $c_P$ and fixes $S_P$ pointwise.

\paragraph{Type one coordinates.}

We now describe local coordinates around points of $C_P \setminus \{c_P \cup c_{P^-}\}$ which will be used in the proof of Lemma~\ref{lem:hyperbolicnorm}.

Let $W\subset \sC_\lambda\setminus\{c_\lambda, c_{\lambda+1}\}$ be a simply connected domain.  For each $z\in W$, write $(X_z, \omega_z)$
for the associated stable differential, normalized as before so that the residues at the nonseparating nodes are $1$ and $\lambda$.  Write
$N=\{n_1, n_2\} \subset X_z$ for the set of nonseparating nodes with $n_1$ being the residue $1$ node.  Choose homology classes
\begin{equation*}
  \alpha_i \in H_1(X_z \setminus N_z; \zed) \quad \text{and} \quad \beta_i\in H_1(X_z, N_z;\zed)
\end{equation*}
so that $\alpha_i$ goes around $n_i$, and $\alpha_i\cdot\beta_j = \delta_{ij}$.  Choose these classes consistently so that they are parallel
with respect to the Gauss-Manin connection.

We can plumb the nodes $n_i$ as described in \S\ref{sec:abel-diff} to obtain a new surface $(X_z', \omega_z')$.  The homology class
$\beta_i$ becomes a class in $H_1(X_z'; \zed)$ which we will continue to call $\beta_i$.  It is well defined up to adding a multiple of
$\alpha_i$.  If $w_i$ is sufficiently large, then there is a unique way to plumb the node $n_i$ so that
\begin{equation}
  \label{eq:26}
  e^{2\pi i\omega'_{z}(\beta_i)/\omega'_z(\alpha_i)} = w_i.
\end{equation}
Let $\Plumb(z, w_1, w_2)$ be the result of plumbing $(X_z, \omega_z)$ so that \eqref{eq:26} holds.  If one of the $w_i$ is zero, the
corresponding node of $(X_z, \omega_z)$ remains a node of $\Plumb(z, w_1, w_2)$.

Let $U\subset C_P$ be a component of $\pi^{-1}(W)$.  Define $f\colon U\times V \to \Y^\sigma$ by
\begin{equation*}
  f(u, v) = \Plumb(\pi(u), v^{-c'/\gcd(a', c')}, r v^{a'/\gcd(a', c')}),
\end{equation*}
where $r$ satisfies
\begin{equation}
  \label{eq:27}
  r^{-c'} = e^{2 \pi i q'},
\end{equation}
and $V\subset \cx$ is a neighborhood of $0$ small enough that $f$ is defined on $U\times V$.

From Theorem~7.22 and Proposition~8.1 of \cite{bainbridge06}, we have:

\begin{theorem}
  \label{thm:typeonecoordinates}
  For some choice of $r$ satisfying \eqref{eq:27}, $f$ is biholomorphic onto its image in $\Y^\sigma$, a neighborhood of $U\subset C_P$.
\end{theorem}

\begin{remark}
  The choice of $r$ in \eqref{eq:27} is equivalent to the choice of the component $U$ of $\pi^{-1}(W)$.
\end{remark}

It follows that $f$ defines coordinates $(u, v)$ around points in $U$.  We will call these \emph{type one coordinates}.

\paragraph{Type two coordinates.}

We now describe local coordinates around the points $c_P\in\Y$.  Define $R_\lambda\colon \Delta^3\to\pobarm$ by
\begin{equation*}
  R_\lambda(w_1, w_2, w_3) = S\left(1, \lambda, \lambda-1, \frac{1}{2\pi i}\log w_1, \frac{\lambda}{2\pi i}\log w_2, \frac{\lambda-1}{2\pi
      i}\log w_3\right),  
\end{equation*}
where if $w_i=0$, then the corresponding cylinder is instead a node.  Define $g_P\colon\Delta^2\to \pobarm$ by
\begin{multline*}
  g_P(u, v) = R_{\lambda(P)}\bigg(e^{2\pi i q'/a'} u^{-c'/\gcd(a', c')}v^{-(a'+b'+c')/\gcd(a', a'+b'+c')},\\
  u^{a'/\gcd(a', c')}, v^{a'/\gcd(a', a'+b'+c')}\bigg).
\end{multline*}

Given relatively prime  $m, s\in\zed$ with $m>0$, let $\theta_m$ be an $m$-th root of unity, and let $\Gamma_{m, s}$ be the group of
automorphisms of $\Delta^2$ generated by
\begin{equation*}
  (z, w)\mapsto (\theta_m z, \theta_m^s w).
\end{equation*}

From Theorem~7.27 and Proposition~8.3 of \cite{bainbridge06}, we have
\begin{theorem}
  \label{thm:typetwocoordinates}
  $g_P$ lifts to a holomorphic map $f\colon\Delta^2\to\Y^\sigma$, and $f$ descends to a map $\tilde{f}\colon \Delta^2/\Gamma_{m, s}\to
  \Y^\sigma$, where $m=m_P$ defined in \eqref{eq:25}.  The map $f$ is biholomorphic onto its image $\mathcal{U}_P$, a neighborhood of $c_P$.
  Furthermore,
  \begin{align*}
    f^{-1}(C_P) &= \{(u, v)\in\Delta^2\colon u=0\} \\
    f^{-1}(C_{P^+}) &= \{(u, v)\in\Delta^2\colon v=0\}.
  \end{align*}
\end{theorem}

We call the coordinates $(u, v)$ induced by $f$ on the neighborhood $\mathcal{U}_P$ of $c_P$ \emph{type two coordinates}.

\paragraph{Singularities of $\F$.}

We finish this section by discussing the geometry of the foliation $\F$ near $\bdry \X$.

\begin{prop}
  \label{prop:Fdoesntextend}
  The foliation $\F$ extends to a foliation of
  \begin{equation}
    \label{eq:30}
    \Y\setminus\bigcup_{P\in\mathcal{Y}_D}(c_P\cup S_P)
  \end{equation}
  and does not extend to any larger set.
\end{prop}

\begin{proof}
  We work in a local covering $\Y^\sigma$ where the leaves of $\F$ are the orbits of $B$.  Since the $B$-action is free on the locus
  \eqref{eq:30}, its orbits there are the required extension of $\F$.
  
  In the neighborhood $\mathcal{U}_P$ of $c_P$, the horizontal foliation of every surface $(X, \omega)$ has three cylinders, so $a_t\cdot(X,
  \omega)\to c_P$.  Thus every leaf of $\F$ near $c_P$ passes through $c_P$, which is impossible if $\F$ can be extended to a foliation.
  
  Suppose $\F$ can be extended over a neighborhood of $p\in S_P$.  Choose type one coordinates $(u,v)$ around $p$ so that $p=(0,0)$.  In
  these coordinates $\{0\}\times\Delta_\epsilon$ is contained in a leaf of $\F$ transverse to $C_P$.  The subsets $U_P^-$ and $U_{P^-}^+$ of
  $C_P$ are also contained in leaves of $\F$.  Thus $C_P$ is a leaf of $\F$ near $p$, a contradiction because then $\F$ would have two
  leaves meeting at $p$.
\end{proof}

\section{$\F$ as a current on $\Y$}
\label{sec:foliation-f-as}

We saw in Proposition~\ref{prop:Fdoesntextend} that $\F$ doesn't extend to a foliation of $\Y$.  Nevertheless, we will show in this
section that the closed current on $\X$ defined by the measured foliation $\F$ does extend to $\Y$:

\begin{theorem}
  \label{thm:Fclosedcurrent}
  Integration of smooth $2$-forms on $\Y$ over $\F$ defines a closed $2$-current on $\Y$.
\end{theorem}

We will then use this to interpret $\vol \E(1,1)$ as an intersection of classes in $H^2(\Y; \reals)$.

\paragraph{\Poincare growth.}

Let $\overline{X}$ be a compact, complex orbifold, and let $X\subset\overline{X}$ be the complement of a divisor $D$.  Suppose $D$ is
covered by coordinate charts of the form $\Delta^n/G$, where
\begin{itemize}
\item Each transformation $g\in G$ is of the form,
$$g(z_1, \ldots, z_n) = (\theta_1 z_1, \ldots, \theta_n z_n),$$ for some roots of unity $\theta_i$;
\item $D\cap \Delta$ is a union of the coordinate axes $z_1=0, \ldots, z_r=0$ for $1\leq r \leq n$.
\end{itemize}
In such a chart $\Delta^n/G$, we have $\Delta^n\cap X=(\Delta^*)^r\times\Delta^{n-r}$.  We give $\Delta^n\cap X$ a metric $\rho$ by putting
the \Poincare metric,
\begin{equation*}
  ds^2=\frac{|dz|^2}{|z|^2(\log|z|)^2},
\end{equation*}
on the $\Delta^*$ factors, putting the Euclidean metric $|dz|^2$ on the $\Delta$ factors, and defining $\rho$ to be the product metric.
Following Mumford \cite{mumford77}, we say that a form $\omega$ on $X$ has \emph{\Poincare growth} if there is a covering of $D$ by charts
$U_\alpha/G_\alpha$ with metrics $\rho_\alpha$ of this form such that for each $\alpha$, \label{page:poincare-growth}
\begin{equation*}
  \|\omega\|_{\rho_\alpha}\domed 1.
\end{equation*}

Let $\omega_i$ be the $2$-form on $\X$ covered by the forms,
\begin{equation*}
  \tilde{\omega}_i = \frac{1}{2\pi}\frac{dx_i\wedge dy_i}{y_i^2},
\end{equation*}
on $\half\times\half$.

\begin{prop}
  \label{prop:omegaipoincare}
  The forms $\omega_i$ have \Poincare growth.  Furthermore, there are smooth $1$-forms $\eta_i$ on $\X$ with \Poincare growth such that
  $\omega_i-d\eta_i$ are smooth $2$-forms on $\Y$.  Integration against $\omega_i$ defines a closed $2$-current on $\Y$.
\end{prop}

\begin{proof}
  We showed in \cite[Proposition~2.8]{bainbridge06} that $\omega_i$ is the Chern form of a metric $h_i$ on a line bundle $Q^i\Y$ over $\Y$, and we showed in
  \cite[Theorem~9.8]{bainbridge06} that $h_i$ is what Mumford calls a good metric in \cite{mumford77}.  It follows from
  \cite[Theorem~1.4]{mumford77} that the $\omega_i$ have \Poincare growth.  The second statement follows from the proof of the same
  theorem, and the last statement follows from Propositions~1.1 and 1.2 of \cite{mumford77}.
\end{proof}

\paragraph{Bounded norm.}

In order to prove Theorem~\ref{thm:Fclosedcurrent}, we need to bound the norm $\|\omega\|_{\F}$ for smooth forms $\omega$ on $\Y$.  In fact,
we will need these bounds for forms which just have \Poincare growth.

\begin{lemma}
  \label{lem:hyperbolicnorm}
  For any $p$-form $\omega$ on $\X$ with \Poincare growth, $\|\omega\|_{\F}$ is bounded.
\end{lemma}

\begin{proof}
  Let $U\subset\Y$ be an open set meeting $\bdry\X$ with the coordinates $(u, v)$ defined in \S\ref{sec:compactification-x}.  Let
  $K\subset U$ be compact.  Let $\rho$ be the metric on $U$ as defined above.  Since $\|\omega\|_\rho$ is
  bounded on $K$, it suffices to show that $\|v\|_\rho/\|v\|_{\F}$ is bounded on $K$ for any vector field $v$ tangent to $\F$ (where $\|v\|_{\F}$ denotes
  the norm of $v$ with respect to the hyperbolic metric on leaves of $\F$).  Since $\F$ is one-dimensional, it suffices to show $\|v\|_\rho$
  is bounded for a single vector field $v$ tangent to $\F$ with unit $\F$-norm.  There are two cases, depending on whether or not $U$ is a
  type one or type two coordinate chart.

  First suppose $U$ is a type two coordinate chart around $c_P$ in $\Y$.  If $U$ is small enough, we can regard it as a subset of $U_P$ from
  \S\ref{sec:three-cylind-surf}.  By Theorem~\ref{thm:threecylcoordinates}, the universal cover $\tilde{U}_P$ is isomorphic to
  $\half\times\half$ with coordinates $(y_2, y_3)$ from \S\ref{sec:three-cylind-surf}.  The $(u, v)$-coordinates are related to the $(y_2,
  y_3)$-coordinates by
  \begin{equation*}
    u = e^{i c_2 y_2} \quad\text{and}\quad v = e^{i c_3 y_3},
  \end{equation*}
  for positive real constants $c_2$ and $c_3$.
  The action of the diagonal group $A\subset\SLtwoR$ on $\tilde{U}_P$ is given by
  \begin{equation*}
    a_t \cdot(u_2+iv_2, u_3+iv_3) = (u_2 + i e^t v_2, u_3 + i e^t v_3),
  \end{equation*}
  where $y_j = u_j + i v_j$.  Thus
  \begin{equation*}
    v = \Im y_2 \pderiv{}{y_2} + \Im y_3 \pderiv{}{y_3}
  \end{equation*}
  is a unit length tangent vector field to the lift of $\F$ to $\tilde{U}_P$.  In the $(u, v)$-coordinates,
  \begin{equation*}
    v = -i u \log|u|\pderiv{}{u}  - i v \log|v|\pderiv{}{v}.
  \end{equation*}
  This is a unit length vector field tangent to $\F$ with bounded $\rho$-length.  

  The case when $U$ is a type one coordinate chart is similar.  For some function $f$, the vector field,
  \begin{equation*}
    v = f(u,v)\pderiv{}{u} - i v \log|v|\pderiv{}{v},
  \end{equation*}
  is a unit length vector field tangent to $\F$.  The function $f$ is continuous because $v$ is tangent to the flow on $\Y$ generated by
  $A$.  Therefore $\|v\|_\rho$ is bounded on $K$.  
\end{proof}

\begin{cor}
  \label{cor:Fcurrent}
  For any 2-form $\omega$ on $X$ with \Poincare growth,
  \begin{equation}
    \label{eq:37}
    \int_{\F} |\omega| < \infty.
  \end{equation}
  In particular, $\F$ defines a current on $\Y$.
\end{cor}

\begin{proof}
  The first statement follows directly from Lemma~\ref{lem:hyperbolicnorm} since $\mu_D$ has finite volume by
  Theorem~\ref{thm:finiteinvariant}.  Since any smooth form on $\Y$ has \Poincare growth, \eqref{eq:37} holds in particular for smooth
  forms.  For $\F$ to define a current, it suffices to show that if $\omega_n\to 0$ is a uniformly convergent sequence of forms on $\Y$,
  then $\int_{\F}\omega_n\to 0$.  This is easily seen from the proof of Lemma~\ref{lem:hyperbolicnorm} using the fact that
  $\|\omega_n\|_\rho\to 0$ in each chart $U$.
\end{proof}

\paragraph{Cusp neighborhoods.}

The proof that $\F$ is closed, will require certain well-behaved neighborhoods of the cusps of $\X$.

\begin{lemma}
  \label{lem:addaptedneighborhood}
  For any $\epsilon>0$, there is a closed, $H$-invariant neighborhood $W\subset \X$ of the cusps of $\X$ such that $\Vol W<\epsilon$ and
  $\bdry W$ is a submanifold of $\X$ transverse to $\F$.
\end{lemma}

\begin{proof}
  It suffices to construct a neighborhood of the cusp at infinity in the covering $\X^\sigma$ with the required properties.  Let
  $\rho\colon\reals\to\reals$ be smooth and  increasing with $\rho(x)=0$ for $x<1/2$ and $\rho(x)=x$ for $x>1$.  Recall that each point in $\X^\sigma$ represents an
  eigenform $(X, \omega)$ together with a choice of horizontal direction.  Normalize each $(X, \omega)$ to have unit area.  Define a smooth
  function $f$ on $\X^\sigma$ by
  \begin{equation*}
    f(X, \omega) = \sum_i \rho(\height(C_i)),
  \end{equation*}
  where the sum is over the horizontal cylinders of $\X$ ($f$ is zero if there are no horizontal cylinders).  The bump function $\rho$ is
  there to make $f$ smooth; otherwise $f$ would just be continuous.

  Let $W_\ell = f^{-1}[\ell, \infty)$.  The function $f$ is $H$-invariant, so $W_\ell$ is as well.  We have $f(x)\to\infty$ as $x$
  approaches the cusp of $\X^\sigma$ because surfaces develop tall cylinders as they approach the cusp (see Theorem~5.5 and Proposition~7.1
  of \cite{bainbridge06}).  Thus $W_\ell$ is a neighborhood of the cusp.

  If $f(X, \omega)>0$, then $\deriv{}{t}f(a_t\cdot(X, \omega))>0$.  It follows that $\ell$ is a regular value of $f$, so $\bdry W_\ell$ is
  smooth if $\ell>0$.  Also if $\ell>0$, then $\bdry W_\ell$ is transverse to $\F$.

  Since $\X$ has finite volume, the cusp of $\X^\sigma$ does as well, so $\Vol W_\ell\to 0$ as $\ell\to\infty$.
\end{proof}

Given $W$ as in Lemma~\ref{lem:addaptedneighborhood}, let $\mathcal{W}$ be the measured foliation of $\bdry W$ induced by $\F$.  Give the
leaves of $\mathcal{W}$ their arclength measure as horocycles in $\half$, and let $\nu$ be the measure on $\bdry W$
which is the product of the leafwise and transverse measures.

\begin{lemma}
  \label{lem:volumerelation}
  We have
  \begin{equation*}
    \vol\bdry W = \vol W
  \end{equation*}
  with respect to the measures $\nu$ and $\mu_D$.
\end{lemma}

\begin{proof}
  This follows directly from the fact that the length of the horocycle in $\half$ joining $i$ to $i+t$ is equal to the area of the region in
  $\half$ lying above this horocycle.
\end{proof}

\paragraph{Proof of Theorem~\ref{thm:Fclosedcurrent}.}
After Corollary~\ref{cor:Fcurrent}, it remains to show that the current defined by integration over $\F$ is closed.  Let $\eta$ be a smooth
$1$-form on $\Y$.  We need to show $\int_{\F} d\eta = 0$.

Since $\eta$ is smooth, $\eta$ and $d\eta$ have \Poincare growth, so $\|\eta\|_{\F}$ and $\|\eta\|_{\F}$ are bounded.  Given $\epsilon>0$,
let $W \subset \X$ be the neighborhood of the cusps constructed in Lemma~\ref{lem:addaptedneighborhood} with $\Vol W<\epsilon$.  By Stokes'
Theorem and Lemma~\ref{lem:volumerelation},
\begin{align*}
  \int_{\F} d\eta &= \int_{\F|_W} d\eta + \int_{\mathcal{W}} \eta \\
  &\leq \epsilon(\|\eta\|_{\F}^\infty + \|d\eta\|_{\F}^\infty) 
\end{align*}
where the last integral is over the measured foliation of $\mathcal{W}$ of $\bdry W$ defined above.  Thus $\int_{\F} d\eta=0$.
\qed

\paragraph{Cohomological interpretation of $\vol \E(1,1)$.}

Let $[\barF], [\omega_i] \in H^2(\Y; \reals)$ be the cohomology classes defined by the respective closed currents.

\begin{theorem}
  \label{thm:cohomologicalvolume}
  We have
  \begin{equation*}
    \vol \E(1,1) = 2\pi [\omega_1] \cdot [\barF],
  \end{equation*}
  where the pairing is the intersection pairing on $H^2(\Y; \reals)$.
\end{theorem}

\begin{proof}
  By Proposition~\ref{prop:omegaipoincare}, there is a smooth $1$-form $\eta$ on $\X$ with \Poincare growth such that
  \begin{equation*}
    \alpha = \omega_1 - d\eta
  \end{equation*}
  is a smooth $2$-form on $\Y$.  We have,
  \begin{align*}
    [\omega_1]\cdot[\barF] &= \int_{\F} \alpha \\
    &= \int_{\F} \omega_1 - \int_ {\F} d\eta \\ &=\int_{\F} \omega_1 \\
    &= \frac{1}{2\pi} \vol \E(1,1).
  \end{align*}
\end{proof}

\section{Intersection with closed leaves}
\label{sec:inters-with-clos}

In this section, we will relate the foliation $\F$ in a neighborhood of $\P$ and $\W$ with the canonical foliations of the tangent bundles
of these curves.  This, together with Theorem~\ref{thm:thomintegralzero} will imply that the classes $[\barP]$ and $[\barW]$ have trivial
intersection with $[\barF]$.

\paragraph{Bundles over $\Y$.}

The line bundle $\Omega\X$ over $\X$ defined in \S\ref{sec:hilb-modul-surf} extends to a line bundle $\Omega\Y$ over $\Y$  whose fiber over
$(X, [\omega])$ is the line in $\Omega(X)$ determined by $[\omega]$.  Let $Q\Y = (\Omega\Y)^2$, the line bundle whose fiber over $(X,
[\omega])$ is the space $QX$ of quadratic differentials on $X$ which are multiples of $\omega^2$.  The $L^1$-metric on $Q\Y$ is the Hermitian
metric which assigns to $q$ norm $\int_X|q|$.

\paragraph{Cotangent bundles.}

For the rest of this section, $C$ will denote either of the curves $\P$ or $\W$.  For any curve $E\subset\Y$, we let $QE$ denote the
restriction of $Q\Y$ to $E$.  We start by
sketching a proof of the following well-known fact from \Teichmuller theory:
\begin{prop}
  \label{prop:quaddiffbundle}
  There is a natural isomorphism,
  \begin{equation}
    \label{eq:41}
    Q C \to T^* C,
  \end{equation}
  which identifies the $L^1$-metric on $Q C$ with twice the hyperbolic metric on $T^*C$.
\end{prop}

\begin{proof}
  Suppose $C=\W$.  Given $(X, [\omega])\in\W$, the quadratic differential $\omega^2$ defines via \Teichmuller theory a cotangent vector
  $\tau$ to $\moduli[2]$ at $X$.  Pulling back $\tau$ by the immersion $\W\to\moduli$ defines a cotangent vector $\sigma$ to $\W$ at $(X,
  [\omega])$.  This defines the map \eqref{eq:41}.
  
  Consider the unit speed geodesic $(X_t, \omega_t) = a_t\cdot(X, \omega)$ on $\W$.  The induced \Teichmuller map $f_t\colon X \to X_t$ has
  complex dilatation,
  \begin{equation*}
    \mu_t = \frac{e^{-t/2} - e^{t/2}}{e^{-t/2} + e^{t/2}} \nu,
  \end{equation*}
  where $\nu = \overline{\omega}/\omega$.  We have
  \begin{equation*}
    \deriv{}{t}\mu_t = -\frac{1}{2}\nu,
  \end{equation*}
  so $\nu$ represents a tangent vector to $\W$ of hyperbolic norm $2$.  The pairing with $\sigma$ is given by
  \begin{equation*}
    \langle\sigma, \nu\rangle = \int \frac{\overline{\omega}}{\omega}\omega^2 = \int|\omega^2| = \|\omega^2\|,
  \end{equation*}
  so the  hyperbolic norm of $\sigma$  is half the $L^1$ norm.

  The case where $C=\P$ is identical, except that we use the immersion $\P\to\moduli[1]\times\moduli[1]$ instead of $\moduli$.
\end{proof}

\begin{lemma}
  \label{lem:quaddiffbundle}
  The natural isomorphism \eqref{eq:41}  extends to an isomorphism,
  \begin{equation*}
    Q\overline{C} \to T^*\overline{C}(D),
  \end{equation*}
  where the divisor $D$ is the sum of the cusps of $C$.
\end{lemma}

\begin{proof}
  Given a cusp $c$, we need to construct a nonzero holomorphic section of $Q\overline{C}$ on a neighborhood of $c$ whose associated section
  of $T^*\overline{C}$ has a simple pole at $c$.  Let $(X, \omega)\in C$ with $a_t\cdot(X, \omega)\to c$ at $t\to\infty$.  Let $h_r$ be a generator of the
  subgroup of the stabilizer $\Stab(X, \omega)\subset\SLtwoR$ which stabilizes $(X, \omega)$.  Define a holomorphic map $f\colon\Delta\to C$ by
    $f(w) = (X_w, \omega_w)$,
  where if $w\neq 0$, then
  \begin{equation*}
    (X_w, \omega_w) =
    \begin{pmatrix}
      1 & \dfrac{r}{2\pi}\arg w \\
      0 & \dfrac{r}{2\pi}\log\dfrac{1}{|w|}
    \end{pmatrix}
    \cdot(X, \omega),
  \end{equation*}
  and $(X_0, \omega_0)$ is the unique limiting stable Abelian differential, which is visibly nonzero.  The restriction of $f$ to some
  $\Delta_\epsilon$ is a conformal isomorphism onto some neighborhood of $c$.

  We have a nonzero holomorphic section $w\mapsto \omega_w^2$ of $Q\overline{C}$ over $\Delta_\epsilon$.  The associated holomorphic
  $1$-form $\tau$ on $\Delta^*_\epsilon$ has  hyperbolic norm,
  \begin{equation*}
    \|\tau\| = \frac{1}{2}\int_X |\omega_w|^2 \comp -\log|w|,
  \end{equation*}
  with respect to the hyperbolic metric on $\Delta^*_\epsilon$.  Since a nonzero holomorphic section of $T^* \Delta_\epsilon$ has norm
  comparable to
    $-|w|\log|w|$,
  $\tau$ has a simple pole at $0$.
\end{proof}

\paragraph{Foliated tubular neighborhoods.}

Let $U\subset\pobarm[2]$ be a tubular neighborhood of $\overline{C}$ which is small enough that each point of $U\setminus \overline{C}$
represents a surface with a unique shortest saddle connection (up to the hyperelliptic involution) connecting distinct zeros.

Define a map $ \Phi \colon U \to (Q\overline{C})^* $ as follows.  Given $(X, \omega)\in U$, let $I\subset X$ be the unique shortest saddle
connection connecting distinct zeros, and let $(Y, \eta)$ be the surface obtained by collapsing $I$.  Identify $QY$ with $QX$ by the unique
linear map $T$ sending $\eta^2$ to $\omega^2$.  Define $S\in (Q\Y)^*$ by
\begin{equation*}
  S(q) = \left(\int_I \sqrt{T(q)}\right)^2,
\end{equation*}
and define $\Phi(X, \omega) = (Y, S)$.

We observed in the proof of \cite[Theorem~12.2]{bainbridge06}:

\begin{prop}
  \label{prop:tubularneighborhoods}
  In the case where $C=\P$, the map $\Phi$ is an isomorphism between $U$ and a neighborhood of the zero section in $(Q\barP)^*$.  In the case
  where $C=\W$, there is a tubular neighborhood $V\subset U$ which is a threefold branched covering by $\Phi$ of a neighborhood of the zero section in
  $(Q\barW)^*$, branched over $\barW$.
\end{prop}

The proof is essentially immediate from the fact that given a short segment $I\subset \cx$, there are three ways to split a double zero
along $I$ and one way to perform a connected sum along $I$.

\begin{prop}
  \label{prop:foliationsconcide}
  The map $\Phi$ sends the foliation $\F$ to $2i$ times the canonical foliation of $TC$.
\end{prop}

\begin{proof}
  Given $(X, \omega)\in C$,  identify the universal cover $\widetilde{C}$ with $\half$ by $z\mapsto (X_z, \omega_z)$, where
  \begin{equation*}
    (X_z, \omega_z) =
    \begin{pmatrix}
      1 & \Re z \\
      0 & \Im z
    \end{pmatrix}
    \cdot(X, \omega).
  \end{equation*}

  By the proof of Proposition~\ref{prop:quaddiffbundle}, the Beltrami differential
  $
    \nu_z = \overline{\omega}_z/\omega_z
  $
  on $X_z$ represents the tangent at $z$ to the geodesic $\gamma(s) = e^s z$ at $s=t$, that is, the vector 
  \begin{equation*}
    v_z = - 2 i y\pderiv{}{z},
  \end{equation*}
  at $z$.
  
  Now restrict to the case where $C=\P$.  Let $U\subset\half$ be a small neighborhood of $i$.  Given $w\in\cx\setminus\{0\}$ let $I_w\in\cx$ be
  the segment joining $\sqrt{w}/2$ and $-\sqrt{w}/2$.  Choose some $\epsilon$ small enough that for every $w\in\Delta_\epsilon$ and $z\in
  U$, we can form the connected sum of $(X_z, \omega_z)$ along $I_w$.  We  parameterize a neighborhood of $(X, \omega)$ in $\Y$ by $f\colon
  U\times\Delta_\epsilon\to \Y$, where
  \begin{equation}
    \label{eq:42}
    f(z, w) = \connsum ((X_z, \omega_z), I_w).
  \end{equation}
  In these $(z, w)$-coordinates, we have
  $
    \Phi(z, w) = (X_z, S_w),
  $
  where $S_w\in (QX)^*$ is characterized by
  $
    S(\omega_z^2) = w,
  $
  which corresponds to the Beltrami differential,
  \begin{equation*}
    \frac{w}{\Im z} \nu_z,
  \end{equation*}
  which in turn corresponds to the tangent vector,
  \begin{equation*}
    -2 i w\pderiv{}{z}.
  \end{equation*}
  Thus, identifying $(Q\P)^*$ with $T\P$, we have
  \begin{equation}
    \label{eq:43}
    \frac{i}{2}\Phi(z, w) = \left(z, w\pderiv{}{z}\right).
  \end{equation}

  We have,
  \begin{equation*}
    \begin{pmatrix}
      1 & \Re z \\
      0 & \Im z
    \end{pmatrix}
    \cdot \connsum((X, \omega), I_{(u+ iv)^2}) = \connsum((X_z, \omega_z), I_{(u + zv)^2}), 
  \end{equation*}
  so in $(z, w)$ coordinates, the leaf of $\F$ through $(i, (u+iv)^2)$ is parameterized by $z\mapsto (z, (u+ zv)^2)$, which corresponds to 
  the leaf of $\mathcal{C}_{\P}$ through $\frac{i}{2}\Phi(i, w)$.

  The proof in the case where $C = \W$ is nearly identical except that we use the operation of splitting a double zero to parameterize a
  tubular neighborhood of $\W$.  Recall  that a splitting of a double zero of a differential $(X, \omega)$ is
  determined by an embedded ``X'' as defined in \S\ref{sec:abel-diff}.  Equation  \eqref{eq:42} becomes
  \begin{equation*}
    f(z, w) = \zerosplit((X_z, \omega_z), E_z(w)),
  \end{equation*}
  where $E_z$ is a holomorphic map from some small disk $\Delta_\epsilon\subset\cx$ to the space of embedded ``X''s on $(X_z, \omega_z)$
  such that for each segment $I$ of $E_z(w)$,
  \begin{equation*}
    \left(\int_I \omega_z\right)^2 = w^3.
  \end{equation*}
  Then, \eqref{eq:43} becomes
  \begin{equation*}
    \frac{i}{2}\Phi(z, w) = \left(z, w^3\pderiv{}{z}\right).
  \end{equation*}
\end{proof}

\begin{theorem}
  \label{thm:intersectionzero}
  We have,
  \begin{equation}
    \label{eq:44}
    [\P]\cdot[\barF] = [\W]\cdot[\barF]= 0.
  \end{equation}
\end{theorem}

\begin{proof}
  Let $\Psi$ be the form defined in \S\ref{sec:canonical-foliations} which represents the Thom class on $T\barP(-D)$ or $T\barW(-D)$, where
  $D$ is the divisor of cusps.  By Proposition~\ref{prop:tubularneighborhoods}, $\Phi^*\Psi$ represents either the cohomology class $[\barP]$ or
  $3[\barW]$.  Then we obtain \eqref{eq:44} from Proposition~\ref{prop:foliationsconcide} and Theorem~\ref{thm:thomintegralzero}.
\end{proof}

\section{Intersection with boundary curves}
\label{sec:inters-with-bound}

In this section, we will calculate the intersection numbers of the curves $C_P$ with the class $[\barF]$.  We will show:

\begin{theorem}
  \label{thm:intCPandF}
  For any prototype $P=(a, b, c, \overline{q})\in\Yprot$,
  \begin{equation*}
    [C_P]\cdot[\barF] =
    \begin{cases}
      \dfrac{\gcd(a, c)}{2\gcd(a, b, c)} \left( 1 -\dfrac{b}{\sqrt{D}}\right), & \text{if $a-b+c < 0$};\\
      \dfrac{\gcd(a, c)}{2\gcd(a, b, c)} \left( 1 -\dfrac{a+c}{\sqrt{D}}\right), & \text{if $a-b+c > 0$}.
    \end{cases}
  \end{equation*}
\end{theorem}

From this, we will derive Theorem~\ref{thm:volED}.

\paragraph{Definitions.}

We regard $U_P\subset \Y^c$ as having coordinates $(y_1, y_2, y_3)$ defined by \eqref{eq:23} and Theorem~\ref{thm:threecylcoordinates}.
Recall we defined a cover $\hat{U}_P$ of $U_P$ having coordinates $(y_2, y_3)$ given by \eqref{eq:24}.  Let
$\tilde{U}_P\isom\half\times\half$ be the universal cover of $U_P$ with natural coordinates $(y_2, y_3)$.  In each case, $y_i$ ranges in
either $\half$ or $\half/\mu\zed$ for some $\mu>0$.  We set
\begin{equation*}
  y_j = u_j + i v_j.
\end{equation*}

We now define several objects on $U_P$ and its covers.  We use the convention that for any object $X$ on $U_P$, such as a subset,
a foliation, or a form, $\hat{X}$ denotes the pullback of $X$ to $\hat{U}_P$, and $\tilde{X}$ denotes the pullback to $\tilde{U}_P$.

\begin{itemize}
\item Define
  \begin{equation*}
    W = \{(y_1, y_2, y_3) \in U_P : \Im y_2 = 1\}.
  \end{equation*}
  The locus $W$ limits on a closed leaf of $q_{P^+}$ in $C_{P^+}$ and should be regarded as a boundary of a tubular neighborhood of $C_P$ in
  $U_P$.  We have
  \begin{equation*}
    \hat{W}\isom \reals/\frac{a'}{\gcd(a', c')}\lambda\zed \times \reals/\frac{a'}{\gcd(a', a'+b'+c')}(\lambda-1)\zed \times \reals^+
  \end{equation*}
  with coordinates $(u_2, u_3, v_3)$.  
\item Define
  \begin{equation*}
    W^0 = \{(y_1, y_2, y_3)\in U_P : y_2 = i\}\subset W
  \end{equation*}
  We have
  \begin{equation*}
    \hat{W}^0\isom \reals/\frac{a'}{\gcd(a', a'+b'+c')}(\lambda-1)\zed \times \reals^+
  \end{equation*}
  with coordinates $(u_3, v_3)$.
\item Define
  \begin{equation*}
    W_x = \{(y_1, y_2, y_3)\in U_P : \Im y_2 = 1 \text{ and } \Im y_3 = x\}\subset W.
  \end{equation*}
  We have
  \begin{equation*}
    \hat{W}_x\isom\reals/\frac{a'}{\gcd(a', c')}\lambda\zed \times \reals/\frac{a'}{\gcd(a', a'+b'+c')}(\lambda-1)\zed 
  \end{equation*}
  with coordinates $(u_2, u_3)$.
\item The lift $\tilde{\mathcal{A}}_D$ of the foliation $\A$ to $\tilde{U}_P$ is given by
  \begin{equation*}
    y_2 + y_3 = {\rm const}.
  \end{equation*}
  We define the leaf,
  \begin{equation*}
    \tilde{\mathcal{A}}_D(r) = \{(y_2, y_3)\in\half\times\half : y_2 + y_3 = r\}.
  \end{equation*}
\item Let $\nu_P$ be a $2$-form on $\Y^c$ which is \Poincare dual to $C_P$, and whose support is compact and disjoint from $C_P$.
\item Let $p\colon U_P \to W$ be the canonical projection, defined by
  \begin{equation*}
    \pi(X, \omega) = a_t \cdot(X, \omega),
  \end{equation*}
  where $t$ is the unique real number so that $a_t \cdot(X, \omega)\in W$.
\item Let $\eta_P = p_*\nu_P$, a $1$-form on $W$.
\item Let $A(r)$ be the area of any surface $(X, \omega)\in\tilde{\mathcal{A}}_D(r)$.
\end{itemize}

\paragraph{Induced foliation of $W$.}

We now study the intersection of $W$ with $\F$ and the induced measured foliation of $W$.

\begin{prop}
  \label{prop:Wistransverse}
  $W$ is transverse to $\F$.
\end{prop}

\begin{proof}
  At $p\in W$, the tangent space $T_p\Y$ is spanned by $T_p W$ and the tangent to the curve $a_t \cdot p$.
\end{proof}

Let $\mathcal{W}$ be the induced measured foliation of $W$.

\begin{prop}
  \label{prop:leavesofW}
  Each leaf of $\mathcal{W}$ is contained in some $W_x$.  The induced foliation $\mathcal{W}_x$ of $W_x$ is the foliation by lines of slope
  $x$.  More precisely, $\mathcal{W}_x$ is generated by the vector field
  \begin{equation*}
    \pderiv{}{u_2} + x \pderiv{}{u_3}.
  \end{equation*}
  The cylinder $W^0\subset W$ is transverse to $\mathcal{W}_x$.
\end{prop}

\begin{proof}
  The leaf of $\tilde{W}$ through $p=(u_2+i, u_3 + i v_3)\in \tilde{W}$ is $\{h_t\cdot p\}$.  So the tangent vector to $\mathcal{W}$ at $p$,
  \begin{equation*}
    \deriv{}{t} h_t \cdot p = \pderiv{}{u_2} + v_3 \pderiv{}{v_3},
  \end{equation*}
  is the required tangent to $\tilde{W}_{v_3}$.  The transversality statement is then clear.
\end{proof}

Recall that we have local coordinates $(u_3, v_3)$ on $W^0$.  

\begin{prop}
  \label{prop:transversemeasuretoW0}
  The measure on $W^0$ induced by the transverse measure to $\mathcal{W}$ is
  \begin{equation}
    \label{eq:45}
    \frac{a}{\sqrt{D}}\frac{1}{(v_3 + 1)^3} du_3 dv_3
  \end{equation}
\end{prop}

\begin{proof}
  Recall that we defined in \S\ref{sec:hilb-modul-surf} a leafwise measure on $\A$ which is invariant under the holonomy of $\F$.  The leaf
  $\tilde{\mathcal{A}}_D(i)$ is parameterized by the coordinates $(u_3, v_3)$, and in these coordinates, the measure on
  $\tilde{\mathcal{A}}_D(i)$ is
  \begin{equation*}
    \mu = \frac{1}{\sqrt{A(1)}} du_3 dv_3
  \end{equation*}
  (the constant factor is to normalize the surfaces parameterized by this leaf to have unit area).

  Let $\Hol\colon \tilde{W}^0 \to \tilde{\mathcal{A}}_D(i)$ be the holonomy map for $\tilde{F}_D$.  We have
  \begin{equation*}
    \begin{pmatrix}
      1 & -\dfrac{u_3}{v_3+1} \\
      0 & \phantom{-}\dfrac{1}{v_3+1}
    \end{pmatrix}
    \cdot(i, u_3 + i v_3) = \left(\frac{-u_3 + i}{v_3+1}, \frac{u_3+iv_3}{v_3+1}\right) \in\tilde{\mathcal{A}}_D(i),
  \end{equation*}
  so
  \begin{equation*}
    \Hol(u_3 + iv_3) = \frac{u_3 + iv_3}{v_3 + 1}
  \end{equation*}
  in the coordinates $(u_3, v_3)$ on the domain and range.
  Thus the induced measure on $W^0$ is
  \begin{equation*}
    \frac{1}{A(1)}|D\Hol| du_3dv_3 =     \frac{1}{A(1)(v_3+1)^3} du_3dv_3, 
  \end{equation*}
  which is equal to \eqref{eq:45} by the following Lemma.
\end{proof}

\begin{lemma}
  \label{lem:area}
  We have
  \begin{equation}
    \label{eq:46}
    A(r) = \frac{\sqrt{D}}{a} \Im r. 
  \end{equation}
\end{lemma}

\begin{proof}
  Consider $(X_p, \omega_p)\in\tilde{\mathcal{A}}_D(r)$ with $v_2+v_3= \Im r$.  We have
  \begin{equation}
    \label{eq:47}
    \Area(X_p, \omega_p) = v_1 + \lambda v_2 + (\lambda-1) v_3.
  \end{equation}
  Using \eqref{eq:24} and the equations,
  \begin{equation*}
    \lambda- \frac{c}{a\lambda} = \lambda-1 - \frac{a+b+c}{a(\lambda-1)} = \frac{\sqrt{D}}{a},
  \end{equation*}
  which follow from elementary algebra we see that \eqref{eq:47} reduces to \eqref{eq:46}.
\end{proof}

\paragraph{Integral of $\nu_P$ over $U_P$.}

For any $t\in\ratls\cdot (\lambda-1)/\lambda$, the foliation $\hat{\mathcal{W}}_t$ of $\hat{W}_t$ consists of closed leaves.  Let $\delta_t$ be
one of these closed leaves, and let $\Delta_t$ be the leaf of $\hat{\mathcal{F}}_D$ containing $\delta_t$, a punctured disk.

\begin{prop}
  \label{prop:integraloverdelta}
  Given
  \begin{equation}
    \label{eq:48}
    t = \frac{p}{q} \frac{\gcd(a', c')}{\gcd(a', a'+b'+c')} \frac{\lambda-1}{\lambda}
  \end{equation}
  with $\gcd(p, q) =1$, we have
  \begin{align}
    \label{eq:49}
    \int_{\delta_t} \hat{\eta}_P &= q, \quad\text{and} \\
    \label{eq:50}
    \int_{\delta_t} \hat{\eta}_{P^+} &= p.
  \end{align}
\end{prop}

\begin{proof}
  Define
  \begin{align*}
    u &= e^{2\pi i (\gcd(a', c')/a'\lambda)y_2}\quad\text{and} \\
    v &= e^{2\pi i (\gcd(a', a'+b'+c')/a'(\lambda-1))y_3}.
  \end{align*}
  These $(u, v)$ are the same coordinates as in Theorem~\ref{thm:typetwocoordinates} on the polydisk $\Delta^2$ which parameterizes a
  neighborhood of $c_P$ in $\Y^c$ by $f\colon\Delta^2\to\Y^c$.  Define
  \begin{align*}
    \hat{C}_P &= f^{-1}(C_P) = \{(u, v)\in\Delta^2\colon u=0\} \\
    \hat{C}_{P^+} &= f^{-1}(C_{P^+}) = \{(u, v)\in\Delta^2\colon v=0\}.
  \end{align*}

  We have
  \begin{align*}
    \int_{\delta_t} \hat{\eta}_P &= \int_{\Delta_t}\hat{\nu}_P = \overline{\Delta}_t\cdot\hat{C}_P\quad\text{and} \\
    \int_{\delta_t} \hat{\eta}_{P^+} &= \int_{\Delta_t}\hat{\nu}_{P^+} = \overline{\Delta}_t\cdot\hat{C}_{P^+},
  \end{align*}
  where the last expressions are the local intersection numbers of these curves at $(0,0)\in\Delta^2$.  In $(u, v)$-coordinates,
  \begin{equation*}
    \hat{W}_t = \{(u, v)\in\Delta^2 : |u|=r \text{ and } |v| = s\},
  \end{equation*}
  for some $0<r, s<1$.  Let $S\subset\Delta^2$ be the sphere of radius $\sqrt{r^2+s^2}$, which contains $\hat{W}_t$.  We have
  \begin{equation*}
    \overline{\Delta}_t\cdot\hat{C}_P = \Link(\delta_t, \zeta)
  \end{equation*}
  where $\Link$ denotes the linking number, and $\zeta=S\cap\hat{C}_P$.  In $S$, the curve $\zeta$ bounds a disk which intersects
  $\hat{W}_t$ in the curve $\xi= \{u_2 = 0\}$.  Thus
  \begin{equation*}
    \Link(\delta_t, \zeta) = \delta_t\cdot\xi = q
  \end{equation*}
  (where the last expression is the intersection number in $\hat{W}_t$) because $\delta_t$ is a geodesic on $\hat{W}_t$ with its flat metric
  of slope $p/q$.  Thus we obtain \eqref{eq:49} and \eqref{eq:50} follows in the same way.  
\end{proof}

\begin{prop}
  \label{prop:intoverUP}
  We have
  \begin{align}
    \label{eq:51}
    \int_{U_P}\nu_P &= \frac{1}{2}\frac{a}{\sqrt{D}}\frac{\gcd(a,c)}{\gcd(a, b, c)}(\lambda-1), \quad\text{and} \\
    \label{eq:52}
    \int_{U_P}\nu_{P^+} &= \frac{1}{2}\frac{a}{\sqrt{D}}\frac{\gcd(a, a+b+c)}{\gcd(a, b, c)}\lambda.
  \end{align}
\end{prop}

\begin{proof}
  Give $\hat{\mathcal{W}}_t$ the transverse measure $du_3$.  With $t$ as in \eqref{eq:48}, a segment transverse to $\hat{\mathcal{W}}_t$ of
  transverse measure
  \begin{equation*}
    \frac{1}{q}\frac{a'}{\gcd(a', a'+b'+c')}(\lambda-1)
  \end{equation*}
  meets each leaf of $\hat{\mathcal{W}}_t$ once, so by Proposition~\ref{prop:integraloverdelta},
  \begin{align*}
    \int_{\hat{\mathcal{W}}_t}\hat{\eta}_P &= \frac{a'}{\gcd(a', a'+b'+c')}(\lambda-1),\quad\text{and} \\
    \int_{\hat{\mathcal{W}}_t}\hat{\eta}_{P^+} &= \frac{p}{q}\frac{a'}{\gcd(a', a'+b'+c')}(\lambda-1) \\
    &= t\frac{a'}{\gcd(a', c')}\lambda.
  \end{align*}
  By continuity, these formulas hold for all $t\in\reals^+$.
  By the definition of $\eta_P$, we have
  \begin{equation*}
    \int_{\F|U_P} \nu_P = \int_{\mathcal{W}} \eta_P,
  \end{equation*}
  and similarly for $\nu_{P^+}$.  Finally, from Fubini's Theorem, Proposition~\ref{prop:transversemeasuretoW0} and the fact that $\hat{U}_P$
  is an $a'/\gcd(a', c')\gcd(a', a'+b'+c')$-fold cover of $U_P$, we obtain
  \begin{align*}
    \int_{\F|U_P} \nu_P &= \frac{\gcd(a', c')\gcd(a', a'+b'+c')}{a'}\int_{\hat{\mathcal{W}}}\hat{\eta}_P \\
    &=\frac{a}{\sqrt{D}} \gcd(a', c') (\lambda-1) \int_0^\infty \frac{dv_3}{(1+v_3)^3} \\
    &=\frac{a}{2\sqrt{D}} \gcd(a', c') (\lambda-1), \\
    \intertext{and similarly}
    \int_{U_P}\nu_{P^+} &= \frac{a}{2\sqrt{D}}\gcd(a', a'+b'+c')\lambda.
  \end{align*}
\end{proof}

\paragraph{Proof of Theorem~\ref{thm:intCPandF}.}

Since $\nu_P$ is supported on the locus of periodic surfaces, and the locus of two-cylinder surfaces has measure zero, we have
\begin{equation}
  \label{eq:53}
  [C_P]\cdot[\barF] = \int_{\F|U_P}\nu_P + \int_{\F|U_{P^-}}\nu_P.
\end{equation}
There are two cases to consider.  When $a-b+c<0$, then by the definition of $P^-$ and by Proposition~\ref{prop:intoverUP}, \eqref{eq:53}
becomes 
\begin{equation*}
  [C_P]\cdot[\barF] = \frac{1}{2}\left(\frac{a}{\sqrt{D}}(\lambda-1) + \frac{a}{\sqrt{D}}(\lambda+1)\right)\frac{\gcd(a, c)}{\gcd(a, b, c)}.
\end{equation*}
When $a-b+c>0$, \eqref{eq:53} becomes
\begin{equation*}
   [C_P]\cdot[\barF] = \frac{1}{2}\left(\frac{a}{\sqrt{D}}(\lambda-1) -\frac{c}{\sqrt{D}}\frac{\lambda+1}{\lambda}\right)\frac{\gcd(a,
     c)}{\gcd(a, b, c)}. 
\end{equation*}
Both formulas reduce to the desired expression by elementary algebra.
\qed

\begin{cor}
  \label{cor:CPdotPD}
  For every prototype $P\in\Yprot$,
  \begin{equation*}
    [C_P]\cdot([\barF]+\tau^*[\barF]) = [C_P]\cdot[\barP].
  \end{equation*}
\end{cor}

\begin{proof}
  From Theorem~\ref{thm:intCPandF}, and from the definition of the involution $t$ on $\Yprot$, we have
  \begin{equation*}
    [C_P]\cdot[\barF] +[C_{t(P)}]\cdot[\barF] = \frac{\gcd(a,c)}{\gcd(a, b, c)}.
  \end{equation*}
  By Theorem~\ref{thm:propertiesofYD}, this is just $[C_P]\cdot[\barP]$.
\end{proof}

\paragraph{Volume of $\E(1,1)$.}

We are now ready to calculate the volume of $\E(1,1)$ with respect to $\mu_D$.  In \cite{bainbridge06}, we showed:

\begin{theorem}
  \label{thm:fundamentalclasses}
  The fundamental classes of $\barW$ and $\barP$ in $H^2(\Y; \ratls)$ are given by
  \begin{align*}
    [\barW] &= \frac{3}{2}[\omega_1] + \frac{9}{2}[\omega_2] + B_D, \quad\text{and} \\
    [\barP] &= \frac{5}{2}[\omega_1] + \frac{5}{2}[\omega_2] + B_D, 
  \end{align*}
  where $B_D\in H^2(\Y; \ratls)$ is a linear combination of the fundamental classes of the curves $C_P$.
\end{theorem}

We also showed in \cite[Lemma~13.1]{bainbridge06}:

\begin{lemma}
  \label{lem:selfintBD}
  The self-intersection number of $B_D$ is
  \begin{equation*}
    B_D\cdot B_D = -15 \chi(\X) .
  \end{equation*}
\end{lemma}

\begin{proof}
  This can be seen easily from Theorem~\ref{thm:fundamentalclasses} using the fact that $$[\barW]\cdot[\barP]=0$$ because these curves are
  disjoint by Theorem~\ref{thm:propertiesofYD}.
\end{proof}

\begin{theorem}
  \label{thm:volXD}
  The volume of $\E(1,1)$ with respect to $\mu_D$ is
  \begin{equation*}
    \Vol(\E(1,1)) = 4 \pi \chi(\X).
  \end{equation*}
\end{theorem}

\begin{proof}
  By Theorem~\ref{thm:fundamentalclasses},
  \begin{equation*}
    [\omega_1] = -\frac{1}{3} [\barW] + \frac{3}{5}[\barP] - \frac{4}{15} B_D.
  \end{equation*}
  Then by Theorem~\ref{thm:intersectionzero},
  \begin{equation*}
    [\omega_1]\cdot[\barF] = -\frac{4}{15} B_D\cdot[\barF].
  \end{equation*}

  By Proposition~\ref{prop:taufixesP}
  and Theorem~\ref{thm:fundamentalclasses}, $\tau^* B_D = B_D$.  We then have
  \begin{alignat*}{2}
    B_D \cdot [\barF] &= \frac{1}{2}(B_D + \tau^* B_D)\cdot[\barF]\\
    &= \frac{1}{2}B_D \cdot([\barF]+ \tau^*[\barF])\\
    &= \frac{1}{2} B_D \cdot [\P] & &\quad\text{(by Corollary~\ref{cor:CPdotPD})}\\
    &= \frac{1}{2} B_D\cdot B_D & &\quad\text{(by Theorem~\ref{thm:fundamentalclasses})}\\
    &= -\frac{15}{2}\chi(\X).& &\quad\text{(by Lemma~\ref{lem:selfintBD})}
  \end{alignat*}
  Thus $[\omega_1]\cdot[\barF] = 2 \chi(\X)$, and the claim follows by Theorem~\ref{thm:cohomologicalvolume}.
\end{proof}

\section{Counting functions}
\label{sec:siegelveech}

In this section, we calculate the integrals over $\Omega_1\E(1,1)$ of the counting functions $N_{\rm c}(T, L)$ and $N_{\rm s}^i(T, L)$
defined in \S\ref{sec:introduction}, proving Theorem~\ref{thm:svintegral}.

\paragraph{Integral of $N_{\rm c}(T, L)$.}

Given a prototype $P=(a, b, c, \overline{q})$, define
\begin{equation*}
  v(P) = a' \lambda(\lambda-1)\left(1 + \frac{1}{\lambda^2} + \frac{1}{(\lambda-1)^2}\right). 
\end{equation*}

\begin{theorem}
  \label{thm:svED2}
  For any $\epsilon>0$,
  \begin{equation*}
    \int\limits_{\Omega_1\E(1,1)} N_c(T, \epsilon)\,d\mu_D^1(T) = \epsilon^2\sum_{P\in\Yprot}v(P).
  \end{equation*}
\end{theorem}

\begin{proof}
  Since $N_c(T, \epsilon)$ is rotation-invariant, we can regard $N_c(T, \epsilon)$ as a function on $\X$ or $\X^c$, where as a function on
  $\X^c$, $N_c((X, \omega), \epsilon)$ is the number of cylinders in the horizontal foliation of $(X, \omega)$ which have circumference at
  most $\epsilon$, with $(X, \omega)$ normalized to have unit area.  We then have
  \begin{align*}
    \int\limits_{\Omega_1\E(1,1)} N_c(T, \epsilon)\,d\mu_D^1(T) &= \sum_{c\in C(\X)} \int_{\X^c}N_c(T, \epsilon)\, d\mu_D(T) \\
    &= \sum_{P\in\Yprot} \int_{U_P} N_c(T, \epsilon)\, d\mu_D(T).
  \end{align*}
   We will show that
  \begin{equation}
    \label{eq:54}
    \int_{U_P} N_c(T, \epsilon)\, d\mu_D(T) = \epsilon^2 v(P).
  \end{equation}

  We continue to use the notation of \S\ref{sec:inters-with-bound}, in particular the locus $W\subset U_P$ and the coordinates
  $y_j=u_j+i v_j$ on $U_P$.  Given $p=(X, \omega)\in W$ and
  $L\in\reals^+$, define
  \begin{equation*}
    B_L = \{a_s h_t p : s\in\reals \text{ and } 0\leq t\leq L\},
  \end{equation*}
  a subset of a leaf of $\F$ which we regard as a copy of $\half$ with its hyperbolic area measure $\rho$.  We identify the point
  $a_th_s p$ of this leaf with $e^{t}i+s\in\half$.

  We claim that
  \begin{equation}
    \label{eq:55}
    \int_{B_L} N_c(T, \epsilon) \, d\rho(T) = \epsilon^2 \frac{\sqrt{D}}{a}(1+v_3)\left(1+\frac{1}{\lambda^2} + \frac{1}{(\lambda-1)^2}\right) L.
  \end{equation}
  The form $(X, \omega)$, normalized to have unit area has three horizontal cylinders of circumference $1/\sqrt{A}$, $\lambda/\sqrt{A}$, and
  $(\lambda-1)/\sqrt{A}$, where
  \begin{equation*}
    A= \frac{\sqrt{D}}{a}(1+v_3)
  \end{equation*}
  by Lemma~\ref{lem:area}.  A cylinder of circumference $\ell$ on $(X, \omega)$ has circumference less than $\epsilon$ on $a_t h_s (X,
  \omega)$ if and only if
  \begin{equation*}
    e^{t} > \frac{\ell^2}{\epsilon^2 A},
  \end{equation*}
  so the contribution of this cylinder to \eqref{eq:55} is
  \begin{equation*}
    \int_0^L \int_{\ell^2/\epsilon^2A}^\infty \frac{1}{y^2}\,dy\,dx = \frac{\epsilon^2 A L}{\ell^2},
  \end{equation*}
  which yields \eqref{eq:55} by summing over all cylinders.

  Now, with $p\colon U_P\to W$ the canonical projection along $a_s$-orbits,
  \begin{equation*}
    \sigma=p_*(N_c(\cdot, \epsilon)\rho)
  \end{equation*}
  is a leafwise measure for the measured foliation $\mathcal{W}$.  If we parameterize the leaves of $\mathcal{W}$ as unit-speed horocycles,
  we obtain the leafwise measure $du_2$, so by \eqref{eq:55},
  \begin{equation*}
    \sigma = \epsilon^2 \frac{\sqrt{D}}{a}(1+v_3)\left(1+\frac{1}{\lambda^2} + \frac{1}{(\lambda-1)^2}\right) du_2.
  \end{equation*}
  Let $\tau$ be the transverse measure to $\mathcal{W}$, given by Proposition~\ref{prop:transversemeasuretoW0}.  We have
  \begin{align*}
    \int_{U_P}N_c(T, \epsilon)\,d\mu_D(T) &= \int_W d\sigma\,d\tau \\
    &=\int_0^\infty\int_{W_{v_3}}\epsilon^2\left(1+\frac{1}{\lambda^2} +
      \frac{1}{(\lambda-1)^2}\right)\frac{1}{(1+v_3)^2}\,du_2\,du_3\,dv_3 \\
    &= \epsilon^2 a'\lambda(\lambda-1)\left(1+\frac{1}{\lambda^2} + \frac{1}{(\lambda-1)^2}\right) \int_0^\infty\frac{1}{(1+v_3)^2}dv_3 \\
    &= \epsilon^2 v(P),
  \end{align*}
  which proves \eqref{eq:54}.
\end{proof}

\paragraph{Sums over prototypes.}

We now evaluate the sum over prototypes appearing in Theorem~\ref{thm:svED2}.  Define an involution $s\colon\Yprot\to\Yprot$ by
\begin{equation*}
  s(P) = t(P^+),
\end{equation*}
or equivalently
\begin{equation*}
  s(a, b, c, \overline{q}) = (a, -2a-b, a+b+c, \overline{q}).
\end{equation*}
We also make the following definitions:
\begin{align*}
  &v'(P) = a \lambda(\lambda-1)\left(1+\frac{1}{\lambda^2}+\frac{1}{(\lambda-1)^2}\right),\\
  &w(P) = v'(P) + v'(s(P)),\\
  &S_D = \{(a, b, c)\in\zed^3 : b^2-4ac=D, a>0, c<0, \text{ and }a+b+c<0 \},\quad\text{and} \\
  &S_D'= \{(a, b, c)\in\zed^3 : b^2-4ac=D, a>0, \text{ and } c<0\}.
\end{align*}
Since the definitions of $v$, $v'$, $w$, and $s$ make sense on elements of $S_D$, we regard them to be defined there as well.

\begin{theorem}
  \label{thm:sumw}
  For any quadratic discriminant $D$,
  \begin{equation}
    \label{eq:56}
    \sum_{P\in \Yprot} v(P) = 60 \chi(\X).
  \end{equation}
\end{theorem}

\begin{lemma}
  \label{lem:sumvPequalssumvprime}
  Given a quadratic discriminant $D = f^2 E$ with $E$ fundamental and $f\in\nats$, we have
  \begin{equation}
    \label{eq:57}
    \sum_{s|f}\sum_{P\in\Yprot[s^2E]} v(P) = \sum_{P\in S_D} v'(P).
  \end{equation}
\end{lemma}

\begin{proof}
  Given $P=(a, b, c)\in S_D$, note that for any discriminant $D/s^2$ prototype $Q=(a/s, b/s, c/s, \overline{q})$, we have $v(P)=v(Q)$.  We
  need to show that the contribution to the left hand side of \eqref{eq:57} from prototypes $Q$ of this form is $v'(P)$.  If $s|\gcd(a, b,
  c)$, then there are $\phi(\gcd(a, b, c)/s)$ prototypes in $\Yprot[D/s^2]$ of the form $(a/s, b/s, c/s, \overline{q})$ (where $\phi$ is the Euler
  function).  The contribution to the left hand side from $P=(a, b, c)$ is
  \begin{equation*}
    v(P)\sum_{s|\gcd(a, b, c)}\phi\left(\frac{\gcd(a, b, c)}{s}\right) = v(P)\gcd(a, b, c) = v'(P).
  \end{equation*}
\end{proof}

The following lemma is just a computation in elementary algebra and will be left to the reader.

\begin{lemma}
  \label{lem:wformula}
  For any $P=(a, b, c)\in S_D$,
  \begin{equation*}
    w(P) = 4a + b + \frac{b^2}{a} + \frac{ab}{c} - 2c - \frac{a(2a+b)}{a+b+c}.
  \end{equation*}
\end{lemma}

\begin{lemma}
  \label{lem:identities}
  {\allowdisplaybreaks
    The following identities hold:
    \begin{align}
      \label{eq:58}
      &\sum(a+b)=0 \\
      \label{eq:59}
      &\sum\frac{ab}{c} = -\sum\frac{a(2a+b)}{a+b+c} \\
      \label{eq:60}
      &\sum\frac{2bc}{a}+\frac{b^2}{a} - a + 2c =0 \\
      \label{eq:61}
      &\sum\frac{ab}{c} = \sum\frac{bc}{a} \\
      \label{eq:62}
      &\sum(a-c) = -5 H(2, D).
    \end{align}
    All sums are over $(a, b, c)\in S_D$.
  }
\end{lemma}

\begin{proof}
  Applying the involution $s$, we have
  \begin{equation*}
    \sum a+b = \sum a + (-2a-b) = -\sum a+b, 
  \end{equation*}
  which proves \eqref{eq:58}.

  By applying $s$ to $\sum ab/c$, we obtain \eqref{eq:59} immediately.

  Applying the involution $s$ and \eqref{eq:58} yields
  \begin{align*}
    \sum\frac{bc}{a} &= -\sum 2a+3b+2c+\frac{b^2}{a} + \frac{bc}{a} \\
    &= \sum a - 2c - \frac{b^2}{a} -\frac{bc}{a},
  \end{align*}
  which is \eqref{eq:60}.

  The sum,
  \begin{equation*}
   \sum \frac{ab}{c}-\frac{bc}{a} = b\sum\frac{a}{c}-\frac{c}{a},
  \end{equation*}
  is invariant under the involution,
  \begin{equation*}
    \sigma(a, b, c) = (-c, -b, -a), 
  \end{equation*}
  of $S_D'$.  Since $S_D'$ is the disjoint union $S_D\cup\sigma(S_D)$, this implies
  \begin{equation*}
    \sum_{(a, b, c)\in S_D} \frac{ab}{c}-\frac{bc}{a} = \frac{1}{2} \sum_{(a, b, c)\in S_D'}b\left(\frac{a}{c}-\frac{c}{a}\right) = 0,
  \end{equation*}
  where the second sum vanishes because the elements $(a, b, c)$ and $(a, -b, c)$ contribute opposite values to this sum.  This implies
  \eqref{eq:61}.

  Since the summand in \eqref{eq:62} is also $\sigma$-invariant, we have
  \begin{alignat*}{2}
    \sum_{P\in S_D} a-c &= \frac{1}{2} \sum_{P\in S_D'} a-c & &(\text{because the sum is $\sigma$-invariant})\\
    &= \sum_{P\in S_D'} a & & \text{(applying $\sigma$ again)}\\
    &=  \sum_{e\equiv D\,(2)} \sigma_1\left(\frac{D-e^2}{4}\right) \\
    &= -5 H(2, D), & &\text{(by Theorem~\ref{thm:h2dsum})}
  \end{alignat*}
  which proves \eqref{eq:62}.
\end{proof}

\paragraph{Proof of Theorem~\ref{thm:sumw}.}

We have
\begin{alignat*}{2}
  \sum w(P) &= \sum 3a + \frac{b^2}{a} + \frac{2ab}{c} -2c & & \quad \text{(Lemma~\ref{lem:wformula}, \eqref{eq:58}, and \eqref{eq:59})} \\
  &= \sum 4a - 4c - \frac{2bc}{a} + \frac{2ab}{c} & & \quad\text{(by \eqref{eq:60})} \\
  &= \sum 4a - 4c & & \quad\text{(by \eqref{eq:61})} \\
  &= -20 H(2, D) & &\quad \text{(by \eqref{eq:62})},
\end{alignat*}
with all sums over $(a,b,c)\in S_D$.  Thus
\begin{equation*}
  \sum v'(P) = \frac{1}{2}\sum w(P) = -10 H(2, D),
\end{equation*}
and we obtain \eqref{eq:56} from Lemma~\ref{lem:sumvPequalssumvprime}, Theorem~\ref{thm:chiX}, \eqref{eq:18}, and \Mobius inversion.
\qed

This completes the proof of \eqref{eq:2}.

\paragraph{Integral of $N_{\rm s}^i(T, L)$.}

We now compute the integrals of the counting functions $N_{\rm s}^i(T, L)$ over $\Omega_1\E(1,1)$, using the operation of collapsing a
saddle connection to relate this to the volumes of $\Omega_1\P$ and $\Omega_1\W$.

Define $S_i$ to be the space of pairs $(T, I)$, where $T\in\Omega_1\E(1,1)$ and $I$ is a multiplicity $i$ saddle connection on $T$ joining
distinct zeros.  Let $p_i\colon S_i\to\Omega_1\E(1,1)$ be the natural local homeomorphism forgetting the saddle connection.  We equip $S_i$
with the measure $\sigma_i= p^*\mu_D^1$ which is determined by $\sigma_i(U) = \mu_D^1(U)$ if $p_i$ is injective on $U$.  Given
$\epsilon>0$, let $S_i(\epsilon)$ be the locus of saddle connections of length less than $\epsilon$.

\begin{lemma}
  \label{lem:vol1}
  For every $\epsilon>0$,
  \begin{equation}
    \label{eq:63}
    \Vol_{\sigma_i}S_i(\epsilon) = \int\limits_{\Omega_1\E(1,1)}N_{\rm s}^i(T, \epsilon)\,d\mu_D^1(T).
  \end{equation}
\end{lemma}

\begin{proof}
  Every point $T$ in $\Omega_1\E(1,1)$ has exactly $N_{\rm s}^i(T, \epsilon)$ preimages under $p_i$.
\end{proof}

Recall that the operation of collapsing a saddle connection is defined if it is unobstructed in the sense of \S\ref{sec:abel-diff}.

\begin{lemma}
  \label{lem:unobstructed}
  Almost every saddle connection in $S_2$ is unobstructed.
\end{lemma}

\begin{proof}
  It suffices to show that almost every surface in $\Omega_1\E(1,1)$ has no obstructed saddle connection, which in turn would follow from
  the same statement for $\Omega\E(1,1)$.  If a saddle connection $I$ on $(X, \omega)$ is obstructed, then from the definition of an obstruction,
  there must necessarily be an absolute period of $\omega$ which is a real multiple of the relative period $\omega(I)$.
  
  Let $R\subset\Omega_1\E(1,1)$ be the locus of surfaces which have a relative period and an absolute period which are real multiples of
  each other.  It suffices to show that $R$ has measure zero.  Consider $U\subset\Omega\moduli$ with period coordinates $U\to\cx^5$ defined
  in \S\ref{sec:abel-diff}.  In these coordinates, $\Omega\E(1,1)$ is a codimension two linear subspace.  The locus $R$ is a countable union
  of real-linear subspaces.  Since we can vary any relative period while keeping the absolute periods constant by moving along the kernel
  foliation, each of these subspaces is proper, thus $R$ has measure zero.
\end{proof}

\begin{lemma}
  \label{lem:vol2}
  For any $\epsilon>0$, we have
  \begin{align*}
    \Vol_{\sigma_1}S_1(\epsilon) &= \frac{\pi\epsilon^2}{2} \Vol \Omega_1\P, \quad\text{and} \\
    \Vol_{\sigma_2}S_2(\epsilon) &= \frac{3\pi\epsilon^2}{2}\Vol\Omega_1\W.
  \end{align*}
\end{lemma}

\begin{proof}
  Let $q_1\colon S_1(\epsilon)\to\Omega_1\P\times(\Delta_\epsilon/\pm1)$ and $q_2\colon S_2(\epsilon)\to\Omega_1\W\times(\Delta_\epsilon/\pm
  1)$ be defined by
  \begin{equation*}
    q_i((X, \omega), I) = \left(\Collapse((X, \omega),I), \int_I\omega\right).
  \end{equation*}
  The map $q_2$ is really only defined almost everywhere, on the set of unobstructed saddle connections (using
  Lemma~\ref{lem:unobstructed}).  The map $q_1$ is an isomorphism of measure spaces, since the connected sum operation provides a
  measurable inverse.  The map $q_2$ is a local isomorphism of measure spaces and almost everywhere three-to-one, since the operation of
  splitting a zero provides a local inverse, and there are three ways to split a zero along almost every segment $\overline{0w}\subset\cx$.
  Therefore,
  \begin{align*}
    \Vol_{\sigma_1}S_1(\epsilon) &= \Vol(\Omega_1\P \times (\Delta_\epsilon/\pm 1)) = \frac{\pi\epsilon^2}{2} \Vol \Omega_1\P,
    \quad\text{and} \\
    \Vol_{\sigma_2}S_2(\epsilon) &=  3\Vol(\Omega_1\W \times (\Delta_\epsilon/\pm 1)) = \frac{3\pi\epsilon^2}{2}\Vol\Omega_1\W.
  \end{align*}  
\end{proof}

\begin{cor}
  \label{cor:intsad}
  For every $\epsilon>0$, we have
  \begin{align*}
    \int\limits_{\Omega_1\E(1,1)}N_{\rm s}^1(T, \epsilon)\,d\mu_D^1(T) &= \frac{27}{2}\pi^2\epsilon^2\chi(\X), \quad\text{and} \\
    \int\limits_{\Omega_1\E(1,1)}N_{\rm s}^2(T, \epsilon)\,d\mu_D^1(T) &= \frac{5}{2}\pi^2\epsilon^2\chi(\X).
  \end{align*}
\end{cor}

\begin{proof}
  This follows from Lemmas~\ref{lem:vol1} and \ref{lem:vol2} together with
  \begin{align*}
    \Vol \Omega_1\P &= -2\pi \chi(\P) = 5 \pi\chi(\X) \\
    \Vol \Omega_1\W &= -2\pi\chi(\W) = 9 \pi\chi(\X),
  \end{align*}
  which we proved in \cite{bainbridge06}.
\end{proof}

This completes the proof of Theorem~\ref{thm:svintegral}.

\section{Equidistribution of large circles}
\label{sec:equid-large-circl}

In this section, we prove Theorem~\ref{thm:uniformdistribution}, that circles in $\Omega_1\E(1,1)\setminus\Omega_1 D_{10}$ become
equidistributed as their radius goes to infinity.

Let $\cE[D]$ be the two fold covering of $\Omega_1\E(1,1)$ whose points are eigenforms $(X, \omega)\in \Omega_1\E(1,1)$ together with a choice
of ordering of the zeros of $\omega$.  In $\cE[5]$, let $\cD$ be the inverse image of the decagon curve $\Omega_1 D_{10}\subset
\Omega_1\E(1,1)$.  The action of $\SLtwoR$ on $\Omega_1\E(1,1)$ lifts to an action of $\SLtwoR$ on $\cE[D]$, and there is a natural ergodic,
finite, $\SLtwoR$-invariant measure $\nu_D$ on $\cE[D]$ obtained by pulling back the measure $\mu_D^1$ on $\Omega_1\E(1,1)$.

Given $x\in\cE[D]$, let $m_x$ be the uniform measure on $\SOtwoR \cdot x$ and let $m_x^t= (a_t)_* m_x$.
Since $\nu_D$ projects to $\mu_D^1$, Theorem~\ref{thm:uniformdistribution} follows from the following theorem:

\begin{theorem}
  \label{thm:uniformdistribution2}
  For any $x\in \cE[D]$ for $D\neq 5$ or $x\in \cE \setminus \cD$,
  \begin{equation*}
    \lim_{t\to\infty} m_x^t = \frac{\nu_D}{\Vol(\nu_D)}.
  \end{equation*}
\end{theorem}

\begin{lemma}
  \label{lem:sufficestoshow}
  To prove Theorem~\ref{thm:uniformdistribution2}, it suffices to show that for any subsequential limit $m_x^\infty = \lim m_x^{t_n}$,
  the measure $m_x^\infty$ assigns zero mass to $\cD$.  In particular, Theorem~\ref{thm:uniformdistribution2} is true for any $x\in\cE[D]$
  for $D\neq 5$.
\end{lemma}

\begin{proof}
  Let $\cE[D]^h\subset\cE[D]$ be the locus of surfaces which have a horizontal saddle connection.  Points in $\cE[D]^h$ diverge under the
  action of $a_t$ in the sense that for every compact set $K$ of $\cE[D]^h$, the images $a_t(K)$ eventually leaves every compact set of
  $\cE[D]$.

  Let $M$ be the space of measures on $\cE[D]$ with the weak${}^*$ topology, and let $S\subset M$ be the compact, $a_t$-invariant set of
  subsequential limits of $m_x^t$.  By Corollary~5.3 of \cite{eskinmasur}, every measure in $S$ has unit mass.

  It is well-known that any subsequential limit of $m_x^t$ is $h_t$-invariant (see \cite[Lemma~7.3]{emm}).  By the main result of
  \cite{caltawortman}, every finite $h_t$-invariant measure on $\cE[D]$ is either supported on $\cE[D]^h$, is uniform measure on $\cD$, or
  is the uniform measure $\nu_D$ on $\cE[D]$.  But no measure $\mu\in S$ can be supported on $\cE[D]^h$  because by divergence we would have
  $(a_t)_*\mu\to 0$, but then $0\in S$, which contradicts every measure in $S$ having unit mass.
\end{proof}

We now show that $\cD$ is not a subsequential limit, following the proof of a similar statement in \cite{emm}.

Let $\A'$ be the kernel foliation of $\cE[D]$, that is, the foliation of $\cE[D]$ given by pulling back the foliation $\Omega_1\A$ of
$\Omega_1\E(1,1)$.  The leafwise quadratic differential $q$ on $\Omega_1\A$ lifts to a leafwise Abelian differential on $\A'$, so leaves of
$\A'$ have natural translation structures.  Given $x\in\cE[D]$ and $v\in\cx$, we write $x+v$ for the point in $\cE[D]$ obtained by
translating $x$ by $v$ in its leaf of $\A'$.  This is not always well defined because the segment through $x$ in the direction of $v$ might
hit a cone point of the translation structure.  The map $x \mapsto x+v$ is defined for almost every $v$ and given $x$, for sufficiently
small $v$.  We have the fundamental identity,
\begin{equation*}
  A\cdot(x+v) = A\cdot x + Av
\end{equation*}
for every $A\in \SLtwoR$.

Now let $\mathcal{D}'$ be the Riemann surface orbifold of which $\cD$ is the unit tangent bundle and let $\pi\colon\cD\to \mathcal{D}'$ be
the natural projection.  Let $K\subset K'\subset \cD$ be the inverse images under $\pi$ of closed hyperbolic discs in $\mathcal{D}'$ with
$K\subset \interior K'$.  

Given $r, s>0$, define
\begin{equation*}
  \BOX(r, s) = \{(x, y)\in\reals^2 : |x|<r\text{ and } |y|<s\}.
\end{equation*}
Since $K'$ is compact, the natural map $K'\times \BOX(\rho,\rho)\to \cE[D]$ defined by $(x, v)\mapsto x+v$
is everywhere-defined and injective for some $\rho>0$.  Given any $S\subset K'$ and $r<\rho$, let $B_r(S)$ be the image of this map, which we identify with
the box $S\times\BOX(r, r)$.  Define
\begin{equation*}
  R(S, r, t) = \{\theta \in \reals/\zed: a_t r_\theta x \in B_r(S)\}. 
\end{equation*}

\begin{lemma}
  \label{lem:nestedintervals}
  For any $\delta>0$, we have
  \begin{equation*}
    |R(K, \delta\rho, t)|<2\delta|R(K', \rho, t)|
  \end{equation*}
  whenever $t$ is sufficiently large.
\end{lemma}

\begin{remark}
  We use the notation $|S|$ for the Lebesgue measure of $S$.
\end{remark}

\begin{proof}
  Let $I$ be a component of $R(K', \rho, t)$.  It suffice to show that
  \begin{equation*}
    |I\cap R(K, \delta\rho, t)| < 2 \delta |I|
  \end{equation*}
  for $t$ sufficiently
  large.  Suppose this intersection is nontrivial and that $\theta_0\in I\cap R(K, \delta\rho, t)$.  By the choice of $\rho$, there is a unique
  $v\in\BOX(\delta\rho, \delta\rho)$ such that
  \begin{equation}
    \label{eq:67}
    y := a_t r_{\theta_0} x + v \in \cD.
  \end{equation}
  We define
  \begin{align*}
    S &= \{ \theta\in (\theta_0 -\pi/2, \theta_0+\pi/2): a_t r_\theta r_{\theta_0}^{-1} a_t^{-1} v \in \BOX(\delta\rho, \delta\rho)\}\\
    &= \{\theta\in (\theta_0 -\pi/2, \theta_0+\pi/2) : r_{\theta-\theta_0} a_t^{-1} v \in \BOX(e^{-t}\delta\rho, e^t \delta\rho)\}, \quad\text{and}\\
    S' &= \{ \theta\in (\theta_0 -\pi/2, \theta_0+\pi/2) : a_t r_\theta r_{\theta_0}^{-1} a_t^{-1} v \in \BOX(\rho, \rho)\}.
  \end{align*}

  Let $\kappa$ be the distance from $x$ to $\cD$ along $\A'$.  We have $\kappa>0$ since $x\not\in\cD$.  Since $a_t^{-1}v +
  r_{\theta_0}x\in\cD$, we have
  \begin{equation}
    \label{eq:69}
    |a_t^{-1} v| \geq \kappa.
  \end{equation}
  Thus, if $t$ is sufficiently large, we have
  \begin{equation}
    \label{eq:64}
    |a_t^{-1}v| > e^{-t}\rho.
  \end{equation}
  It then follows from \eqref{eq:64} and elementary trigonometry that
  \begin{align*}
    \sin\frac{|S|}{2} &\leq \frac{e^{-t}\delta\rho}{|a_t^{-1}v|}, \quad\text{and}\\
    \sin\frac{|S'|}{2} &= \frac{e^{-t}\rho}{|a_t^{-1}v|},
  \end{align*}
  so
  \begin{gather}
    \notag
    |S| \leq \frac{4 e^{-t}\delta\rho}{|a_t^{-1}v|}, \quad\text{and}\\
    \label{eq:68}
    \frac{2e^{-t}\rho}{|a_t^{-1}v|} \leq |S'| \leq  \frac{4 e^{-t}\rho}{|a_t^{-1}v|}. 
  \end{gather}
  Therefore,
  \begin{equation}
    \label{eq:66}
    \frac{|S|}{|S'|}\leq 2\delta.
  \end{equation}

  We have $I\cap R(K, \delta\rho, t)\subset S$, so by \eqref{eq:66}, it is enough to show that $I=S'$.  To show this, it is enough to show
  that for either endpoint $\psi$ of $S'$, we have
  \begin{equation}
    \label{eq:70}
    a_t r_\psi x \in \interior K' \times \bdry\BOX(\rho, \rho),
  \end{equation}
  identifying $B_\rho(K')$ with $K'\times\BOX(\rho, \rho)$.

  From \eqref{eq:67}, we have
  \begin{equation*}
    a_tr_\psi x + a_t r_\psi r_{\theta_0}^{-1} a_t^{-1} v = a_t r_\psi r_{\theta_0}^{-1}a_t^{-1}y.
  \end{equation*}
  Let $y'= a_t r_\psi r_{\theta_0}^{-1}a_t^{-1}y$.  The points $\pi(y)$ and $\pi(y')$ in $\mathcal{D}'$ are joined by an arc of angle
  $|\psi-\theta_0|$ in a hyperbolic circle of radius $t$.  Therefore, the distance between $\pi(y)$ and $\pi(y')$ in $\mathcal{D}'$ is at most
  \begin{equation*}
    |\psi-\theta_0| 2\pi\frac{e^t-1}{e^{t/2}} < 8\pi e^{-t/2}\frac{\rho}{\kappa},
  \end{equation*}
  using \eqref{eq:69} and \eqref{eq:68}.  Since $\pi(K)\subset\interior \pi(K')$, this is smaller than the distance between $K$ and $\bdry
  K'$ in $\mathcal{D}'$ if $t$ is sufficiently large.  This proves \eqref{eq:70}.
\end{proof}

\paragraph{Proof of Theorem~\ref{thm:uniformdistribution2}.}

Let $m_x^\infty = \lim m_x^{t_n}$ for some sequence $t_n$.  Suppose $m_x^\infty$ assigns positive mass to $\cD$.  Then we can choose compact
subsets $K\subset K'\subset\cD$ as above so that $m_x^\infty(K)>0$.
For any $\delta>0$, we have
\begin{align*}
  m_x^\infty(K) &= \lim_{n\to\infty} m_x^{t_n} (K)\\
  &\leq\lim_{n\to\infty} m_t^{t_n} B_{\delta\rho}(K)\\
  &= \lim_{n\to\infty} |R(K, \delta\rho, t_n)|\\
  &\leq \lim_{n\to\infty} 2\delta|R(K', \rho, t_n)|\\
  &\leq 2\delta,
\end{align*}
using Lemma~\ref{lem:nestedintervals}.  This contradicts $\mu_x^\infty(K)>0$.  Therefore, $m_x^\infty(\cD)=0$, which implies
Theorem~\ref{thm:uniformdistribution2} by Lemma~\ref{lem:sufficestoshow}.
\qed

\section{Applications to billiards}
\label{sec:appl-bill}

In this section, we summarize results of Eskin, Masur, and Veech which allow us to obtain the asymptotics for counting saddle connections and
closed geodesics claimed in Theorems~\ref{thm:countingsaddleconnections} and \ref{thm:EDcountings}.

\paragraph{Counting problems.}

Consider the following situation.  Let $S\subset\Omega_1\moduli[g]$ be a stratum equipped with a finite, ergodic, $\SLtwoR$-invariant
measure $\mu$.  Let $M(\reals^2)$ be the space of measures on $\reals^2$.  Consider a function $V\colon S\to
M(\reals^2)$ which is $\SLtwoR$-equivariant and satisfies
\begin{equation*}
  N_V(T, R) < C N_{\rm s}(T, R)
\end{equation*}
for some constant $C$ and any $T\subset S$, where $N_V(T, R)$ is the $V(T)$-measure of the ball $B_R(0)\subset \reals^2$, and $N_{\rm s}(T,
R)$ is the number of saddle connections of length at most $R$ on $T$.  We call the triple $(S, \mu, V)$ a \emph{counting problem}.

\paragraph{Siegel-Veech transform.}

Given any measurable nonnegative $f\colon\reals^2\to\reals$, the \emph{Siegel-Veech transform} of $f$ is the function
$\hat{f}\colon S\to\reals$ defined by
\begin{equation*}
  \hat{f}(T) = \int_{\reals^2}f \, dV(T).
\end{equation*}
Note that if $\chi_R$ is the characteristic function of $B_R(0)$, then $\hat{\chi}_R(T) = N_V(T, R)$. 
\begin{theorem}[\cite{veech98}]
  \label{thm:svtransform}
  There is a constant $c_{\mu, V}$ such that for any nonnegative, measurable $f$, we have
  \begin{equation*}
    \frac{1}{\Vol(\mu)}\int_S \hat{f}\,d\mu = c_{\mu, V}\int_{\reals^2} f(x, y)\,dx\,dy.
  \end{equation*}
\end{theorem}
The constant $c_{\mu, V}$ is called a \emph{Siegel-Veech constant}.

\paragraph{Pointwise asymptotics.}

Given $T\in S$, let $m_T$ be the uniform measure on $\SOtwoR\cdot T$.  Eskin and Masur showed in \cite{eskinmasur}:
\begin{theorem}
  \label{thm:pointwiseasymptotics}
  Consider a counting problem $(S, \mu, V)$.  Suppose for some $T\in S$, we have
  \begin{equation*}
    \lim_{t\to\infty}(a_t)_*m_T = \frac{\mu}{\Vol(\mu)}.
  \end{equation*}
  We then have
  \begin{equation}
    \label{eq:71}
    N_V(T, R)\sim \pi c_{\mu, V}R^2.
  \end{equation}
\end{theorem}

\begin{remark}
  For $T$ with nonunit area, the right hand side of \eqref{eq:71} is scaled by a factor of $1/\Area(T)$.
\end{remark}

\paragraph{Counting problems in genus two.}

Now restrict to the stratum $\Omega_1\moduli(1,1)$, equipped with the period measure $\mu_D^1$ supported on $\Omega_1\E(1,1)$.
Given a surface $T\in \Omega_1\moduli(1,1)$, let $\mathfrak{C}(T)$ be the set of maximal cylinders on $T$, and let $\mathfrak{S}_i(T)$ be the set of saddle
connections joining distinct zeros of multiplicity $i$  on $T$ (we count pairs of saddle connections related by the hyperelliptic involution
as a single saddle connection).  A cylinder or saddle connection $I\in \mathfrak{C}(T)$ or $\mathfrak{S}_i(T)$ determines a
holonomy vector $v(I)\in\cx$, well-defined up to sign.  Define
\begin{align*}
  C(T) &= \frac{1}{2}\sum_{I\in\mathfrak{C}(T)}(\delta_{v(I)} + \delta_{-v(I)}) \quad\text{and}\\
  S_i(T) &= \frac{1}{2}\sum_{I\in\mathfrak{S}_i(T)}(\delta_{v(I)} + \delta_{-v(I)}).
\end{align*}

\paragraph{Proof of Theorem~\ref{thm:EDcountings}.}

Consider the counting problems $(\Omega_1\moduli(1,1), \mu_D^1, C)$ and $(\Omega_1\moduli(1,1), \mu_D^1, S_i)$ and let $c_D$ and $s_D^i$ be the
associated Siegel-Veech constants.  We calculate these constants by applying Theorem~\ref{thm:svtransform} with $f$ the characteristic
function $\chi_R$.  By Theorems~\ref{thm:volED} and \ref{thm:svintegral}, we have
\begin{align*}
  c_D &= \frac{15}{\pi^2},\\
  s_D^1 &= \frac{27}{8}, \quad\text{and}\\
  s_D^2 &= \frac{5}{8}.
\end{align*}
We then obtain the desired asymptotics from Theorems~\ref{thm:pointwiseasymptotics} and \ref{thm:EDcountings}.
\qed

\begin{figure}[htbp]
  \centering
  \includegraphics{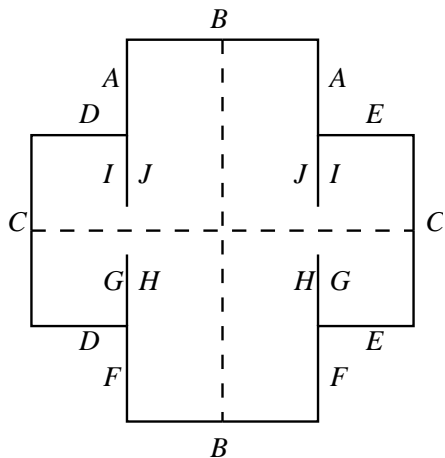}
  \caption{Unfolding of $P$}
  \label{fig:unfolding}
\end{figure}

\paragraph{Unfolding.}

There is associated to every rational-angled polygon $P$ a compact Riemann surface with a holomorphic one-form, called the \emph{unfolding}
$U(P)$ of $P$, obtained by gluing together several reflected copies of $P$.

If $P$ is the L-shaped polygon with barrier
$P(a, b, t)$ shown in Figure~\ref{fig:Lshaped}, then the unfolding $U(P)$ is the surface obtained by gluing four reflected copies of $P$ as shown in
Figure~\ref{fig:unfolding}.  It is easy to see that $U(P)$ is genus two with two simple zeros.  There is a natural refolding map $\pi\colon U(P)\to P$.  The map
$\pi$ sends each closed geodesic of length $L$ to a closed billiards path of length $L$.  Conversely, each closed billiards path on $P$ 
which is neither horizontal nor vertical is the image of exactly two closed geodesics on $U(P)$ of the same length.  The analogous statement
holds for saddle connections.

A saddle connection $I$ on $U(P)$ has multiplicity one if and only if $I$ passes through a Weierstrass point of $U(P)$.  The six Weierstrass
points of $U(P)$ map to the six corners of $P$.  Therefore, a saddle connection on $U(P)$ is multiplicity one if and only if the
corresponding saddle connection of $P$ is multiplicity one.  We therefore have
\begin{equation}
  \label{eq:72}
  N_{\rm c}(U(P), R)\sim 2 N_c(P, R),
\end{equation}
and similarly for the counting functions for type one and two saddle connections.

It is straightforward to check that $U(P(x, y, t))$ lies on the decagon curve if and only if
\begin{equation*}
  x = y = \frac{1+\sqrt{5}}{2} \quad \text{and}\quad t=\frac{5-\sqrt{5}}{10}.
\end{equation*}
We then obtain Theorem~\ref{thm:countingsaddleconnections} directly from Theorem~\ref{thm:EDcountings} and \eqref{eq:72}.

\bibliography{my}

\end{document}

%% file: Lshaped.pstex_t
\begin{picture}(0,0)%
\includegraphics{Lshaped.pstex}%
\end{picture}%
\setlength{\unitlength}{3947sp}%
\begingroup\makeatletter\ifx\SetFigFont\undefined%
\gdef\SetFigFont#1#2#3#4#5{%
  \reset@font\fontsize{#1}{#2pt}%
  \fontfamily{#3}\fontseries{#4}\fontshape{#5}%
  \selectfont}%
\fi\endgroup%
\begin{picture}(2117,2284)(3451,-4835)
\put(4351,-4786){\makebox(0,0)[lb]{\smash{\SetFigFont{12}{14.4}{\rmdefault}{\mddefault}{\updefault}{\color[rgb]{0,0,0}$a$}%
}}}
\put(4651,-3736){\makebox(0,0)[lb]{\smash{\SetFigFont{12}{14.4}{\rmdefault}{\mddefault}{\updefault}{\color[rgb]{0,0,0}$t$}%
}}}
\put(3451,-3736){\makebox(0,0)[lb]{\smash{\SetFigFont{12}{14.4}{\rmdefault}{\mddefault}{\updefault}{\color[rgb]{0,0,0}$b$}%
}}}
\end{picture}

%% file: Zshaped.pstex_t
\begin{picture}(0,0)%
\includegraphics{Zshaped.pstex}%
\end{picture}%
\setlength{\unitlength}{3947sp}%
\begingroup\makeatletter\ifx\SetFigFont\undefined%
\gdef\SetFigFont#1#2#3#4#5{%
  \reset@font\fontsize{#1}{#2pt}%
  \fontfamily{#3}\fontseries{#4}\fontshape{#5}%
  \selectfont}%
\fi\endgroup%
\begin{picture}(4464,3145)(544,-3223)
\put(2472,-1888){\makebox(0,0)[lb]{\smash{{\SetFigFont{12}{14.4}{\rmdefault}{\mddefault}{\updefault}{\color[rgb]{0,0,0}$\alpha_2$}%
}}}}
\put(3838,-853){\makebox(0,0)[lb]{\smash{{\SetFigFont{12}{14.4}{\rmdefault}{\mddefault}{\updefault}{\color[rgb]{0,0,0}$\alpha_3$}%
}}}}
\put(1690,-2566){\makebox(0,0)[lb]{\smash{{\SetFigFont{12}{14.4}{\rmdefault}{\mddefault}{\updefault}{\color[rgb]{0,0,0}$P_1$}%
}}}}
\put(2556,-1618){\makebox(0,0)[lb]{\smash{{\SetFigFont{12}{14.4}{\rmdefault}{\mddefault}{\updefault}{\color[rgb]{0,0,0}$P_2$}%
}}}}
\put(3052,-654){\makebox(0,0)[lb]{\smash{{\SetFigFont{12}{14.4}{\rmdefault}{\mddefault}{\updefault}{\color[rgb]{0,0,0}$\gamma_3$}%
}}}}
\put(1088,-2625){\makebox(0,0)[lb]{\smash{{\SetFigFont{12}{14.4}{\rmdefault}{\mddefault}{\updefault}{\color[rgb]{0,0,0}$\gamma_1$}%
}}}}
\put(3722,-611){\makebox(0,0)[lb]{\smash{{\SetFigFont{12}{14.4}{\rmdefault}{\mddefault}{\updefault}{\color[rgb]{0,0,0}$P_3$}%
}}}}
\put(1407,-1616){\makebox(0,0)[lb]{\smash{{\SetFigFont{12}{14.4}{\rmdefault}{\mddefault}{\updefault}{\color[rgb]{0,0,0}$\gamma_2$}%
}}}}
\put(1560,-2829){\makebox(0,0)[lb]{\smash{{\SetFigFont{12}{14.4}{\rmdefault}{\mddefault}{\updefault}{\color[rgb]{0,0,0}$\alpha_1$}%
}}}}
\end{picture}%

%% file: clambda.pstex_t
\begin{picture}(0,0)%
\includegraphics{clambda.pstex}%
\end{picture}%
\setlength{\unitlength}{3947sp}%
\begingroup\makeatletter\ifx\SetFigFont\undefined%
\gdef\SetFigFont#1#2#3#4#5{%
  \reset@font\fontsize{#1}{#2pt}%
  \fontfamily{#3}\fontseries{#4}\fontshape{#5}%
  \selectfont}%
\fi\endgroup%
\begin{picture}(4244,1454)(4554,-4884)
\put(5551,-3586){\makebox(0,0)[lb]{\smash{{\SetFigFont{12}{14.4}{\rmdefault}{\mddefault}{\updefault}{\color[rgb]{0,0,0}$U_\lambda^-$}%
}}}}
\put(7351,-3586){\makebox(0,0)[lb]{\smash{{\SetFigFont{12}{14.4}{\rmdefault}{\mddefault}{\updefault}{\color[rgb]{0,0,0}$U_\lambda^+$}%
}}}}
\end{picture}%